\documentclass[12pt]{amsart}
 \usepackage[margin=30mm]{geometry}
\usepackage{amsmath, amsfonts, amssymb, amsthm, mathrsfs, enumerate}
\usepackage[utf8,utf8x]{inputenc}
\usepackage[pagebackref=true]{hyperref}
\hypersetup{%
  pdftitle = {Flag-accurate arrangements},
  pdfkeywords = {},
  pdfsubject = {},
  pdfauthor = {},
  colorlinks = true,
  linktoc = page,
  citecolor = magenta,
  citecolor = purple,
  linkcolor = blue 
}

\usepackage{xcolor}
\usepackage{subcaption}
\usepackage{tikz-cd}
\usepackage{tikz}
\usetikzlibrary{arrows}
\usetikzlibrary{shapes}
\usetikzlibrary{decorations}
\usetikzlibrary{arrows,decorations.markings}
\usepackage{multirow}
\tikzstyle{v} = [circle, draw, inner sep=2pt, minimum size=3pt, fill=black]
\tikzstyle{l} = [rectangle, draw, rounded corners]

\usepackage{caption}
\usetikzlibrary{calc}
\usepackage{enumerate}
\usepackage{verbatim}

\usepackage{mfirstuc}

\usepackage{todonotes}

\newcommand{\CC }{\mathbb{C}}
\newcommand{\RR }{\mathbb{R}}

\newcommand{\ZZ }{\mathbb{Z}}
\newcommand{\KK }{\mathbb{K}}

\newcommand{\Ac }{\mathscr{A}}
\newcommand{\Aa }{\mathscr{A}}
\newcommand{\Ec }{\mathscr{E}}
\newcommand{\Bc }{\mathscr{B}}
\newcommand{\Bb }{\mathscr{B}}
\newcommand{\Cc }{\mathscr{C}}
\newcommand{\Dc }{\mathscr{D}}

\newcommand{\Hc }{\mathscr{H}}
\newcommand{\Kc }{\mathcal{C}}

\newcommand{\Wc }{\mathcal{W}}
\newcommand{\Csc }{\mathscr{C}}
\newcommand{\Ccal }{\mathscr{C}}
\newcommand{\Dcal }{\mathscr{D}}
\newcommand{\Rc }{\mathscr{R}}
\newcommand{\Ss }{\mathscr{S}}

\newcommand{\Gc }{\mathcal{G}}

\newcommand{\Fc }{\mathscr{F}}

\newcommand{\Qc }{\mathcal{Q}}
\newcommand{\Ic}{\mathcal{I}}

\newcommand{\cc}{\mathbf{c}} 

\def\DynkinNodeSize{2mm}
\def\DynkinArrowLength{3mm}
\tikzset{
  dnode/.style={
    circle,
    inner sep=0pt,
    minimum size=\DynkinNodeSize,
    fill=white,
    draw},
  middlearrow/.style={
    decoration={markings,
      mark=at position 0.6 with
      {\draw (0:0mm) -- +(+135:\DynkinArrowLength); \draw (0:0mm) -- +(-135:\DynkinArrowLength);},
    },
    postaction={decorate}
  },
  leftrightarrow/.style={
    decoration={markings,
      mark=at position 0.999 with
      {
      \draw (0:0mm) -- +(+135:\DynkinArrowLength); \draw (0:0mm) -- +(-135:\DynkinArrowLength);
      },
      mark=at position 0.001 with
      {
      \draw (0:0mm) -- +(+45:\DynkinArrowLength); \draw (0:0mm) -- +(-45:\DynkinArrowLength);
      },
    },
    postaction={decorate}
  },
  sedge/.style={
  },
  dedge/.style={
    middlearrow,
    double distance=1mm,
  },
  tedge/.style={
    middlearrow,
    double distance=1.0mm+\pgflinewidth,
    postaction={draw}, % third line
  },
  infedge/.style={
    leftrightarrow,
    double distance=0.5mm,
  },
}

\DeclareMathOperator{\Shi}{Shi}
\DeclareMathOperator{\Cat}{Cat}
\DeclareMathOperator{\Cox}{Cox}

\DeclareMathOperator{\Der}{Der}
\DeclareMathOperator{\GL}{GL}

\newcommand\ddxi[1]{\partial/\partial x_{#1}}

\DeclareMathOperator{\rk}{rk}
\DeclareMathOperator{\codim}{codim}

\DeclareMathOperator{\h}{ht}

\makeatletter
\@namedef{subjclassname@2020}{%
  \textup{2020} Mathematics Subject Classification}
\makeatother

\numberwithin{equation}{section}

\theoremstyle{plain}
\newtheorem{lemma}[equation]{Lemma}
\newtheorem{theorem}[equation]{Theorem}

\newtheorem{conjecture}[equation]{Conjecture}

\newtheorem{corollary}[equation]{Corollary}
\newtheorem{proposition}[equation]{Proposition}
\newtheorem{claim}[equation]{Claim}
\theoremstyle{definition}
\newtheorem{definition}[equation]{Definition}
\newtheorem{remark}[equation]{Remark}
\newtheorem{remarks}[equation]{Remarks}
\newtheorem{example}[equation]{Example}
\newtheorem{problem}[equation]{Problem}
\newtheorem{notation}[equation]{Notation}

\title[Flag-accurate arrangements]
{Flag-accurate arrangements}

\author[P.~M\"ucksch]{Paul M\"ucksch}
\address{Paul M\"ucksch,
	Kyushu University,
	Department of Mathematics,
	744, Motooka, Nishi-ku,
	Fukuoka, Japan 819-0395}
\email{paul.muecksch+uni@gmail.com}

\author[G.~R\"ohrle]{Gerhard R\"ohrle}
\address{Gerhard R\"ohrle,
Fakult\"at f\"ur Mathematik,
Ruhr-Universit\"at Bochum,
D-44780 Bochum, Germany}
\email{gerhard.roehrle@rub.de}

\author[T.~N.~Tran]{Tan Nhat Tran}
\address{Tan Nhat Tran,
	Fakult\"at f\"ur Mathematik,
	Ruhr-Universit\"at Bochum,
	D-44780 Bochum, Germany}
\email{tan.tran@ruhr-uni-bochum.de}

\begin{document}

\begin{abstract}
	In \cite{mueckschroehrle:accurate}, the first two authors introduced the notion of an accurate arrangement, a particular notion of freeness. 
   	In this paper, we consider a special subclass, where the property of accuracy 
   	stems from a flag of flats in the intersection lattice of the underlying arrangement. 
   	Members of this family are called \emph{flag-accurate}.
   	One relevance of this new notion is that it entails divisional freeness. 
   There are a number of important natural classes which are flag-accurate, the most prominent one among them is the one consisting of Coxeter arrangements.
   This warrants a systematic study which is put forward in the present paper.
   
   More specifically, let $\Ac$ be a free arrangement of rank $\ell$.
   Suppose that
   for every $1\leq d \leq \ell$, the first $d$ exponents of $\Ac$ -- 
   when listed in increasing order -- 
   are realized as the exponents of a free restriction
   of $\Ac$ to some intersection of reflecting hyperplanes of $\Ac$ of dimension $d$.
   Following \cite{mueckschroehrle:accurate}, we call such an arrangement $\Ac$ with this natural property  \emph{accurate}.
   If in addition the flats involved 
   can be chosen to form a flag, we call $\Ac$ \emph{flag-accurate}. 
   
   We investigate flag-accuracy among reflection arrangements, extended Shi and extended Catalan arrangements, and further for various families of graphic and digraphic arrangements. We pursue these both from theoretical and computational perspectives. Along the way we present examples of accurate arrangements that are not flag-accurate.
   
   The main result of \cite{mueckschroehrle:accurate} shows that MAT-free arrangements are accurate. We provide strong evidence for the conjecture that MAT-freeness actually entails flag-accuracy. 
\end{abstract}

%%%%%%%%%%%%%%%%%%%%%%%%%%%%%%%%%%%%%%%%%%%%%%%%%%%%%%%%%%%%%%%%%%%%%%%%%%%%%%%%%%%%%%%


\maketitle

%%%%%%%%%%%%%%%%%%%%%%%%%%%%%%%%%%%%%%%%%%%%%%%%%%%%%%%%%%%%%%%%%%%%%%%%%%%%%%%%%%%%%%%

\tableofcontents

%%%%%%%%%%%%%%%%%%%%%%%%%%%%%%%%%%%%%%%%%%%%%%%%%%%%%%%%%%%%%%%%%%%%%%%%%%%%%%%%%%%%%%%

\section{Introduction and Statements of Results}

In the study of hyperplane arrangements and their freeness it is important to understand the behavior
of different classes of arrangements with respect to combinatorial or geometric constructions.
Moreover, a central theme is to fix certain numerical properties of (free) arrangements and to
investigate a possible classification of such arrangements in different prominent classes
such as reflection arrangements, their subarrangements and deformations thereof.
Following this philosophy, our present work is  a systematic study of a new notion 
which is shared by many arrangements in those 
well established
classes and to initiate a
detailed study of connections to other known families of free arrangements.

We begin by recalling and by extending the definition of an accurate arrangement from \cite{mueckschroehrle:accurate}.

\begin{definition}
	\label{Def_TF}
	Suppose $\Ac$ is free with exponents $\exp(\Ac) = (e_1, e_2, \ldots, e_\ell)_ \leq$ (this notation, used throughout,   
	simply means that $e_1 \leq e_2 \leq \ldots \leq e_\ell$).
	\begin{enumerate}[(i)]
		\item $\Ac$ is called \emph{almost accurate} provided for each $1 \leq d \leq \ell$ there exists a flat $X_d$ in the intersection lattice $L(\Ac)$ of $\Ac$ of dimension $d$ such that the restriction $\Ac^{X_d}$ of $\Ac$ to $X_d$ is free with $\exp(\Ac^{X_d}) \subseteq \exp(\Ac)$.
	The tuple $(X_1, X_2, \ldots, X_\ell)$ is called a \emph{witness} for the almost accuracy of $\Ac$.

		\item   $\Ac$ is said to be  \emph{accurate} provided 
		there is a witness $(X_1, X_2, \ldots, X_\ell)$ for the almost accuracy of $\Ac$ such that for every $1 \leq d \leq \ell$, we have 
		$\exp(\Ac^{X_d}) = (e_1, e_2, \ldots, e_d)_ \leq$.
		We say that $(X_1, X_2, \ldots, X_\ell)$ is a \emph{witness} for the accuracy of $\Ac$.
	
		\item   For $1 \leq k \leq \ell-1$, $\Ac$ is \emph{$k$-accurate} provided 
		there is a witness $(X_1, X_2, \ldots, X_\ell)$ for the accuracy of $\Ac$ such that $k$ is maximal subject to $(X_1, X_2, \ldots, X_k)$ being a flag in $L(\Ac)$.
		Then $(X_1, X_2, \ldots, X_\ell)$ is a \emph{witness} for the $k$-accuracy of $\Ac$.	
	
		\item    Dually, $\Ac$ is \emph{$k$-coaccurate} provided 
		there is a witness $(X_1, X_2, \ldots, X_\ell)$ for the accuracy of $\Ac$ such that $k$ is minimal subject to $(X_1, X_2, \ldots, X_{\ell-k})$ being a flag in $L(\Ac)$.
		We say that $(X_1, X_2, \ldots, X_\ell)$ is a \emph{witness} for the $k$-coaccuracy of $\Ac$. 
		Thus an $\ell$-arrangement $\Ac$ is $k$-accurate if and only if $\Ac$ is $(\ell-k)$-coaccurate.
	
		\item $\Ac$ is called \emph{flag-accurate} provided
		there is a witness for the $(\ell-1)$-accuracy (equivalently, for the $1$-coaccuracy)
	   	of $\Ac$.
		Such an $\ell$-tuple is a \emph{witness} for the flag-accuracy of $\Ac$.
		
		\item 
		$\Ac$ is called \emph{ind-flag-accurate} if $\Ac$ is both inductively free and flag-accurate, and there is a witness $(X_1, X_2, \ldots, X_\ell)$ for the flag-accuracy of $\Ac$ such that $\Ac^{X_d}$ is inductively free for every $1 \leq d \leq \ell$. 
		In that case $(X_1, X_2, \ldots, X_\ell)$ is said to be a  \emph{witness} for the ind-flag-accuracy of $\Ac$.	
	\end{enumerate}	
\end{definition}

In the sequel we give examples which discriminate between these different notions of accuracy. For instance, the reflection arrangement of the 
exceptional complex reflection group of type 
$G_{31}$ turns out to be flag-accurate, but it is not ind-flag-accurate (as it is not inductively free), see Remark \ref{rem:g31}.
On the other hand, all Coxeter arrangements are
ind-flag-accurate, see Corollary \ref{cor:coxeter-ind-flag-acc}.
Also in Example \ref{ex:star-but-not-divfree} we present a rank $5$ arrangement $\Dc$  which is $2$-accurate, and so $\Dc$  is accurate but not flag-accurate.
Moreover, in Section \ref{sect:graph} we construct a family of graphic arrangements which are 
$(k+1)$-coaccurate but not $k$-coaccurate for any $k$, 
see Theorem \ref{thm:Q4-ext}.

\bigskip

In his seminal work on the connection of freeness to properties of characteristic polynomials \cite{abe:divfree}, Abe introduced the following notion, where $\chi(\Ac, t)$ denotes the characteristic polynomial of $\Ac$, see \eqref{eq.chi}.
 
\begin{definition}[{\cite[Def.~1.5]{abe:divfree}}]
\label{def:DF}
An $\ell$-arrangement $\Ac$ is  \emph{divisionally free}  if there is a flag
$$X_1 \subseteq X_2  \subseteq \cdots  \subseteq X_{\ell-1} \subseteq X_\ell = V$$
such that $\dim(X_i)=i$ for each $1 \le i \le \ell$ and $\chi(\Ac^{X_{i}}, t)$ divides  $\chi(\Ac^{X_{i+1}}, t)$ for  each $1 \le i \le \ell-1$. Such a flag is called a \emph{divisional flag}.
\end{definition}

\begin{remarks}	
	\label{rem_MATRest-flag}
	(i). 
	It follows from the definitions above that if an arrangement is flag-accurate, then it is both accurate and divisionally free, since any witness for the flag-accuracy 
	is simultaneously a divisional flag and a witness for accuracy for free.
	The converse of this implication is false. For, in Example \ref{ex:shif4} we present an accurate and divisionally free arrangement which is not flag-accurate (see also  Corollary \ref{coro:acc-nfa}).

	(ii).
	It is clear that flag accuracy only depends on the intersection lattice of the underlying arrangement
	and is thus a combinatorial property, ditto for ind-flag-accuracy. 
	Likewise, divisional freeness is also  combinatorial \cite[Thm.~4.4(3)]{abe:divfree}. But it is not known whether this is also the case for accuracy itself.

	(iii).
	Non-divisionally free accurate arrangements are not necessarily flag-accurate, see Example \ref{ex:star-but-not-divfree}.
	
	(iv).
	While flag-accuracy implies divisional freeness,  the converse is false, see 
	\cite[Ex.~5.4]{mueckschroehrle:accurate}.
	
	(v).
	A product of arrangements is flag-accurate if and only if each factor is flag-accurate, cf.~\cite[Prop.~2.14, Prop.~4.28]{OrTer92_Arr}; ditto for ind-flag-accuracy, cf.~\cite[Prop.~2.10]{hogeroehrle:inductivelyfree}.
\end{remarks}

We focus on specific classes of arrangements in this paper and investigate (ind-)flag-accuracy among them.
The beginning of our investigation concentrates on the question of ind-flag-accuracy of 
complex reflection arrangements and of their restrictions.
For instance, in Theorem \ref{thm:complex-flag} we show that the notions of flag-accuracy and ind-flag-accuracy coincide for complex reflection arrangements.
Here is a compelling consequence of this result.

\begin{theorem}
	\label{thm:complex-ind-flag}
	Let $G$ be an irreducible complex reflection group with reflection arrangement $\Ac = \Ac(G)$. 
	Suppose $G \ne G_{31}$. Then 
	the following are equivalent:
	\begin{itemize}
		\item [(i)] $\Ac$ is accurate;
		\item [(ii)] $\Ac$ is flag-accurate;
		\item [(iii)] $\Ac$ is ind-flag-accurate;
		\item [(iv)] $\Ac$ is divisionally free;
		\item [(v)] $\Ac$ is inductively free.
	\end{itemize}
\end{theorem}

\begin{remark}
	\label{rem:g31}
	The reason why we need to exclude the exceptional complex reflection group of type 
	$G_{31}$ in Theorem \ref{thm:complex-ind-flag} is due to the fact that
	$\Ac(G_{31})$ itself is not inductively free, \cite[Thm.~1.1]{hogeroehrle:inductivelyfree}, but $\Ac(G_{31})$ does satisfy the properties in parts (i), (ii), and (iv) of the theorem, cf.~\cite[Thm.~5.8]{mueckschroehrle:accurate}, Theorem \ref{thm:complex-flag}, and \cite{abe:divfree}, respectively.
	In particular, $\Ac(G_{31})$ is flag-accurate, but not ind-flag-accurate.
\end{remark}

The same equivalence 
as the one in Theorem \ref{thm:complex-flag}
prevails for 
restrictions of complex reflection arrangements. The following is a consequence of 
Theorem \ref{thm:restr-flag}.

\begin{theorem}
	\label{thm:restr-ind-flag}
	Let $G \subseteq \GL(\CC^n)$
	 be an irreducible complex reflection group with reflection arrangement $\Ac(G)$. Let $\Ac = \Ac(G)^Y$ be the restriction of $\Ac(G)$ to some flat $Y \in L(\Ac(G))\setminus \{\CC^n\}$. Then 
	the following are equivalent:
	\begin{itemize}
		\item [(i)] $\Ac$ is accurate;
		\item [(ii)] $\Ac$ is flag-accurate;
		\item [(iii)] $\Ac$ is ind-flag-accurate.
	\end{itemize}
\end{theorem}

For an explicit list of all instances 
that satisfy the equivalent statements above, see
Theorem \ref{thm:restr-flag}.

We emphasize that the equivalences in  
Theorems \ref{thm:complex-ind-flag} 
and \ref{thm:restr-ind-flag} 
do hold in particular for
Coxeter arrangements; we do record this as a separate result and do note in passing that all Coxeter arrangements are hereditarily inductively free, cf.~\cite{BC12} and \cite[\S 3.2.2]{hogeroehrle:inductivelyfree}.

\begin{corollary}
	\label{cor:coxeter-ind-flag-acc}
	Coxeter arrangements are ind-flag-accurate.
\end{corollary}

In \cite{ABCHT16_FreeIdealWeyl}, Abe, Barakat, Cuntz, Hoge and Terao proved the
so-called  \emph{Multiple Addition Theorem} (MAT) (Theorem \ref{Thm_MAT})
which is a variation
of the addition part of Terao's seminal Addition-Deletion Theorem \cite{Terao1980_FreeI} (\cite[Thm.~4.51]{OrTer92_Arr}).
Using this theorem, they went on
to uniformly derive the freeness of \emph{ideal subarrangements} of Weyl arrangements (Definition \ref{DEF_ideal} and Theorem \ref{thm:IdealFree}).
As a special case of this result, they obtained a new uniform proof of the
classical Kostant-Macdonald-Shapiro-Steinberg formula for the exponents of a Weyl group. 

\bigskip

In \cite{CunMue19_MATfree}, Cuntz and the first author introduced the notion of \emph{MAT-freeness} (Definition \ref{Def_MATfree})
to investigate arrangements whose freeness can be derived using an iterative application of the Multiple Addition Theorem (Theorem \ref{Thm_MAT}). 

Next we recall the principal result 
from \cite[Thm.~1.2]{mueckschroehrle:accurate},
which asserts that MAT-freeness
is sufficient for accuracy from Definition \ref{Def_TF}.

\begin{theorem}
	\label{Thm_MATRest}
	MAT-free arrangements are accurate.
\end{theorem}

As ideal subarrangements of Weyl arrangements are MAT-free, by \cite{ABCHT16_FreeIdealWeyl}
(see Theorem \ref{thm:IdealFree}),  
Theorem \ref{Thm_MATRest}
readily yields the following \cite[Thm.~1.3]{mueckschroehrle:accurate}.

\begin{theorem}
	\label{Thm_HtRestIdeal} %
	Ideal arrangements are accurate.
\end{theorem}

It is natural to ask whether ideal arrangements also satisfy the stronger property
of flag-accuracy. Unfortunately, this cannot be derived from Theorem \ref{Thm_HtRestIdeal}
since the class of ideal arrangements is not closed under restrictions.
However, we can give a positive answer to this question for ideal arrangements of rank up to $8$ (Theorem \ref{thm:IdealUpToRk8FA}).
In particular, \emph{all} ideal arrangements stemming from exceptional reflection groups are flag-accurate.
This motivates the following.

\begin{conjecture}
	\label{conj:Ideal-flag} %
	Ideal arrangements are flag-accurate.
\end{conjecture}

In fact much evidence points to an even stronger assertion. For, firstly all ideal arrangements are 
inductively free, thanks to \cite{CRS17_IdealIF}, and secondly a large number of them are even supersolvable thus hereditarily inductively free (e.g.~this applies to \emph{all} ideal arrangements in type $A$, $B$, $C$ and $G_2$). 
(Though in general, ideal arrangements are not hereditarily free, cf.~\cite{amendmoellerroehrle18}.)
So if a supersolvable ideal arrangement is  flag-accurate it is trivially also ind-flag-accurate.
Thus it is tantalizingly tempting to formulate the following.

\begin{conjecture}
	\label{conj:Ideal-ind-flag} %
	Ideal arrangements are ind-flag-accurate.
\end{conjecture}

We have confirmed  Conjecture \ref{conj:Ideal-ind-flag} for all ideal arrangements up to rank $6$, see 
Remark \ref{rem:IdealUpToRk6IFA}.

In view of Theorem \ref{Thm_MATRest}, Conjecture \ref{conj:Ideal-flag} naturally leads to the following more general question, see also Corollary \ref{coro:MATRk4-FA}.

\begin{problem}
	\label{prob:MAT-FA}
	Are all MAT-free arrangements flag-accurate?
\end{problem}

Our evidence for Conjecture \ref{conj:Ideal-ind-flag} does point to an even stronger implication.

\begin{problem}
	\label{prob:MAT-IFA}
	Are all MAT-free arrangements ind-flag-accurate?
\end{problem}

\bigskip
The next class we consider is given by certain deformations of Weyl arrangements -- extended Shi and extended Catalan arrangements.

In \cite[Thm.~1.8]{mueckschroehrle:accurate}
it was shown that extended Shi arrangements, ideal-Shi arrangements and extended Catalan arrangements are accurate.
In Theorems \ref{thm:ExtShi-FA} and \ref{thm:ExtCat-FA} we  partially  strengthen these results as follows.

\begin{theorem}
	\label{Thm_ResShiCatTF-flag}
	Extended Shi arrangements $\Shi^m$ are flag-accurate.	
	Extended Catalan arrangements $\Cat^m$ of Dynkin type $A,B$, or $C$
	are flag-accurate.
\end{theorem}

The second part of this theorem strongly hints towards the following.

\begin{conjecture}
		\label{conj:Cat}
	Extended Catalan arrangements are flag-accurate.
\end{conjecture}

Moreover, in view of \cite[Thm.~1.8]{mueckschroehrle:accurate} we might even pose the following.
\begin{problem}
	\label{prob:ideal-Shi-FA}
	Are all ideal-Shi arrangements flag-accurate?
\end{problem}

\bigskip
In Section \ref{sect:graph}, we consider free graphic arrangements under the aspect of 
flag-accuracy.
We exhibit an infinite family of flag-accurate graphic arrangements (Theorem \ref{thm:Qfamily} and Theorem \ref{thm:P-FA}).
In particular, all trivially perfect graphs are ind-flag-accurate.
In Example \ref{ex:FA-Rest-Loc}, we present a non-flag-accurate free graphic arrangement which admits an extension by one vertex which 
gives a flag-accurate graphic arrangement. 
Thus, (ind-)flag-accuracy, similarly to accuracy, is neither compatible with  restriction, nor  with localization.

Finally, in the last section we investigate flag-accuracy for $\psi$-digraphic arrangements. Among them are $N$-Ish arrangements. Here we also discuss 
classes of arrangements called \emph{Shi descendants} and \emph{Catalan descendants},  
which are closely related to 
extended Shi and extended Catalan arrangements by means of graph theoretic operations which we call \emph{mutation}. 
The core results here are as follows. 
Theorem \ref{thm:A-H} demonstrates that 
the origins in the Shi descendant sequences have ind-flag-accurate cones. This in particular implies that 
the extended Shi arrangements have this property, 
see Remark \ref{rem:Shi-A-alter}.
In Theorem \ref{thm:Shi-ally-FA} we then show 
that all Shi descendants admit ind-flag-accurate cones.
There are analogous results for Catalan descendants. For instance, 
Corollary \ref{cor:CatA-FA} proves that the cones of extended Catalan arrangements of type $A$ are also  
ind-flag-accurate. Finally, in Theorem \ref{thm:Cat-ally-FA} we show that indeed all 
Catalan descendants are ind-flag-accurate.
\bigskip

For general information about arrangements  
we refer the reader to \cite{OrTer92_Arr}.

%%%%%%%%%%%%%%%%%%%%%%%%%%%%%%%%%%%%%%%%%%%%%%%%%%%%%%%%%%%%%%%%%%%%%%%%%%%%%%%%%%%%%%%
%%%%%%%%%%%%%%%%%%%%%%%%%%%%%%%%%%%%%%%%%%%%%%%%%%%%%%%%%%%%%%%%%%%%%%%%%%%%%%%%%%%%%%%

\section{Preliminaries}
\label{ssect:recoll}

In this section we review some basic concepts and preliminary results on various classes of free arrangements. Our standard reference is \cite{OrTer92_Arr}.

\subsection{Hyperplane arrangements}
\label{ssect:hyper}
Let $ \mathbb{K} $ be a field and let $V = \KK^\ell$ be an $ \ell $-dimensional vector space over $ \mathbb{K} $. 
A \textit{hyperplane} in $V$ is an affine subspace of codimension $1$ of $V$.
An \emph{arrangement} $\Ac = (\Ac, V)$  is a finite collection of hyperplanes in  $ V$. 
We say that $ \Ac $ is \emph{central} if every hyperplane in $ \Ac $ passes through the origin. 
If we want to emphasize the dimension $\ell$ of the ambient vector space
we say that $\Ac$ is an $\ell$-arrangement. 

Let  $ \Ac $ be an arrangement.
Define the \emph{intersection poset} $ L(\Ac) $ of $ \Ac $ by 
\begin{align*}
L(\Ac) :=\left\{\bigcap_{H \in \Bc }H \neq \emptyset \mid  \Bc \subseteq \Ac \right\},
\end{align*}
where the partial order is given by reverse inclusion $X\le Y\Leftrightarrow Y\subseteq X$ for $X, Y \in L(\Ac)$. 
We agree that $V $ is a unique minimal element in $ L(\Ac) $ as the intersection over the empty set. 
Thus $ L(\Ac) $ is a semi-lattice which can be equipped with the rank function $ \rk(X) :=\operatorname{codim}(X) $ for $X \in L(\Ac)$. 
We also define the \emph{rank} $\rk(\Ac)$ of $\Ac$ as the rank of a maximal element of $L(\Ac)$. 
The intersection poset $ L(\Ac) $ is sometimes referred to as the  \emph{combinatorics} of  $ \Ac $.

All of the properties considered in the present paper, such as supersolvability, (inductive, divisional, recursive, MAT-)freeness, ((almost, ind-)flag-)accuracy of arrangements are compatible with products.

An $\ell$-arrangement $ \Ac $ is called \emph{essential} if $ \rk(\Ac) = \ell $. 
Any arrangement $ \Ac $ of rank $r$ in $\KK^\ell$ can be written as the product of an essential arrangement $ \Ac^{\mathrm{ess}} $ and the $(\ell-r)$-dimensional empty arrangement $\varnothing_{\ell-r}$. 
We call $ \Ac^{\mathrm{ess}} $ the \emph{essentialization} of $ \Ac $. 
Thus $ \Ac $ satisfies any of these properties  if and only if  $ \Ac^{\mathrm{ess}} $ does.

\bigskip

The \emph{characteristic polynomial} $ \chi(\Ac, t) \in \mathbb{Z}[t] $ of $ \Ac $ is defined by
\begin{equation}
	\label{eq.chi}
\chi(\Ac, t) :=\sum_{X \in L(\Ac)}\mu(X)t^{\dim(X)}, 
\end{equation}
where $ \mu $ denotes the \emph{M\"{o}bius function} $ \mu \colon L(\Ac) \to \mathbb{Z} $ defined recursively by 
\begin{align*}
\mu\left(V \right) :=1 
\quad \text{ and } \quad 
\mu(X) :=-\sum_{\substack{Y \in L(\Ac) \\ X \subsetneq Y}}\mu(Y). 
\end{align*}

Let $S = S(V^*)$ be the symmetric algebra of the dual space $V^*$ of $V$.
We fix a basis $x_1,\ldots,x_\ell$ for $V^*$ and identify $S$ with the polynomial ring $\KK[x_1,\ldots,x_\ell]$.
The algebra $S$ is equipped with the grading by polynomial degree: $S = \bigoplus_{p\in \ZZ} S_p$,
where $S_p$ is the $\KK$-space of homogeneous polynomials of degree $p$ (along with $0$), where $S_p = \{0\}$ for $p < 0$.

The \emph{defining polynomial} $Q(\Ac)$ of $\Ac$ is given by
$$Q(\Ac):=\prod_{H \in \Ac} \alpha_H \in S,$$
where $ \alpha_H=a_1x_1+\cdots+a_\ell x_\ell+d$  $(a_i, d \in \mathbb{K})$ satisfies $H = \ker(\alpha_H)$. 

The operation of \emph{coning} is a standard way to pass from  an arbitrary arrangement to a central one. 
The \emph{cone} $ \cc\Ac$ over $\Ac$ is the central arrangement in $\mathbb{K}^{\ell+1}$ with the defining polynomial
$$Q(\cc\Ac):=z Q(\Ac)'
%\prod_{H \in \Ac} {}^h\alpha_H 
\in \mathbb{K}[x_1,\ldots, x_\ell,z],$$
where $Q(\Ac)'$ 
is the homogenization of $Q(\Ac)$, and $z=0$ is the \emph{hyperplane at infinity}, denoted $H_{\infty}$; cf.~\cite[Def.~1.15]{OrTer92_Arr}. 
By abuse of notation, if the cone $ \cc\Ac$ has a property, e.g., supersolvability, freeness, etc, we sometimes say $\Ac$ has that property too.
The characteristic polynomials of $\Ac$ and $\cc\Ac$ are related by the following simple formula (e.g., \cite[Prop.~ 2.51]{OrTer92_Arr}):
$$\chi(\cc\Ac, t)=(t-1)\chi(\Ac,t).$$

\bigskip
Associated with $X \in L(\Ac)$ we have two canonical arrangements, 
the \emph{localization $\Ac_X$} of $\Ac$ at $X$, given by
\[
\Ac_X := \{H \in \Ac  \mid H \supseteq X \},
\]
and the \emph{restriction $\Ac^X$ of $\Ac $ to $X$}, defined by
\[
\Ac^X := \{H \cap X  \ne \emptyset \mid H \in \Ac\setminus \Ac_X  \}.
\]

Two (central) arrangements $\Ac$ and $\Bc$ in $\KK^\ell$ are said to be \emph{(linearly) affinely equivalent} if there is an invertible (linear) affine endomorphism $\varphi: \KK^\ell \to \KK^\ell$ such that $\Bc=\varphi(\Ac)=\{\varphi(H)\mid H\in \Ac\}$. 
In particular, the intersection posets of two affinely equivalent arrangements are isomorphic. 
All of the properties of arrangements considered in the present paper are preserved under affine and linear equivalences. 
In the rest of the paper, we often identify affinely equivalent arrangements and note that for such non-central $\Ac$ and $\Bc$, the cones $\cc\Ac$ and $\cc\Bc$ are linearly equivalent.

\subsection{Free, inductively free, and recursively free arrangements}
\label{SSec_freeArr}

A $\KK$-linear map $\theta:S\to S$ which satisfies $\theta(fg) = \theta(f)g + f\theta(g)$ is called a $\KK$-\emph{derivation}.
Let $\Der(S)$ be the $S$-module of $\KK$-derivations of $S$. 
It is a free $S$-module with basis $\ddxi{1},\ldots,\ddxi{\ell}$.
We say that a non-zero derivation $\theta  = \sum_{i=1}^\ell f_i \ddxi{i}$  is \emph{homogeneous of degree} $p$ provided all coefficients $f_i$ belong to $S_p$, cf.~\cite[Def.~4.2]{OrTer92_Arr}.
In this case we write $\deg{\theta} = p$.
We obtain a $\ZZ$-grading $\Der(S) = \bigoplus_{p \in \ZZ} \Der(S)_p$ of the $S$-module $\Der(S)$.

In the remainder of this section, we assume that $ \Ac $ is a \textbf{central} $\ell$-arrangement.

\begin{definition}
	The \emph{module of $\Ac$-derivations}  of $\Ac$  is defined by
	\begin{equation*}
	D(\Ac) := \{ \theta \in \Der(S) \mid \theta(Q(\Ac)) \in Q(\Ac)S\}.
	\end{equation*}
	In particular, if $\Bc \subseteq \Ac$, then $D(\Ac) \subseteq D(\Bc)$.
	
	We say that $\Ac$ is \emph{free} if the module of $\Ac$-derivations is a free $S$-module. 
\end{definition}

If $\Ac$ is a free arrangement we may choose a homogeneous basis $\theta_1, \ldots, \theta_\ell$ of $D(\Ac)$.
Then the degrees of the $\theta_i$ are called the \emph{exponents} of $\Ac$. 
They are uniquely determined by $\Ac$ \cite[Def.\ 4.25]{OrTer92_Arr}. 
In that case we write 
\[
\exp(\Ac) := (\deg{\theta_1},\ldots,\deg{\theta_\ell})
\]
for the exponents of $\Ac$. 
If $\exp(\Ac) = (e_1, e_2, \ldots, e_\ell)$ with $ e_1 \leq e_2 \leq \ldots \leq e_\ell$ we often write  $\exp(\Ac) = (e_1, e_2,\ldots, e_\ell)_\le$, as in the introduction.
If $\Ac$ has an exponent $e$ appearing  $d\ge0$ times  in $\exp(\Ac)$ we also write $e^d \in \exp(\Ac)$.

Note that the empty arrangement $\varnothing_\ell$ in $V$ is free with
$D(\varnothing_\ell) = \Der(S)$ so that $\exp(\varnothing_\ell)=(0^\ell)\in\ZZ^\ell$.

Fix $H \in \Ac$, denote $\Ac':=\Ac\setminus \{H\}$ and $\Ac'':=\Ac^H$. 
We call $(\Ac, \Ac', \Ac'')$ the triple with respect to the hyperplane $H \in\Ac$. 
 \begin{theorem}[Addition-Deletion Theorem {\cite{Terao1980_FreeI}},  {\cite[Thms.~ 4.46 and 4.51]{OrTer92_Arr}}]
 \label{thm:AD}
Let $\Ac$ be a non-empty arrangement and let $H \in \Ac$. Then two of the following statements imply the third:
\begin{itemize}
\item[(i)] $\Ac$ is free with $\exp(\Ac) = (e_1, \ldots, e_{\ell-1}, e_\ell)$.
\item[(ii)]  $\Ac'$ is free with $\exp(\Ac')=(e_1, \ldots, e_{\ell-1}, e_\ell-1)$. 
\item[(iii)] $\Ac''$ is free with $\exp(\Ac'') = (e_1, \ldots, e_{\ell-1})$.
\end{itemize}
Moreover, all three assertions hold if $\Ac$ and $\Ac'$ are both free.
\end{theorem}

Theorem \ref{thm:AD} above  motivates the following concepts.
\begin{definition}[{\cite[Def.~4.53]{OrTer92_Arr}}]
\label{def:IF}
The class $\mathcal{IF}$ of \emph{inductively free} arrangements is the smallest class of arrangements which satisfies
\begin{itemize}
\item[(i)] the empty arrangement $\varnothing_\ell$ is in $\mathcal{IF}$ for $\ell \ge 0$,
\item [(ii)] if there exists $H \in\Ac$ such that $\Ac'' \in \mathcal{IF}$, $\Ac' \in \mathcal{IF}$, and $\exp(\Ac'') \subseteq \exp(\Ac')$, then $\Ac \in \mathcal{IF}$.
\end{itemize} 
\end{definition}

\begin{definition}[{\cite[Def.~4.60]{OrTer92_Arr}}]
\label{def:IF}
The class $\mathcal{RF}$ of \emph{recursively free} arrangements is the smallest class of arrangements which satisfies
\begin{itemize}
\item[(i)]  $\varnothing_\ell \in \mathcal{RF}$ for $\ell \ge 0$,
\item[(ii)]  if there exists $H \in\Ac$ such that $\Ac'' \in \mathcal{RF}$, $\Ac' \in \mathcal{RF}$, and $\exp(\Ac'') \subseteq \exp(\Ac')$, then $\Ac \in \mathcal{RF}$, 
\item[(iii)]  if there exists $H \in\Ac$ such that $\Ac'' \in \mathcal{RF}$, $\Ac \in \mathcal{RF}$, and $\exp(\Ac'') \subseteq \exp(\Ac)$, then $\Ac' \in \mathcal{RF}$. 
\end{itemize} 
\end{definition}

Thus  $\mathcal{IF} \subseteq  \mathcal{RF}$, and by Theorem \ref{thm:AD} every recursively free arrangement is free.

\subsection{Multiarrangements}
\label{SSec_multi}

A \emph{multiarrangement} is a pair $(\Ac, \mathbf{m})$ where $\Ac = (\Ac, V)$ is an arrangement  and $ \mathbf{m}$ is a map $ \mathbf{m} : \Ac \to \ZZ_{\ge0}$, called a \emph{multiplicity} on $\Ac$. 
Let $\Ac$ be a central arrangement  in $\KK^\ell$ and let $\mathbf{m}$ be a multiplicity on $\Ac$.
The defining polynomial $Q(\Ac, \mathbf{m})$ of the multiarrangement $(\Ac, \mathbf{m})$ is given by 
$$Q(\Ac, \mathbf{m}):= \prod_{H \in \Ac} \alpha^{\mathbf{m}(H)}_H \in S= \mathbb{K}[x_{1}, \dots, x_{\ell}].$$
When $\mathbf{m}(H) = 1$ for every $H \in \Ac$, $(\Ac, \mathbf{m})$ is simply a hyperplane arrangement. 
The \emph{module $D(\Ac, \mathbf{m})$ of logarithmic derivations} of  $(\Ac, \mathbf{m})$ is defined by
$$D(\Ac, \mathbf{m}):=  \{ \theta\in \Der(S) \mid \theta(\alpha_H) \in \alpha^{\mathbf{m}(H)}_HS \mbox{ for all } H \in \Ac\}.$$

We say that $(\Ac, \mathbf{m})$  is \emph{free}  with the multiset $ \exp(\Ac, \mathbf{m}) = (d_{1}, \dots, d_{\ell}) $ of \emph{exponents}  if $D(\Ac, \mathbf{m})$ is a free $S$-module with a homogeneous basis $ \{\theta_{1}, \dots, \theta_{\ell}\}$  such that $ \deg \theta_{i} = d_{i} $ for each $ i $. 
It is known that  $(\Ac, \mathbf{m})$ is always free for $\ell\le 2$ \cite[Cor.~7]{Z89}.

\bigskip
Next we consider the \emph{Ziegler restriction} of a simple arrangement, \cite[Thm.~11]{Z89}.
Let $H \in \Ac$. The \emph{Ziegler restriction} of $\Ac$ onto $H$ is a multiarrangement $(\Ac^H,\mathbf{m}^H)$  defined by $$\mathbf{m}^H(X):=|\Ac_X|-1 \quad \text{ for } \quad X \in \Ac^H.$$ 

We say that $\Ac$ is \emph{locally free in codimension three along $H$} if $\Ac_X$ is free for every $X\in L(\Ac)$
with $X\subseteq H$ and $\codim_V(X)=3$. 
 
The following fundamental theorem gives
a deep connection between the freeness of a simple arrangements and the freeness of a particular Ziegler restriction.
 
\begin{theorem}[{\cite[Thm.~ 2.2]{Yos04_CharaktFree}}, {\cite[Thm.~ 4.1]{AY13}}, {\cite[Thm.~11]{Z89}}]
\label{thm:Yoshinaga's criterion}
Let $\Ac$ be a central arrangement in $\KK^\ell$ with $\ell \ge 3$ and let $H\in \Ac$. Then
$\Ac$ is free  with $\exp(\Ac) = (1, d_2,\ldots, d_\ell)$ if and only if the Ziegler restriction  $(\Ac^H,\mathbf{m}^H)$ is free  with $\exp (\Ac^H,\mathbf{m}^H)=(d_2,\ldots, d_\ell)$ and $\Ac$ is  locally free in codimension three along $H$.
\end{theorem}

\subsection{Supersolvable arrangements}
\label{subsec:SS-arr}
Let $ \Ac $ be a  central arrangement.
An element  $ X \in L(\Ac) $ is said to be \emph{modular} if $ X+Y \in L(\Ac) $ for all $ Y \in L(\Ac) $. 
A modular element of corank $ 1 $ is called a \emph{modular coatom}. 
If $ X \in L(\Ac) $ is a modular coatom, we call the localization $ \Ac_{X}$ a modular coatom of $\Ac$ as well. 

 \begin{definition}
 \label{def:supersolvable}
A central arrangement $\Ac$ of rank $r$ is called \emph{supersolvable} if there exists a chain of arrangements, called an \textrm{M}-chain,
$$\emptyset  =  \Ac_{X_0} \subseteq \Ac_{X_1} \subseteq \dots \subseteq \Ac_{X_{r}} = \Ac,$$
in which $ \Ac_{X_{i}} $ is a modular coatom of $ \Ac_{X_{i+1}} $ for each $0 \le i \le r-1$. 
\end{definition}

The following result is useful to check whether  an element is a modular coatom. 
\begin{proposition}[{\cite[Thm.~ 4.3]{BEZ90}}]
 \label{prop:MC-criterion}
Let   $ X \in L(\Ac) $ be a coatom. 
Then $ \Ac_{X} $ is a modular coatom of $\Ac$ if and only if for any distinct $ H, H^{\prime} \in \Ac\setminus \Ac_{X} $, there exists $ H^{\prime\prime} \in \Ac_{X} $ such that $ H\cap H^{\prime} \subseteq H^{\prime\prime} $. 
\end{proposition}

 Next we recall a classical result on supersolvable and free arrangements. 
 
 \begin{theorem}[{\cite[Thm.~ (4.2)]{JT84}}]
 \label{thm:exp-ss}
If $\Ac$ is supersolvable, then $\Ac$ is inductively free hence free. 
Furthermore, if $\Ac$ has an \textrm{M}-chain $\emptyset  =  \Ac_{X_0} \subseteq \Ac_{X_1} \subseteq \dots \subseteq \Ac_{X_{\rk(\Ac)}} = \Ac,$ then
$\exp(\Ac) = (0^{\ell-\rk(\Ac)}, e_1, \ldots, e_{\rk(\Ac)})$ where $e_i = | \Ac_{X_{i}} \setminus  \Ac_{X_{i-1}}|$.
 \end{theorem}

\subsection{MAT-free arrangements}
\label{SSec_MAT}

We begin by recalling the core result from \cite{ABCHT16_FreeIdealWeyl}, the so-called Multiple Addition Theorem (MAT). 

\begin{theorem}[{\cite[Thm.~3.1]{ABCHT16_FreeIdealWeyl}}]
	\label{Thm_MAT}
	Let $\Ac' = (\Ac', V)$ be a free arrangement with
	$\exp(\Ac')=(e_1,\ldots,e_\ell)_\le$
	and let $1 \le p \le \ell$ be the multiplicity of the highest exponent, i.e.
	\[ e_{\ell-p} < e_{\ell-p+1} =\cdots=e_\ell=:e. \]
	Let $H_1,\ldots,H_q$ be hyperplanes in $V$ with
	$H_i \not \in \Ac'$ for $i=1,\ldots,q$. Define
	\[ \Ac''_j:=(\Ac'\cup \{H_j\})^{H_j}=\{H\cap H_{j} \mid H\in \Ac'\}, \quad \text{ for }j=1,\ldots,q. \]
	Assume that the following conditions are satisfied:
	\begin{itemize}
		\item[(1)]
		$X:=H_1 \cap \cdots \cap H_q$ is $q$-codimensional.
		\item[(2)]
		$X \not \subseteq \bigcup_{H \in \Ac'} H$.
		\item[(3)]
		$|\Ac'|-|\Ac''_j|=e$ for $1 \le j \le q$.
	\end{itemize}
	Then $q \leq p$ and $\Ac:=\Ac' \cup \{H_1,\ldots,H_q\}$ is free with
	$$\exp(\Ac)=(e_1,\ldots,e_{\ell-q},(e+1)^q)_\le.$$
\end{theorem}

We frequently consider the addition of several hyperplanes using
Theorem \ref{Thm_MAT}. This motivates the next terminology.

\begin{definition}
	Let $\Ac'$ and $\{H_1,\ldots,H_q\}$ be as in Theorem \ref{Thm_MAT} such that
	conditions (1)--(3) are satisfied. Then the addition 
	of $\{H_1,\ldots,H_q\}$ to $\Ac'$ resulting in $\Ac = \Ac' \cup \{H_1,\ldots,H_q\}$
	is called an \emph{MAT-step}.
\end{definition}

An iterative application of Theorem \ref{Thm_MAT} motivates the following natural concept.

\begin{definition}%
	[{\cite[Def.~3.2, Lem.~3.8]{CunMue19_MATfree}}]
	\label{Def_MATfree}%
	
	An arrangement $\Ac$ is called \emph{MAT-free} if there exists an ordered partition
	\[
	\pi = (\pi_1|\cdots|\pi_n)
	\] 
	of $\Ac$ such that the following hold. 
	Set $\Ac_0 := \varnothing_\ell$ and
	\[
	\Ac_k := \bigcup_{i=1}^k \pi_i \quad\text{ for } 1 \leq k \leq n.
	\]
	Then for every $0 \leq k \leq n-1$ suppose that
	\begin{itemize}
		\item[(1)] $\rk(\pi_{k+1}) = \vert \pi_{k+1} \vert$,
		\item[(2)] $\cap_{H \in \pi_{k+1}} H \nsubseteq \bigcup_{H' \in \Ac_k}H'$,
		\item[(3)] $\vert \Ac_k \vert - \vert (\Ac_k \cup \{H\})^H \vert = k$ for each $H \in \pi_{k+1}$,
	\end{itemize}
	i.e.\ $\Ac_{k+1} = \Ac_{k}\cup\pi_{k+1}$ is an MAT-step.
	
	An ordered partition $\pi$ with these properties is called an \emph{MAT-partition} for $\Ac$.
\end{definition}

\begin{remark}[{\cite[Rem.~2.15]{mueckschroehrle:accurate}}]
	\label{rem:MAT-free}
	Suppose that $\Ac$ is MAT-free with MAT-partition
	$\pi = (\pi_1|\cdots|\pi_n)$. Then we have:
	\begin{itemize}
		\item [(i)]
		for each $1 \leq k \le n$, $\Ac_k$ is MAT-free with MAT-partition $(\pi_1|\cdots|\pi_{k})$, 
		
		\item[(ii)]
		$\Ac$ is free and the exponents 
		$\exp(\Ac) = (e_1,\ldots,e_\ell)_\le$ of $\Ac$ are given by the block sizes
		of the dual partition of $\pi$: 
		\[ 
		e_i := |\{k \mid |\pi_k|\geq \ell-i+1 \}|, 
		\]
		
		\item [(iii)]
		$|\pi_1| > |\pi_2| \geq \cdots \geq |\pi_n|$.
	\end{itemize}
\end{remark}

The following central result in \cite{mueckschroehrle:accurate} relates arrangements which are constructed by MAT-steps 
from smaller free arrangements with the notion of accuracy.

\begin{theorem}[{\cite[Thm.~3.11]{mueckschroehrle:accurate}}]
	\label{thm:MATRestGen}
	Let $\Ac = \Ac' \dot{\cup} \Bc$ be a free arrangement obtained from the free arrangement $\Ac'$
	through MAT-steps with exponents $\exp(\Ac) = (e_1,\ldots,e_\ell)_\leq$. 
	Suppose that $\pi = (\pi_1 | \cdots | \pi_n )$ is the corresponding
	ordered partition of $\Bc$. 
	Then for each $1 \leq k \leq n$ and each $|\pi_{k+1}| \leq q \leq |\pi_k|$ there is a $\Cc \subseteq \pi_k$ with $|\Cc|=q$
	such that for $X := \cap_{H \in \Cc}H$ 
	the restriction $\Ac^{X}$ is free with exponents 
	\[
	\exp\left(\Ac^{X}\right) = (e_1,\ldots,e_{\ell-q})_\leq.
	\]
\end{theorem}

 \section{Flag-accuracy}
\label{sec:sc-FA}
In this section we collect some sufficient conditions for (ind-)flag-accuracy which are used in subsequent proofs. 
 Let $\KK$ be an arbitrary field and $\Ac$ an arrangement in $V=\KK^\ell$. 
 The following simple criteria, which readily follow from the definitions, are useful to show (ind-)flag-accuracy within an inductive argument. 
\begin{lemma}
	\label{lem:flag-accuracy}
Let $\Ac$ be  (inductively) free with exponents $\exp(\Ac) = (e_1, \ldots, e_\ell)_\le$. Then $\Ac$ is (ind-)flag-accurate if and only if there exist $k$ linearly independent hyperplanes $H_1,\ldots,H_k \in \Ac$ for some $1 \le k \le \ell$ such that $\Ac^{X_i}$ is (inductively) free with $\exp(\Ac^{X_i}) = (e_1, \ldots, e_{\ell-i})_\le$ for each $1 \le i \le k$ where $X_i :=\bigcap_{j=1}^i H_j$  and that $\Ac^{X_k}$ is (ind-)flag-accurate.
	
	In particular, $\Ac$ is (ind-)flag-accurate if and only if there exists an $H$ in $\Ac$ such that $\Ac^H$ is (ind-)flag-accurate with $\exp(\Ac^H) = (e_1, \ldots, e_{\ell-1})_\le$. 
\end{lemma}

\begin{lemma}
\label{lem:samexp}
	Suppose $\Ac $ is free with exponents $\exp(\Ac) = (1,e,\ldots,e)$ for some $e \ge 1$. 
	Then $\Ac$ is flag-accurate if and only if $\Ac$ is divisionally free. 
	In particular, if $\Ac$ is inductively free with $\exp(\Ac) = (1,e,\ldots,e)$, then $\Ac$ is flag-accurate.
\end{lemma}

Next, we recall the well-known modular coatom technique.

 \begin{proposition}
 	 \label{prop:modular coatom}
 Let $\Ac$ be a central arrangement and let  $X\in L(\Ac)$ be a modular coatom. 
Then the following statements hold.
\begin{itemize}
\item[(i)] $t\cdot \chi(\Ac, t) = \left(t-|\Ac\setminus\Ac_{X}|\right)\cdot\chi(\Ac_{X}, t) $. 
\item[(ii)] $ \Ac $ is supersolvable  (resp., (inductively) free)  if and only if $ \Ac_{X} $ is supersolvable  (resp., (inductively) free). 
In this case, $ \exp (\Ac) \cup \{0\} = \exp (\Ac_X) \cup \{|\Ac\setminus\Ac_{X}|\} $. 
\item[(iii)] If  $\Ac_{X}$ is ((ind-)flag-)accurate whose exponents do not exceed  $|\Ac\setminus\Ac_{X}|$, then $ \Ac $ is ((ind-)flag-)accurate.
\item[(iv)] If  $\Ac_{X}$ is almost accurate, then $ \Ac $ is also almost accurate.
\end{itemize}
\end{proposition}

\begin{proof}
Part (i) follows from {\cite[Thm.~ 2]{St71}}. 
For the forward implication in (ii) note that supersolvability, freeness and inductive freeness are closed under taking localizations (see \cite[Prop.~3.2]{St72}, \cite[Thm.~ 4.37]{OrTer92_Arr}, \cite[Thm.~1.1]{HRS17}).
For (iii), (iv) and the reverse implication in (ii) note that there exists a $H\in \Ac\setminus\Ac_{X}$ such that the essentializations $(\Ac^H)^{\mathrm{ess}}$ and $\Ac_X^{\mathrm{ess}}$ are linearly equivalent (see {\cite[Lem.~2.2]{MMR22}}). 
\end{proof}

The following general statement due to Abe \cite[Thm.~ 6.2]{abe:divfree} was first introduced to give a sufficient condition for divisional freeness. 
Its proof applies equally to flag-accuracy.

\begin{theorem}
	\label{thm:extshi-FA-cri}
Assume that there are distinct hyperplanes $H_1, \ldots , H_{\ell-1} \in \Ac$ so that the following conditions hold:
 	\begin{itemize}
		\item [(i)]
		$\Ac'_i :=\Ac\setminus \{H_i\}$ is free with $\exp(\Ac'_i ) = (1,d_1,\ldots,d_{i-1},d_i -1,d_{i+1}, \ldots ,d_{\ell-1})$ for each $1 \le i \le \ell-1$.
		
		\item[(ii)]
		$\Ac' :=\Ac\setminus \{H_1, \ldots , H_{\ell-1} \}$ is free with $\exp(\Ac') = (1,d_1-1,d_2-1,\ldots, d_{\ell-1}-1)$.
	\end{itemize}
Then $\Ac$ is flag-accurate with $\exp(\Ac)= (1,d_1,d_2,\ldots, d_{\ell-1})$. 

Moreover, suppose $1 \le d_1 \le d_2\le \cdots \le d_{\ell-1}$ and set 
$X_i : = \bigcap_{j=i}^{\ell-1} H_j$ for $1 \le i \le \ell-1$.
Then $(X_1, \ldots,  X_{\ell-1}, X_{\ell}=V)$ is a witness for the flag-accuracy of $\Ac$.
\end{theorem}

We also recall the following similar statement from \cite{mueckschroehrle:accurate}.
\begin{corollary}[{\cite[Cor.~2.13]{mueckschroehrle:accurate}}]
	\label{cor:MAT-stepRestriction}
	Let $\Ac'$ be free, $\Ac = \Ac' \dot{\cup} \{H_1,\ldots,H_p\}$ an MAT-step
	and $\Cc \subseteq \{H_1,\ldots,H_p\}$.
	Suppose that $\exp(\Ac) = (e_1,\ldots,e_\ell)_\leq$ and let
	$X := \cap_{H\in\Cc}H$.
	Then $\Ac^X$ is free with $\exp(\Ac^X) = (e_1,\ldots,e_{\ell-|\Cc|})_\leq$.
\end{corollary}

From the previous corollary and Theorem \ref{thm:MATRestGen} we readily obtain the following.

\begin{corollary}
	\label{coro:MATRk4-FA}
	Let $\Ac$ be a free arrangement of rank at most $4$ which is obtained from the free arrangement $\Ac'$ by MAT-steps such that one of the steps has size at least $2$.
	Then $\Ac$ is flag-accurate.
\end{corollary}

\begin{proof}
	Without loss we may assume that $\Ac$ is a free arrangement in $\KK^4$
	with  $\exp(\Ac) = (e_1,e_2,e_3,e_4)_\leq$.
	Assume further that $\Ac = \Ac' \cup \Bc$
	and that $\pi = (\pi_1 | \cdots | \pi_n )$
	is a partition of $\Bc$ yielding the successive MAT-steps form $\Ac'$ to $\Ac$.
	By assumption there is a $k \in \{1,\ldots,n\}$ 
	such that $|\pi_k| \geq 2$ and $|\pi_i|\leq 1$ for $k<i\leq n$.
	By Theorem \ref{thm:MATRestGen} there is an $H \in \pi_k$ such that
	$\Ac^H$ is free with $\exp(\Ac^H) = (e_1,e_2,e_3)$. 
	Moreover, for $\Dc := \Ac'\cup (\cup_{i=1}^k \pi_i) \subseteq \Ac$ we 
	have that $\Dc^H \subseteq \Ac^H$ is free with $\exp(\Dc^H) = (e_1,e_2,e_3)$ by 
	Corollary \ref{cor:MAT-stepRestriction}, since $\pi_k$ is an MAT-step. 
	In particular, $|\Ac^H| = |\Dc^H|$, so $\Ac^H = \Dc^H$.
	Thus, for any other $H' \in \pi_k \setminus \{H\}$ and $X=H \cap H'$
	we have that $\Ac^X = \Dc^X$ is free with $\exp(\Ac^X) = (e_1,e_2)$, again, thanks to Corollary \ref{cor:MAT-stepRestriction}.
	As a rank $2$ arrangement, $\Ac^X$ is flag-accurate, so $\Ac^H$ is flag-accurate
	by Lemma \ref{lem:flag-accuracy} and $\Ac$ is flag accurate as well.
\end{proof}

\section{Flag-accurate reflection arrangements}
\label{sec:ReflArr}

\subsection{Complex reflection arrangements and their restrictions}
\label{ssect:refl}

Let $G \subseteq \GL(V)$ be a finite, 
complex reflection group acting on the complex vector space $V=\CC^\ell$.
The \emph{reflection arrangement} of $G$ in $V$ is the 
hyperplane arrangement $\Ac(G)$ 
consisting of the reflecting hyperplanes 
of the elements in $G$ acting as reflections on $V$.

Terao \cite{Terao1980_FreeUniRefArr} has shown that every 
reflection arrangement $\Ac(G)$ is free 
and that the exponents of $\Ac(G)$
coincide with the coexponents of $G$, 
cf.~\cite[Prop.~6.59 and Thm.~6.60]{OrTer92_Arr}.

In \cite[Thm.~1.6]{mueckschroehrle:accurate}, it was shown that all 
Coxeter arrangements are accurate and in \cite[Thm.~5.8]{mueckschroehrle:accurate}, that more generally  
a complex reflection arrangement is accurate if and only if it is divisionally free.
In view of these results, it is natural to examine flag-accuracy 
for the class of complex reflection arrangements. 
Thanks to Remark \ref{rem_MATRest-flag}(v) and \cite[Prop.~2.12]{roehrle:divfree}, 
flag-accuracy and divisional freeness are compatible with products.

The proofs from \cite{mueckschroehrle:accurate} carry over almost immediately. So that we observe that 
a reflection arrangement is flag-accurate if and only if it is accurate; likewise for restrictions of reflection arrangements. We record both results correcting an omission in the second in \cite[Thm.~5.12]{mueckschroehrle:accurate}.

\begin{theorem}
	\label{thm:complex-flag}
	Let $G$ be a complex reflection group with reflection arrangement $\Ac = \Ac(G)$. Then $\Ac$ is flag-accurate 
	if and only if it is divisionally free. This is the case if and only if $G$ has no irreducible factor isomorphic 
	to one of the monomial groups $G(r,r,\ell)$, $r>2$, $\ell>2$, or   
	$ G_{24}, G_{27}, G_{29}, G_{33}, G_{34}$.
\end{theorem}

In view of Theorem \ref{thm:complex-flag} and the fact that all restrictions of complex reflection arrangements are free (thanks to \cite[\S 6.4, App.~D]{OrTer92_Arr}, 
\cite{OrTer92_CoxeterArrRestr}, and \cite{hogeroehrle:free}), 
it is natural to investigate flag-accuracy 
among restrictions of complex reflection arrangements, 
not all of which are reflection arrangements again. 

In order to state our results, we require some further notation.
Orlik and Solomon defined \emph{intermediate 
	arrangements} $\Ac^k_\ell(r)$ in 
\cite[\S 2]{OrSol83_CoxeterArr}
(cf.\ \cite[\S 6.4]{OrTer92_Arr}) which
interpolate between the
reflection arrangements of the monomial groups $G(r,r,\ell)$ and $G(r,1,\ell)$. 
They show up as restrictions of the reflection arrangement
of $G(r,r,\ell')$, for some $\ell'$,  
\cite[Prop.\ 2.14]{OrSol83_CoxeterArr} 
(cf.~\cite[Prop.\ 6.84]{OrTer92_Arr}).

For 
$\ell, r \geq 2$ and $0 \leq k \leq \ell$ the defining polynomial of
$\Ac^k_\ell(r)$ is given by
\[
Q(\Ac^k_\ell(r)) = x_1 \cdots x_k\prod\limits_{\substack{1 \leq i < j \leq \ell\\ 0 \leq n < r}}(x_i - \zeta^nx_j),
\]
where $\zeta$ is a primitive $r$-th root of unity,
so that 
$\Ac^\ell_\ell(r) = \Ac(G(r,1,\ell))$ and 
$\Ac^0_\ell(r) = \Ac(G(r,r,\ell))$. 
For $k \neq 0, \ell$, these are not reflection arrangements
themselves. 

Next we recall
\cite[Props.\ 2.11,  2.13]{OrSol83_CoxeterArr}
(cf.~\cite[Props.~6.82,  6.85]{OrTer92_Arr}).

\begin{proposition}
	\label{prop:intermediate}
	Let $\Ac = \Ac^k_\ell(r)$ for $\ell, r \geq 2$ and $0 \leq k \leq \ell$. Then
	\begin{itemize}
		\item[(i)] $\Ac$ is free with $$\exp(\Ac) = (1, r + 1, \ldots, (\ell - 2)r + 1, (\ell - 1)r - \ell + k + 1),$$
		\item[(ii)] for $H \in \Ac$, the type of $\Ac^H$ is given in Table \ref{table2}.
	\end{itemize}
\end{proposition}

\begin{table}[ht!b]
	\renewcommand{\arraystretch}{1.5}
	\begin{tabular}{llll}\hline
		$k$ & \multicolumn{2}{l}{$\alpha_H$} & Type of $\Ac^H$\\ \hline
		$0$ & arbitrary & & $\Ac^1_{\ell - 1}(r)$\\
		$1, \ldots, \ell - 1$ & $x_i - \zeta x_j$ & $1 \leq i < j \leq k < \ell$ & $\Ac^{k - 1}_{\ell - 1}(r)$\\
		$1, \ldots, \ell - 1$ & $x_i - \zeta x_j$ & $1 \leq i \leq k < j \leq \ell$ & $\Ac^k_{\ell - 1}(r)$\\
		$1, \ldots, \ell - 1$ & $x_i - \zeta x_j$ & $1 \leq k < i < j \leq \ell$ & $\Ac^{k + 1}_{\ell - 1}(r)$\\
		$1, \ldots, \ell - 1$ & $x_i$ & $1 \leq i \leq \ell$ & $\Ac^{\ell - 1}_{\ell - 1}(r)$\\
		$\ell$ & arbitrary & & $\Ac^{\ell - 1}_{\ell - 1}(r)$\\ \hline
	\end{tabular}
	\caption{Restriction types of $\Ac^k_\ell(r)$}
	\label{table2}
\end{table}

The following two results are the counterparts of \cite[Lem.~5.10]{mueckschroehrle:accurate} and 
\cite[Thm.~5.12]{mueckschroehrle:accurate} for flag-accuracy.

\begin{lemma}
	\label{lem:akl-flag}
	Let $\Ac = \Ac^k_\ell(r)$ for $\ell, r \geq 2$ and $1 \leq k \leq \ell-1$.
	Then 
	\begin{itemize}
		\item[(i)]  for $r =2$, $\Ac$ is flag-accurate;
		\item[(ii)]  for $r > 2$, $\Ac$ is flag-accurate if and only if $r + k \ge \ell$.
	\end{itemize}
\end{lemma}

\begin{theorem}	
	\label{thm:restr-flag}
	Let $G$ be an irreducible complex reflection group with reflection arrangement $\Ac(G)$. Let $\Ac = \Ac(G)^Y$, for $Y \in L(\Ac)\setminus \{V\}$. Then $\Ac$ is flag-accurate
	if and only if one of the following holds:
	\begin{itemize}
		\item[(i)] $G \ne G(r,r,\ell),  G_{34}$; 
		\item[(ii)] $G = G_{34}$ and $\Ac \ne \Ac(G)^H$; 
		\item [(iii)] $G = G(r,r,\ell)$
		 and either $r = 2$ or else $\Ac = \Ac_{\ell'}^k(r)$ with $r + k \ge \ell'$ for $r > 2$.
	\end{itemize}
\end{theorem}

Next we describe all ind-flag-accurate reflection arrangements and all ind-flag-accurate restrictions of reflection arrangements.
Since all inductively free reflection arrangements are hereditarily inductively free, owing to \cite[Thm.~1.2]{hogeroehrle:inductivelyfree} likewise for all inductively free  restrictions of reflection arrangements, by \cite[Thm.~1.3]{amendhogeroehrle:indfree},
the ind-flag-accurate reflection arrangements are simply the flag-accurate ones among the inductively free ones and likewise for the ind-flag-accurate  restrictions of  reflection arrangements. Consequently, 
Theorems \ref{thm:complex-ind-flag} and  \ref{thm:restr-ind-flag}
 are immediate from Theorem \ref{thm:complex-flag} and 
\cite[Thm.~1.1]{hogeroehrle:inductivelyfree}, respectively,  Theorem \ref{thm:restr-flag} and 
\cite[Thm.~1.2]{amendhogeroehrle:indfree}.\\

Next we revisit \cite[Ex.~5.5]{mueckschroehrle:accurate}
which gives an instance of an accurate but non-flag-accurate arrangement.

\begin{example}
	\label{ex:star-but-not-divfree}
	In \cite{HogeRoehrle19_ConjAbeAF} a free arrangement $\Dc$ is constructed within a rank 5 restriction of the Weyl arrangement of type $E_7$ which is not divisionally free with exponents $(1,5,5,5,5)$ and defining polynomial
	\begin{align*}
		Q(\Dc) = \, &x_2(x_1+x_3-x_5)(2x_1+x_2+x_3)(2x_1+x_2+2x_3+x_4-x_5)\\
		&x_5(x_1+x_3)(x_2+x_5)(2x_1+x_2+2x_3+x_4)(2x_1+x_3-x_5)\\
		&(2x_1+2x_2+2x_3+x_4)(x_2+x_3+x_4)(x_1+x_2+x_3+x_4)\\
		&(x_3+x_4)(x_1+x_2+x_3)x_1(x_1+x_3+x_4)(2x_1+x_2+x_3-x_5)\\
		&(x_2+x_3+x_4+x_5)(x_1-x_5)(x_1-x_4-x_5)x_4.
	\end{align*}
	
	One can check that $\Dc $ is still accurate:
	Only the restriction to $H =\ker(x_4)$ has exponents $(1,5,5,5)$ and only the restrictions to
	\begin{align*}
		X_1 & = \ker(2x_1+x_2+x_3) \cap \ker(2x_1+x_2+2x_3+x_4-x_5), \\
		X_2 & = \ker(x_2+x_5) \cap \ker(x_2+x_3+x_4)
	\end{align*}
	are free with exponents $\exp(\Dc^{X_1}) = \exp(\Dc ^{X_2}) = (1,5,5)$. 
	However, neither of those flats is contained in $H = \ker(x_4)$.
	Note further that 
	$ Y_1 = X_1 \cap H$ and also $Y_2 = X_2 \cap H$ are rank 3 flats with 
	$\exp(\Dc^{Y_1}) = \exp(\Dc^{Y_2}) = (1,5)$. 
	In particular, the lack of a suitable rank $2$ flat lying between $Y_1$ (or $Y_2$)
	and $H$ prevents $\Dc$ from being divisionally free.
	In particular, $\Dc$ is not flag-accurate.
	Specifically, $\Dc$ is only $2$-accurate.
	
	It is also easily seen that the free but non-divisionally free rank $7$ arrangement
	$\Bc$ constructed in \cite{HogeRoehrle19_ConjAbeAF} as a certain subarrangement of the Weyl arrangement of type $E_7$ is also still accurate. But again it is not flag-accurate.	
	
	In \cite[\S 6]{CunMue19_MATfree}, Cuntz and M\"ucksch checked that both $\Dc$ and $\Bc$ fail to be MAT-free.
\end{example}

\subsection{Ideal arrangements}
\label{SSec_WeylIdeal}

For general information about Weyl groups and their root systems, see \cite{bourbaki:groupes}.

Let $W$ be a Weyl group acting as a reflection group on the real vector space $V = \RR^\ell$.
Let $\Phi := \Phi(W) \subseteq V^*$ be a (reduced) root system for $W$ and $\Phi^+ \subseteq \Phi$ a positive system
with simple roots $\Delta \subseteq \Phi^+$.
The \emph{rank} of $W$ respectively $\Phi$ is $\rk(W) := \rk(\Phi) := \dim (\RR\Phi)$.
We have $\Phi^+ = (\sum_{\alpha \in \Delta} \ZZ_{\geq 0} \alpha) \cap \Phi$, i.e.\
if $\beta \in \Phi^+$ then there are integers $n_\alpha \in \ZZ_{\geq0}$ such that
$\beta = \sum_{\alpha \in \Delta} n_\alpha \alpha$.
Then the \emph{height} of $\beta$ is defined by
$$
\h(\beta) := \sum_{\alpha\in\Delta}n_\alpha.
$$

The partial order $\leq$ on $\Phi^+$ is defined by
$$
\beta \leq \gamma \, :\iff \, \gamma-\beta \in \sum_{\alpha \in \Delta} \ZZ_{\geq 0}\alpha.
$$

A subset $\Ic \subseteq \Phi^+$ is an \emph{ideal} if it is a (lower) order ideal 
in the poset $(\Phi^+,\leq)$, i.e.\ for $\alpha \in \Ic$ and $\beta \in \Phi^+$
with $\beta \leq \alpha$, we have $\beta \in \Ic$.

The \emph{Weyl arrangement} $\Ac(W)$
is the hyperplane arrangement in $V$ defined by 
$$
\Ac = \Ac(W) := \{\ker(\beta) \mid \beta \in \Phi^+ \}.
$$

\begin{definition}[{\cite{ABCHT16_FreeIdealWeyl}}]
	\label{DEF_ideal}	
	If $\Ic \subseteq \Phi^+$ is an order ideal then 
	\[
	\Ac_\Ic := \{ \ker(\beta) \mid \beta \in \Ic \} \subseteq \Ac(W)
	\]
	is called an 
	\emph{ideal (sub)arrangement}.
\end{definition}

We denote by $m_\Ic$ the maximal height of a root in $\Ic$.
For $1 \leq k \leq m_\Ic$,
let $$\pi_{k,\Ic} := \{ \ker(\alpha) \mid \alpha \in \Ic, \h(\alpha) = k \}$$ 
and let 
\[
\pi_\Ic := \left(\pi_{1,\Ic}| \cdots | \pi_{m_\Ic,\Ic}\right),
\] %
be the \emph{root-height partition} of $\Ac_\Ic$. 
%Set $$p_{k,\Ic} := |\pi_{k,\Ic}|$$ for $1\leq k \leq m_\Ic$ and $p_{m_\Ic+1,\Ic} := 0$.
Set $\pi_{j,\Ic} = \emptyset$ for $j > m_\Ic$.

Next we recall the principal result from \cite[Thm.~1.1]{ABCHT16_FreeIdealWeyl} (Ideal-free Theorem) in our terminology.

\begin{theorem}
	\label{thm:IdealFree}
	Let $\Ac_{\Ic} \subseteq \Ac(W)$ be the ideal subarrangement of the Weyl arrangement $\Ac(W)$
	for an order ideal $\Ic \subseteq \Phi^+$. 
	Then $\Ac_\Ic$ is MAT-free with MAT-partition
	$\pi_\Ic = (\pi_{1,\Ic} |\cdots| \pi_{m_\Ic,\Ic} )$ and exponents 
	\[
	\exp\left(\Ac_\Ic\right) = \left(e_1^{\Ic},\ldots,e_{\ell}^{\Ic}\right),
	\] 
	where
	\begin{align*}
		e_r^{\Ic} = |\{ j \mid |\pi_{j,\Ic}| \geq \ell-r+1\}|.
	\end{align*}
\end{theorem}

Theorems \ref{Thm_MATRest} and \ref{thm:IdealFree} and the following theorem are the basis for 
Conjecture \ref{conj:Ideal-flag}.

\begin{theorem}
	\label{thm:IdealUpToRk8FA}
	Ideal arrangements of rank at most $8$ are  flag-accurate.
\end{theorem}

\begin{proof}
	Firstly, all ideal arrangements of rank at most $4$ are flag-accurate, by Theorem \ref{thm:IdealFree} and Corollary \ref{coro:MATRk4-FA}.
	
	For ideal arrangements of rank $5$ up to $8$, we used a computer to check flag-accuracy.
	The computation can be considerably simplified in a large number of cases by using Theorem \ref{thm:MATRestGen} as follows.
	In the root height partition (of an ideal arrangement of rank at least $5$) 
	there is always a block of size at least $2$ which is either the last block of the partition
	or one only followed by blocks of size $1$.
	Consequently, analogous to the proof of Corollary \ref{coro:MATRk4-FA},
	a partial flag up to the restriction of all the hyperplanes in this block can always be built,
	realizing subsets of the exponents in correct order by Theorem \ref{thm:MATRestGen}.
	Now, in most of the cases it turns out that the restriction to the intersection of all the hyperplanes in this block
	is itself flag-accurate, yielding the flag-accuracy of the whole ideal arrangement.
	Unfortunately, there are still cases of ideals in the root systems of type $E_6,E_7,E_8$ where this heuristic doesn't work.
	Hence a hard computer check in these cases (a few hundred) is unavoidable but readily manageable.
\end{proof}

\begin{remark}
	\label{rem:IdealUpToRk6IFA}
	Concerning Conjecture \ref{conj:Ideal-ind-flag},   we 
	were able to confirm with the aid of a computer that all 
	ideal arrangements $\Ac_\Ic$ of rank up to $6$ are indeed  ind-flag-accurate.
\end{remark}

\subsection{Crystallographic arrangements}

Crystallographic arrangements were first introduced and studied by Cuntz
and Heckenberger in the setting of finite Weyl groupoids,
culminating in a complete classification in \cite{CuntzHeckenberger2015_WeylGroupoids},
see also \cite{Cunty2011_CrystArr}.
They are generalization of Weyl arrangements or even of restrictions of Weyl arrangements
as every Weyl arrangement is crystallographic and the class of crystallographic arrangements
is closed under taking restrictions, cf.\ \cite[Prop.~5.3]{BC12}.

Let us recall the definition of a crystallographic arrangement.

\begin{definition}
	\label{def:CrystArr}
	Let $\Ac$ be a hyperplane arrangement in $V\cong \RR^\ell$.
	Denote the \emph{chambers} of $\Ac$, i.e.\ the connected components of $V \setminus \left(\cup_{H \in \Ac} H\right)$
	by $\Kc(\Ac)$.
	
	For $C \in \Kc(\Ac)$ define the \emph{walls} of $C$ as
	$$
	\Wc^C := \{ H \in \Ac \mid \langle H\cap\overline{C} \rangle = H \}.
	$$
	
	If $\Phi \subseteq V^*$ is a finite set such that $\Ac = \{ \ker(\alpha) \mid \alpha \in \Phi\}$
	and $\RR\alpha \cap \Phi = \{\pm \alpha \}$ for all $\alpha \in \Phi$
	then $\Phi$ is called a \emph{(reduced) root system} for $\Ac$.
	
	If $\Phi$ is a root system for $\Ac$, then for each $C$ set
	$$
	B_\Phi^C := \{ \alpha \in \Phi \mid \ker(\alpha) \in \Wc^C \text{ and } \alpha^{-1}(\RR_{>0})\supseteq  C \} \subseteq \Phi.
	$$
	
	If every $C \in \Kc(\Ac)$ is a simplicial cone 
	and there exists a root system $\Phi \subseteq V^*$ for $\Ac$ such that 
	$$
	\Phi \subseteq \sum_{\alpha \in B_\Phi^C} \ZZ\alpha \text{ for all }C \in \Kc(\Ac),
	$$
	then $\Ac$ is called \emph{crystallographic} and $\Phi$ a \emph{crystallographic root system}
	for $\Ac$.
\end{definition}

Crystallographic arrangements were classified by Cuntz and Heckenberger \cite{CuntzHeckenberger2015_WeylGroupoids}.
Perusing at the classification one observes that
each irreducible crystallographic arrangement of rank at least $4$ and each member of the infinite series
is a restriction of a Weyl arrangement, see  \cite[Thm.~3.7]{CuntzLenter2017_SimplCpxNichols}.
Hence, to derive the following result, 
by Theorem \ref{thm:restr-flag}, it suffices to check the finite list of sporadic 
irreducible rank $3$ arrangements, cf.\  \cite[App.~B.1]{CuntzHeckenberger2015_WeylGroupoids}.

\begin{theorem}
	\label{thm:CrystallographicFA}
	Crystallographic arrangements are ind-flag-accurate.
\end{theorem}

\begin{proof}
	By \cite[Cor.~5.15]{BC12}, all restrictions of crystallographic arrangements are inductively free.
	As mentioned above, in view of Theorem \ref{thm:restr-flag}, it suffices to check the 
	finite list of sporadic irreducible crystallographic arrangements of rank $3$.
	
	Let $\Ac$ be one of the $50$ sporadic arrangements listed in \cite[App.~B.1]{CuntzHeckenberger2015_WeylGroupoids}
	and set $\exp(\Ac) = (1,e_1,e_2)_{\leq}$.
	Then a straightforward calculation shows that there is an $H \in \Ac$,
	such that $\exp(\Ac^H) = (1,e_1)$.
	Hence, $\Ac$ is ind-flag-accurate, by Lemma \ref{lem:flag-accuracy}.
\end{proof}

\section{Extended Shi, extended Catalan, and ideal-Shi arrangements}
\label{Sec_IdealShiCatalan}
Recall the notation from Section \ref{SSec_WeylIdeal}. 
In addition let $$h := m_{\Phi^+}+1$$ be the \emph{Coxeter number} of $W$.

Embed $V \subseteq V' := \RR^{\ell+1}$ and let $z \in (V')^* \setminus \{0\}$
such that $V = \ker(z) =: H_z$, i.e.\ $z$ corresponds to the $(\ell+1)$-st coordinate.
For  $\alpha \in  \Phi$ and $j \in \ZZ$ define the hyperplanes
$$H_\alpha^j := \ker(\alpha-j) \subseteq V \quad \text{and} \quad \cc H_\alpha^j := \ker(\alpha-jz) \subseteq V'.$$

When $j=0$, we simply write $H_\alpha$ for $H_\alpha^0$. 
For integers $a\le b$, let $[a,b]:=\{n\in\ZZ\mid a \le n\le b\}$. 
The \emph{deformed Weyl arrangement} of $\Phi$ is defined by 
$$\Ac_\Phi^{[a,b]} := \{ {H}_{\alpha}^j \mid \alpha \in \Phi^+, \,j \in [a,b]\}.$$  
In particular, $\Ac_\Phi^{[0,0]}$ is identical to the Weyl arrangement $\Ac(W)$ defined above. 
When $[a,b] \ne [0,0]$, the cone $\cc\Ac_\Phi^{[a,b]}$ over the non-central arrangement $\Ac_\Phi^{[a,b]}$ is given by 
$$\cc\Ac_\Phi^{[a,b]} = \{\cc {H}_{\alpha}^j \mid \alpha \in \Phi^+, \,j \in [a,b]\} \cup \{H_z \}.$$  

The study of the combinatorics of deformations of Weyl arrangements was initiated  by Athanasiadis in \cite{Ath96_CharPolySubspaceArr}; among them are the so called \emph{extended Shi arrangements}.
We are interested in the following generalization, investigated by Abe and Terao in \cite{AbeTer15_IdealShi}, the so called \emph{ideal-Shi arrangements}.

\begin{definition}
	\label{Def_IdealShiCatalan}
	For $m \in \ZZ_{>0}$, the \emph{extended Shi arrangement} $\Shi^m$
	is defined as
	\[
	\Shi^m= \Shi^m (\Phi) := \Ac_\Phi^{[1-m,m]}.
	\]
	
	For $m \in \ZZ_{>0}$ and $\Ic \subseteq \Phi^+$ an ideal,  
	the \emph{ideal-Shi arrangement} $\Shi^m_\Ic$ 
	is defined as
	\[
	\Shi^m_\Ic   = \Shi^m_\Ic(\Phi) := \Shi^m \cup \  \left\{H_\beta^{-m} \mid \beta \in \Ic \right\}.
	\]
	
	In the special case when $\Ic = \Phi^+$ 
	we obtain the \emph{extended Catalan  arrangement} $$\Cat^m = \Cat^m  (\Phi) := \Shi_{\Phi^+}^m.$$
\end{definition}

In \cite{Yos04_CharaktFree}, Yoshinaga proved the following remarkable theorem, confirming a conjecture
by Edelman and Reiner \cite[Conj.~3.3]{ER96_FreeArrRhombicTilings}.

\begin{theorem}[{\cite[Thm.~1.2]{Yos04_CharaktFree}}]
\label{thm:Yoshinaga}
	Let $m \in \ZZ_{>0}$. 
	Then the cone  over the extended Shi arrangement $\Shi^m$ is free with exponents
	\[
	\exp\left(\cc\Shi^m\right) = (1,(mh)^\ell),
	\]
	and the cone over the extended Catalan arrangement $\Cat^m$  is free with exponents
	\[
	\exp\left(\cc\Cat^m\right) = (1,mh+e_1,\ldots,mh+e_\ell),
	\]
	where $(e_1,\ldots,e_\ell) = \exp(\Ac(W))$.
\end{theorem}

We need the following special case of a result by Abe and Terao (cf.~\cite[Thm.~1.6]{AbeTer15_IdealShi}).
\begin{proposition}
	\label{Prop_Shi-SimplRoots}
	Let $\Sigma \subseteq \Delta$ be a subset of the simple roots of $\Phi^+$.
	Then the cone of the arrangement 
	\[
	\Shi^m_{-\Sigma} := \Shi^m \setminus \left\{H^m_\alpha \mid \alpha \in \Sigma\right\}
	\]
	is free with exponents given by
	\[
	\exp\left(\cc\Shi^m_{-\Sigma}\right) = \left(1,(mh-1)^{|\Sigma|},(mh)^{\ell-|\Sigma|}\right).
	\] 
\end{proposition}

We recall the main result from \cite{mueckschroehrle:accurate} for this class of arrangements.

\begin{theorem}
	[{\cite[Thm.~1.8]{mueckschroehrle:accurate}}]
	\label{thm:idealshi}
	The cones of the extended Shi arrangements $\cc\Shi^m$ and ideal-Shi arrangements $\cc\Shi^m_\Ic$ are accurate.
	In particular, the cones over the extended Catalan arrangements $\cc\Cat^m$ are accurate.
\end{theorem}

The following example illustrates that there exist  arrangements that are both
accurate and divisionally free but not flag-accurate. 
(We are going to encounter another example arising from graphic arrangements in Corollary \ref{coro:acc-nfa}.)

\begin{example}
	\label{ex:shif4}
	Let $\Delta = \{\alpha_1,\alpha_2,\alpha_3,\alpha_4\}$ be a simple system of the root
	system $\Phi$ of type $F_4$, corresponding to the labeling of the Dynkin diagram as in \cite[Planche VIII]{bourbaki:groupes}.
	%in Figure \ref{fig:Dynkin}.
	Let $\Phi^+$ be the set of positive roots with respect to $\Delta$ and $\Ic$ is the ideal consisting of the roots
	\[
		\Ic = \{\alpha_2,\alpha_3,\alpha_4,\alpha_2+\alpha_3, \alpha_3+\alpha_4\}.
	\]
	 
	Then $\Ac := \cc\Shi^1_{-\Ic} = \cc\Shi^1 \setminus \{H_\alpha^1 \mid \alpha \in \Ic\}$ is
	free with exponents $\exp(\Ac) = (1,10,10,11,12)$.
	Moreover, the arrangement $\Ac$ is accurate and also divisionally free.
	
	A calculation reveals that there are exactly four hyperplanes $H \in \Ac$ whose restriction $\Ac^H$ is still accurate.
	But none of these restrictions contains a rank $2$ intersection $X \in L(\Ac)$ such that $\Ac^X$ is accurate.
	Consequently, no restriction of $\Ac$ to a hyperplane is flag-accurate and 
	by Lemma \ref{lem:flag-accuracy} $\Ac$ is not flag-accurate.
\end{example}

The main objective in this subsection is to show that the cones over the extended Shi arrangement of arbitrary type and the cones over the extended Catalan arrangement of type $A$, $B$ or $C$ are flag-accurate.

We first introduce some notion which allows us to specify the nature of a witness for accuracy in extended Shi arrangements based on simple roots.

\begin{definition}\label{def:SRW}
	Suppose that $[a,b] \ne [0,0]$ and the cone $\cc\Ac_\Phi^{[a,b]}$ is  flag-accurate. 
	A witness $X_{\ell-1} \subseteq \cdots \subseteq  X_2  \subseteq X_1 \subseteq V'=\RR^{\ell+1} $ for the flag-accuracy of $\cc\Ac_\Phi^{[a,b]}$ is called a \emph{simple root witness} if there exist simple roots $\alpha_1,\ldots,\alpha_{\ell-1} \in \Delta$ and integers $j_1,\ldots, j_{\ell-1} \in \ZZ$ such that $X_s = \bigcap_{i=1}^{s} \cc H_{\alpha_i}^{j_i}$ for each $1 \le s \le \ell-1$. 
%	simple root witnesses of $\Ac_\Phi^{[0,0]}$ are defined similarly by replacing $V'$ by $V$.
\end{definition}

\begin{theorem}
	\label{thm:ExtShi-FA}
The cone  over the extended Shi arrangement $\Shi^m$ is  flag-accurate with a simple root witness. 
\end{theorem}

\begin{proof}
Owing to Theorem \ref{thm:extshi-FA-cri} and  Proposition \ref{Prop_Shi-SimplRoots} we derive that $\cc\Shi^m$ is  flag-accurate with a simple root witness. 
Alternately, in view of the nature of the set of exponents of $\cc\Shi^m$ (see Theorem \ref{thm:Yoshinaga}), 
the flag-accuracy of  $\cc\Shi^m$  can also be deduced from Lemma \ref{lem:samexp} and the fact that $\cc\Shi^m$ is divisionally free, thanks to  \cite[Thm.~ 6.1]{abe:divfree}.
\end{proof}

 \begin{theorem}
	\label{thm:ExtCat-FA}
	If $\Phi$ is of type $A$, $B$ or $C$, then the cone over the extended Catalan arrangement $\Cat^m  (\Phi)$  is  flag-accurate with a simple root witness. 
\end{theorem}
\begin{proof}
 The proof follows from Theorems \ref{thm:CatB-FA}, \ref{thm:CatC-FA} below, and Corollary \ref{cor:CatA-FA}. 
\end{proof}

\begin{notation}
	\label{notation}
In what follows, when writing the defining equation of a set of hyperplanes, e.g., for $N\subseteq \mathbb{Z}$ by setting $x_{i} =N$  and  $x_{i} =  Nz$, we mean the affine coordinate hyperplanes $x_{i}=n$ and their homogenizations $x_{i} =nz$ for all  $ n \in N$. 	
\end{notation}

We need an extension of Theorem \ref{thm:Yoshinaga} for extended Catalan arrangements of type $B$.

\begin{theorem}
 \label{thm:CatB-RF}
Let $a, m \ge 0$,  $\ell \ge 1$ and $0 \le p \le \ell$. 
Let $\Bb^p_\ell(m,a)$ be the arrangement consisting of the hyperplanes
\begin{align*}
x_{i}\pm x_{j} &= [-a , a]\quad  (1 \leq i < j \leq \ell), \\
x_i&=  [1-m , m] \quad (1 \leq   i \leq p), \\
x_i&=[-m , m]  \quad (p<  i \leq \ell).
\end{align*}
Then the cone $\cc\Bb^p_\ell(m,a)$
is recursively free with exponents 
$$ \exp\left(\cc\Bb^p_\ell(m,a)\right) = (1,  2\ell-p-1, 1,3,5,\ldots, 2\ell-3) +  (0,  (2m+2a\ell-2a)^{\ell} ).$$ 
 \end{theorem} 
 
 \begin{proof}
Define the lexicographic order on the set of pairs
$$\{ (\ell, \ell-p) \mid 0\le \ell- p \le \ell, \, \ell \ge 1\}.$$
We prove that $\Ac: =\textbf{c}\Bb^p_\ell(m, a)$ belongs to $\mathcal{RF}$ with the  desired exponents by induction on $ (\ell, \ell-p) $.
When $\ell=1$, this is obvious. So suppose $\ell\ge2$.

Case $1$. Suppose that $\ell-p\ge 1$. 
Let   $H \in \Ac$ denote the hyperplane $x_{p+1} = -mz$. 
Then $\Ac'=\Ac\setminus \{H\} = \textbf{c}\Bb^{p+1}_\ell(m,a)$ is recursively free with exponents 
$$ \exp\left(\textbf{c}\Bb^{p+1}_\ell(m,a)\right) = (1,  2\ell-p-2, 1,3,5,\ldots, 2\ell-3) +  (0,  (2m+2a\ell-2a)^{\ell} ),$$ 
by the induction hypothesis. 
Moreover, $\Ac''=\Ac^{H}$ consists of the following hyperplanes 
\begin{align*}
z &= 0 , \\
x_{i}\pm x_{j} &=  [-a , a]z\quad  (1 \leq i < j \leq \ell, i\ne p+1,j\ne p+1 ), \\
x_i&=  [-(m+a) , m+a]z \quad (1 \leq   i \leq \ell, i \ne p+1).
\end{align*}
Thus $\Ac'' =   \textbf{c}\Bb^{0}_{\ell-1}(m+a, a)$ is recursively free with exponents 
$$ \exp\left(\textbf{c}\Bb^{0}_{\ell-1}(m+a, a)\right) = (1,  2m+2a\ell-2a +2i-1 )_{i=1}^{\ell-1},$$  
by the induction hypothesis. 
Therefore, by the addition part of Theorem \ref{thm:AD}, also $\Ac =\textbf{c}\Bb^p_\ell(m, a)$ is recursively free with the  desired exponents. 

Case $2$. Suppose that  $\ell=p$. 
The arrangement  in question is $\textbf{c}\Bb^\ell_\ell(m,a)$ given by
\begin{align*}
z &= 0 , \\
x_{i}\pm x_{j} &=  [-a , a]z\quad  (1 \leq i < j \leq \ell), \\
x_i&=  [1-m , m]z \quad (1 \leq   i \leq \ell).
\end{align*}
We need to prove   that $\textbf{c}\Bb^\ell_\ell(m,a)$ belongs to $\mathcal{RF}$ with exponents 
$$ \exp\left(\textbf{c}\Bb^\ell_\ell(m,a)\right) = (1,  \ell-1, 1,3,\ldots, 2\ell-3) +  (0,  (2m+2a\ell-2a)^{\ell} ).$$ 
By Case $1$, we have $\textbf{c}\Bb^{\ell-1}_\ell(m,a)\in \mathcal{RF}$ with exponents 
$$\exp\left(\textbf{c}\Bb^{\ell-1}_\ell(m,a)\right) = (1,  \ell, 1,3,\ldots, 2\ell-3) +  (0,  (2m+2a\ell-2a)^{\ell} ).$$ 
Note that $\textbf{c}\Bb^\ell_\ell(m,a) =\textbf{c}\Bb^{\ell-1}_\ell(m,a) \setminus \{H_\ell\}$, where $H_\ell\in \textbf{c}\Bb^{\ell-1}_\ell(m,a)$ denotes the hyperplane $x_\ell = -mz$. 
Again by Case $1$, $\left(\textbf{c}\Bb^{\ell-1}_\ell(m,a)\right)^{H_\ell}=  \textbf{c}\Bb^{0}_{\ell-1}(m+a, a)$.
Thus $\left(\textbf{c}\Bb^{\ell-1}_\ell(m,a)\right)^{H_\ell}$ belongs to $\mathcal{RF}$ with exponents $(1,  2m+2a\ell-2a +2i-1 )_{i=1}^{\ell-1}$, by our induction hypothesis. 
By applying the deletion part of Theorem \ref{thm:AD}, we deduce  that  
$ \textbf{c}\Bb^\ell_\ell(m,a)$ is recursively free with the desired exponents. 
\end{proof} 

We have a similar result for type $C$ but the proof is more complicated than the one in type $B$. 
\begin{theorem}
 \label{thm:CatC-F}
Let $a, m\ge0$, $\ell \ge 1$ and $0 \le p \le \ell$. 
Let $\Ccal^p_\ell(m,a)$ be the arrangement consisting of the hyperplanes
\begin{align*}
x_{i}\pm x_{j} &= [-a , a]\quad  (1 \leq i < j \leq \ell), \\
2x_i&=  [1-m , m] \quad (1 \leq   i \leq p), \\
2x_i&=[-m , m]  \quad (p<  i \leq \ell).
\end{align*}
Then the cone $\cc\Ccal^p_\ell(m,a)$ is free with exponents 
$$ \exp\left(\cc\Ccal^p_\ell(m,a)\right) = (1,  2\ell-p-1, 1,3,5,\ldots, 2\ell-3) +  (0,  (2m+2a\ell-2a)^{\ell} ).$$ 
 
 \end{theorem} 
 
 \begin{proof}
Define the lexicographic order on 
$$\{ (\ell, \ell-p) \mid 0\le \ell- p \le \ell, \, \ell \ge 1\}.$$
We prove that $\Ac: =\textbf{c}\Ccal^p_\ell(m, a)$ is free  with the  desired exponents by induction on $ (\ell, \ell-p) $.
When $\ell=1$, the result  is obvious. So suppose $\ell\ge2$.

Case $1$. Suppose $\ell-p\ge 1$. 
Let   $H \in \Ac$ denote the hyperplane $2x_{p+1} = -mz$. 
Then $\Ac\setminus \{H\} = \textbf{c}\Ccal^{p+1}_\ell(m,a)$ is free  with exponents 
$$ \exp\left(\textbf{c}\Ccal^{p+1}_\ell(m,a)\right) = (1,  2\ell-p-2, 1,3,5,\ldots, 2\ell-3) +  (0,  (2m+2a\ell-2a)^{\ell} ),$$ 
by our induction hypothesis. 

For $a, m, n\ge0$ and $\ell \ge 1$, define the arrangement $\widetilde\Ccal_\ell(m,a,n)$  consisting of the hyperplanes
\begin{align*}
x_{i}\pm x_{j} &=  [-a , a]\quad  (1 \leq i < j \leq \ell), \\
2x_i&=  [-m, m] \cup \{\pm(m+2), \pm(m+4), \ldots, \pm(m+2na)\} \quad (1 \leq   i \leq \ell).
\end{align*}

One can check that $\Ac^{H}=\textbf{c} \widetilde\Ccal_{\ell-1}(m, a, 1)$ (we need to treat the two cases $m\ge a$ and $m<a$ separately). 
We show that $\Ac^{H}$ is free  with exponents $(1,  2m+2a\ell-2a +2i-1 )_{i=1}^{\ell-1}$. 
Then we apply Theorem \ref{thm:AD} to deduce that $\Ac =\textbf{c}\Ccal^p_\ell(m, a)$ is free  with the  desired exponents.  

The case   $\ell=2$ is clear. Suppose $\ell\ge3$.
Consider $ \textbf{c}\Ccal^{0}_{\ell-1}(m+a, a)$. 
Note that thanks to our induction hypothesis, the arrangement $ \textbf{c}\Ccal^{0}_{\ell-1}(m+a, a)$ is  free  with exponents 
$$\exp\left(\textbf{c}\Ccal^{0}_{\ell-1}(m+a, a)\right) = (1,  2m+2a\ell-2a +2i-1 )_{i=1}^{\ell-1}.$$ 
Moreover, $\Ac^{H}$ and $ \textbf{c}\Ccal^{0}_{\ell-1}(m+a, a)$ share the  Ziegler restriction onto the hyperplane at infinity $H_\infty = \ker z$. 
By Theorem \ref{thm:Yoshinaga's criterion}, it suffices to prove that $\Ac^{H}$ is  locally free in codimension three along $H_\infty$. 
By \cite[Lem.~3.1]{AbeTer11_ShiCat} (cf.~also the case (2-ii) in the proof of  \cite[Prop.~2.4]{AbeTer15_ShiSRB}), 
this is the case once we know that the cone $\cc \Cat^a (A_2)$ over the type $A_2$ extended Catalan arrangement  and $\textbf{c} \widetilde\Ccal_{2}(m, a, 1)$ are both free. 
The former is known to be free by Theorem \ref{thm:Yoshinaga} (see also Corollary \ref{cor:CatA-FA}). 

Concerning the freeness of the latter, we actually show a more general statement, namely that  $\textbf{c} \widetilde\Ccal_{2}(m, a, n)$ is free for $n \ge 0$  with exponents 
$$\exp\left(\textbf{c} \widetilde\Ccal_{2}(m, a, n)\right) =  (1,2m+2na+2a+1,2m+2na+2a+3).$$ 
Then the special case $n=1$ gives the desired result for $\textbf{c} \widetilde\Ccal_{2}(m, a, 1)$. 
We argue by induction on $n \ge 0$. 
Since $ \textbf{c}\Ccal^{0}_{\ell-1}(m, a)$ is also free by our induction hypothesis on $ (\ell, \ell-p) $, the arrangement  $\textbf{c} \widetilde\Ccal_{2}(m, a, 0)= \textbf{c}\Ccal^{0}_{2}(m, a)$ consisting of 
\begin{align*}
z &= 0 , \\
x_1\pm x_2 &=  [-a , a]z, \\
2x_i&=  [-m , m]z \quad (1 \leq   i \leq 2)
\end{align*}
is a localization of  $ \textbf{c}\Ccal^{0}_{\ell-1}(m, a)$ hence free with  exponents $ (1,2m+2a+1,2m+2a+3)$. 
Thus the base case $n=0$ is clear. 

Next we show that if $\textbf{c} \widetilde\Ccal_{2}(m, a, n)$ is free then so is $\textbf{c} \widetilde\Ccal_{2}(m, a, n+1)$ for $n\ge 0$. 
Define $4a$ hyperplanes by setting
\begin{align*}
K_{4s-3}  & : 2x_2 = (m+2na+2s)z,  \quad K_{4s-2} : 2x_1 = (m+2na+2s)z, \\
K_{4s-1}  & : 2x_2 = -(m+2na+2s)z,  \quad K_{4s} : 2x_1 = -(m+2na+2s)z,
\end{align*}
for each $1 \le s \le a$.
 By adding the hyperplanes $K_1, K_2, \ldots, K_{4a}$ to $ \textbf{c} \widetilde\Ccal_{2}(m, a, n) $ in this order and by applying Theorem \ref{thm:AD} to each addition step, we are able to conclude that  $\textbf{c} \widetilde\Ccal_{2}(m, a, n+1)$ is free with the desired exponents.

Case $2$. Suppose  $\ell=p$. 
The arrangement  in question is $\textbf{c}\Ccal^\ell_\ell(m,a)$ given by
\begin{align*}
z &= 0 , \\
x_{i}\pm x_{j} &=  [-a , a]z\quad  (1 \leq i < j \leq \ell), \\
2x_i&=  [1-m , m]z \quad (1 \leq   i \leq \ell).
\end{align*}
We aim to prove   that $\textbf{c}\Ccal^\ell_\ell(m,a)$ is free  with exponents 
$$ \exp\left(\textbf{c}\Ccal^\ell_\ell(m,a)\right) = (1,  \ell-1, 1,3,\ldots, 2\ell-3) +  (0,  (2m+2a\ell-2a)^{\ell} ).$$ 
According to Case $1$, $\textbf{c}\Ccal^{\ell-1}_\ell(m,a)$ is free  with exponents 
$$\exp\left(\textbf{c}\Ccal^{\ell-1}_\ell(m,a)\right) = (1,  \ell, 1,3,\ldots, 2\ell-3) +  (0,  (2m+2a\ell-2a)^{\ell} ).$$ 
Note that $\textbf{c}\Ccal^\ell_\ell(m,a) =\textbf{c}\Ccal^{\ell-1}_\ell(m,a) \setminus \{H_\ell\}$, where $H_\ell\in \textbf{c}\Ccal^{\ell-1}_\ell(m,a)$ denotes the hyperplane given by $2x_\ell = -mz$. 
Again by Case $1$, $\left(\textbf{c}\Ccal^{\ell-1}_\ell(m,a)\right)^{H_\ell} $ is free  with exponents $(1,  2m+2a\ell-2a +2i-1 )_{i=1}^{\ell-1}$. 
Applying the deletion part of Theorem \ref{thm:AD}, we infer that  
$ \textbf{c}\Ccal^\ell_\ell(m,a) $ is free  with the desired exponents. 
  \end{proof} 

Now we are ready to give the proofs for the flag-accuracy of the extended Catalan arrangements of type $B$ and $C$.
\begin{theorem}
	\label{thm:CatB-FA}
	For integers $m,a, \ell$ as in Theorem \ref{thm:CatB-RF}, the arrangement $\cc\Bb^0_\ell(m,a)$ is flag-accurate. 	
In particular, the extended Catalan arrangement $ \Cat^m (B_\ell)$ of type $B_\ell$ (Definition \ref{Def_IdealShiCatalan}), equal to $\Bb^0_\ell(m,m)$, has  flag-accurate cone with exponents 
$$\exp\left(\cc\Cat^m (B_\ell)\right) = (1, 2m\ell +2i-1 )_{i=1}^{\ell}.$$ 
\end{theorem}

\begin{proof}
 The flag-accuracy of $\cc\Bb^0_\ell(m,a)$ can be shown by constructing a witness as hinted in Case $1$ in the proof of  Theorem \ref{thm:CatB-RF}. 
Define $\ell$ hyperplanes $H_1, \ldots, H_\ell \in \cc\Bb^0_\ell(m,a)$ as follows:
$$H_\ell: x_\ell = -mz, \quad H_i: x_{i} - x_{i+1} = -az \quad (1 \le i \le \ell-1).$$
Set $X_n : = \bigcap_{j=n}^{\ell} H_j$ for $1 \le n \le \ell$. 
Then $X_1 \subseteq X_2 \subseteq \cdots \subseteq  X_\ell \subseteq V'=\RR^{\ell+1} $ and $\dim_{V'}(X_n) =n$.
Moreover,  for each $1 \le n \le \ell$, one may show that 
$$\Ac^{X_n} =   \textbf{c}\Bb^{0}_{n-1}(m+(\ell-n+1)a, a).$$ 
Hence, by  Theorem \ref{thm:CatB-RF}, $ \Ac^{X_n} \in \mathcal{RF}$ with exponents $(1,  2m+2a\ell-2a +2i-1 )_{i=1}^{n-1}$. 
Thus  $(X_1, X_2 , \ldots, X_\ell, V' )$  is a witness for the flag-accuracy of $\cc\Bb^0_\ell(m,a)$.
\end{proof}

\begin{theorem}
	\label{thm:CatC-FA}
	For integers $m,a, \ell$ as in Theorem \ref{thm:CatC-F}, the arrangement $\cc\Ccal^0_\ell(m,a)$ is flag-accurate. 	
In particular, the extended Catalan arrangement $ \Cat^m (C_\ell)$ of type $C_\ell$ (Definition \ref{Def_IdealShiCatalan}) equal to $\Ccal^0_\ell(m,m)$ has  flag-accurate cone with exponents 
$$\exp\left(\cc\Cat^m (C_\ell)\right) = (1, 2m\ell +2i-1 )_{i=1}^{\ell}.$$ 
\end{theorem}

\begin{proof}
The case   $\ell\le 2$ is already done in Case $1$ of the proof of Theorem \ref{thm:CatC-F}. 
Suppose $\ell\ge3$.
Recall the arrangement  $\cc\widetilde{\Ccal}_\ell(m,a,n)$ therein.
First observe that $\cc\widetilde{\Ccal}_\ell(m,a,n)$ is free with 
$$\exp \left( \cc\widetilde{\Ccal}_\ell(m,a,n) \right) =\exp\left(\cc\Ccal^0_\ell(m+na,a)\right) = (1,  2m+2a (\ell+n-1) +2i-1 )_{i=1}^{\ell}.$$
This is because $\cc\widetilde{\Ccal}_\ell(m,a,n)$ and $\cc\Ccal^0_\ell(m+na,a)$ share the  Ziegler restriction onto $H_\infty $, and the former is  locally free in codimension three along $H_\infty$, by an argument  similar to the one used earlier to prove the local freeness of $\textbf{c} \widetilde\Ccal_{\ell-1}(m, a, 1)$ in Theorem \ref{thm:CatC-F} (here the freeness of $\textbf{c} \widetilde\Ccal_{2}(m, a, n)$ for $n \ge 0$ is crucial). 

Now define $\ell$ hyperplanes $H_1, \ldots, H_\ell \in \cc\Ccal^0_\ell(m,a)$ as follows:
$$H_\ell: 2x_\ell = -mz, \quad H_i: x_{i} - x_{i+1} = -az \quad (1 \le i \le \ell-1).$$
Set $X_n : = \bigcap_{j=n}^{\ell} H_j$ for $1 \le n \le \ell$. 
Then $X_1 \subseteq X_2 \subseteq \cdots \subseteq  X_\ell \subseteq V'=\RR^{\ell+1} $ and $\dim_{V'}(X_n) =n$.
Moreover,  for each $1 \le n \le \ell$, one can show that 
$$\left(\cc\Ccal^0_\ell(m,a)\right)^{X_n} =  \cc\widetilde{\Ccal}_{n-1}(m,a,\ell-n+1).$$ 
To see this notice that 
$$\left(\cc\widetilde{\Ccal}_\ell(m,a,n)\right)^K = \cc\widetilde{\Ccal}_{\ell-1}(m,a,n+1),$$
where $K \in \cc\widetilde{\Ccal}_\ell(m,a,n)$ denotes the hyperplane $2x_\ell = -(m+2na)z$.

Thus  for each $1 \le n \le \ell$, $\left(\cc\Ccal^0_\ell(m,a)\right)^{X_n}$ is free with 
$$\exp \left((\cc\Ccal^0_\ell(m,a))^{X_n} \right) = \exp \left( \textbf{c}\Ccal^{0}_{n-1}(m+(\ell-n+1)a, a) \right)= (1,  2m+2a\ell-2a +2i-1 )_{i=1}^{n-1}.$$
Hence  $(X_1, X_2 , \ldots, X_\ell, V' )$  is a witness for the flag-accuracy of $\cc\Ccal^0_\ell(m,a)$.
\end{proof}

Unfortunately, we are unable to show the flag-accuracy of the extended Catalan arrangement of type $D$. 
We propose a potential approach.
 \begin{conjecture}
 \label{thm:CatD-F}
Let $a \ge 0$,  $\ell \ge 1$ and $0 \le r \le \ell$. 
Let $\Dcal^r_\ell(a)$ be the arrangement consisting of the hyperplanes
\begin{align*}
2x_i&=  [-2a , 0] \quad (1 \le    i \le  r), \\
x_{i}+ x_{j} &= [-3a , a]\quad  (1 \le  i < j \le  r), \\
x_{i}- x_{j} &= [-2a , 2a]\quad  (1 \le  i < j \le  r), \\
x_{i}\pm x_{j} &= [-2a , a]\quad  (1 \le  i \le r < j \le  \ell), \\
x_{i}\pm x_{j} &= [-a , a]\quad  (r+1 \le i< j \le  \ell).
\end{align*}
Then the cone $\cc\Dcal^r_\ell(a)$ is free with exponents 
$$ \exp\left(\cc\Dcal^r_\ell(a)\right) = (1,  \ell+r-1, 1,3,5,\ldots, 2\ell-3) +  (0,  (2a\ell+2ar-2a)^{\ell} ).$$ 
 \end{conjecture} 

 Note that $\Dcal^0_\ell(a)$ equals the extended Catalan arrangement $ \Cat^a (D_\ell)$  of type $D_\ell$ (see Definition \ref{Def_IdealShiCatalan} and see also the arrangement $\cc\Bb^\ell_\ell(0,a)$ in Theorem \ref{thm:CatB-RF}), and $\Dcal^\ell_\ell(a)$ is affinely equivalent to  the arrangement  $\Ccal^0_\ell(a,2a)$ from Theorem \ref{thm:CatC-F} via $x_i \mapsto x_i -\frac{a}2$. 
 Thus the cases $r=0$ and $r=\ell$ are already done.
 Note also that if $0 \le r \le \ell-2$, then
 the restriction of $\Dcal^r_\ell(a)$ to the  hyperplane $x_{r+1}-x_{r+2} = a$ is affinely equivalent to $\Dcal^{r+1}_{\ell-1}(a)$. 
 When $r =\ell-1$,  the restriction of $\Dcal^{\ell-1}_\ell(a)$ to the hyperplane $x_{\ell-1}-x_{\ell} = a$ is affinely equivalent (via $x_i \mapsto x_{i+1} -\frac{a}2$ $(1 \le i \le \ell-2)$, $x_\ell \mapsto x_1-\frac{a}2$) to the  arrangement  $\Fc_{\ell-1}(a,0)$ from Conjecture  \ref{thm:CatD-odd-FA} below.
 
 \begin{conjecture}
 \label{thm:CatD-odd-FA}
Let $a, n \ge 0$ and $\ell \ge 1$. 
Let $\Fc_\ell(a,n)$ be the arrangement consisting of the hyperplanes
\begin{align*}
2x_1&= [-(4n+3)a , a],\\
2x_i&=[-a , a]  \quad (2 \le  i \leq \ell), \\
x_{i}\pm x_{j} &= [-2a , 2a]\quad  (2 \leq i < j \leq \ell), \\
x_{1} \pm x_{i} &= [-(2n+3)a , 2a]\quad  (2 \leq i  \leq \ell).
\end{align*}
Then the cone $\cc\Fc_\ell(a,n)$ is flag-accurate with exponents
 $$\exp\left(\cc\Fc_\ell(a,n)\right) = (1,  4a (\ell+n) +2i-1 )_{i=1}^{\ell}.$$
\end{conjecture} 
 
 Note that the flag-accuracy of $\cc\Fc_\ell(a,n)$ follows from its freeness and exponents. 
 Indeed, assume that  $\Ac : = \cc\Fc_\ell(a,n)$ is free with exponents $ (1,  4a (\ell+n) +2i-1 )_{i=1}^{\ell}$ for any $a \ge 1$,   $n\ge 0$,  $\ell \ge 1$. 
 Let   $H \in \Ac$ denote the hyperplane $x_{1}-x_2 = 2az$. 
Then the restriction $\Ac^{H}$ is linearly equivalent to $\cc\Fc_{\ell-1}(a,n+1)$. 
Thus $\Ac^{H}$ is free  with exponents $ (1,  4a (\ell+n) +2i-1 )_{i=1}^{\ell-1}$. 
The strategy is now to use induction on $\ell$ and apply Lemma \ref{lem:flag-accuracy} to deduce the flag-accuracy of $\cc\Fc_\ell(a,n)$.

\begin{remark}
	\label{rem:Conj->CatD}
	We have verified Conjectures \ref{thm:CatD-F} and \ref{thm:CatD-odd-FA} by computer in the following few cases:  for
	 $(\ell, a)= (2,1), (2,2), (3,1), (3,2), (4,1)$ and $(4,2)$ and all $0 \le r \le \ell$ in 
	 the first instance and all values for $n \le 2$ in the second.
	
	We remark that if these conjectures are true, then we may conclude that the extended Catalan arrangement  $\Cat^m (D_\ell)$ of type $D_\ell$ for $m \ge 0$ and  $\ell \ge 2$ has  flag-accurate cone with exponents 
	$$
	\exp\left(\cc\Cat^m (D_\ell)\right) = (1,  \ell-1, 1,3,5,\ldots, 2\ell-3) +  (0,  (2m\ell -2m)^{\ell} ).$$ 

Let us verify this claim. Let $\Ac: = \cc\Cat^m (D_\ell)$. 
Recall that  $\cc H_\alpha^j = \ker(\alpha-jz)$ for $\alpha \in  \Phi$ and $j \in \ZZ$ and  
also that $ \alpha_i=x_i-x_{i+1}\, (1 \le i \le \ell-1)$, and $\alpha_{\ell}=x_{\ell-1}+x_{\ell}$.  

When $\ell=2p$, define $p$ subspaces $X_1, \ldots, X_{p} \in L(\Ac)$ as follows:
$$X_k : = \bigcap_{j=1}^{k} \cc H^m_{\alpha_{2j-1}} \mbox{ for }   1 \le k \le p.$$	
One may show that for every $1 \le k \le p$
$$\Ac^{X_{k}} = \cc\Dcal^k_{\ell-k}(m).$$
If Conjecture \ref{thm:CatD-F} is true, then $\Ac^{X_{k}}$ is free with exponents $$\exp\left(\Ac^{X_{k}}\right) = (1,  \ell-1, 1,3,5,\ldots, 2\ell-2k-3) +  (0,  (2m\ell -2m)^{\ell} ).$$ 
As noted before, $\Dcal^p_p(m)=\Ccal^0_p(m,2m)$ (cf.~Theorem \ref{thm:CatC-F}). 
Moreover, by Theorem \ref{thm:CatC-FA}, $\Ac^{X_{p}} = \cc\Ccal^0_p(m,2m)$ is flag-accurate. 
Thus by  Lemma \ref{lem:flag-accuracy}, $\Ac$ is flag-accurate.

When $\ell=2p+1$, define $p+1$ subspaces $X_1, \ldots, X_{p+1} \in L(\Ac)$ as follows:
$$X_k : = 
\begin{cases}
\bigcap_{j=1}^{k} \cc H^m_{\alpha_{2j}} &  \mbox{ for }   1 \le k \le p ,\\
X_p \cap \cc H^m_{\alpha_{1}} & \mbox{ for }   k = p+1.
\end{cases}
$$	
Using an argument similar to the one in the previous case and the notation from Conjecture  \ref{thm:CatD-odd-FA}, one can show that
$$\Ac^{X_{k}} = \cc\Dcal^k_{\ell-k}(m)  \quad\text{  for } 1 \le k \le p, \text{ and }\quad \Ac^{X_{p+1}} =\cc\Fc_{p}(m,0). 
$$
If Conjecture \ref{thm:CatD-odd-FA} is true, then $\Ac^{X_{p+1}}  $ is flag-accurate with exponents
 $ (1,  4mp +2i-1 )_{i=1}^{p}.$
Thus by  Lemma \ref{lem:flag-accuracy}, $\Ac$ is flag-accurate. 
\end{remark}

\begin{remark}
	\label{rem:exceptional-Cat}
	The extended Catalan arrangement  of type $G_2$ is clearly flag-accurate by Theorem \ref{thm:idealshi} and Lemma \ref{lem:flag-accuracy}. 
	For the other exceptional types $F_4, E_6, E_7$, and $E_8$, our computer could only produce a result in the case of $F_4$ which shows that the cone $\cc\Cat^1 (F_4)$ is indeed flag-accurate. 
	\end{remark}

 The discussion above motivates 
 Conjecture \ref{conj:Cat}. 
 
\section{Flag-accurate graphic arrangements}
\label{sect:graph}

In this section we examine flag-accuracy among free graphic arrangements.
For basics on the latter, we refer to \cite[Sec.~2.4]{OrTer92_Arr}.

Let  $\KK$ be an arbitrary field. 
 Let $G = (V_G,E_G)$ be a simple graph (i.e., no loops and no multiple edges) with vertex set $V_G = \{v_{1}, \dots, v_{\ell}\}$ and edge set $E_G$.
 \begin{definition}
The  \textit{graphic arrangement} $\Ac_G$ in $\Bbb K^\ell$ is defined by
$$\Ac_G:= \{H_{ij}: = \ker(x_i - x_j) \mid \{v_i,v_j\} \in E_G \}.$$
\end{definition}
A simple graph $G$ is  \textit{chordal} if it does not  contain an induced cycle $C_n$ of length $n > 3$, i.e.~if $G$ is \emph{$C_{n}$-free} for $n >3$ for short.
The freeness of graphic arrangements is characterized by chordality:

\begin{theorem}[\cite{St72}, {\cite[Thm.~ 3.3]{ER94_FreeAB}}]
\label{thm:free-chordal}
The graphic arrangement $\Ac_G$ is free if and only if $G$ is chordal.
\end{theorem}

\begin{figure}[ht!b]
	\def\sc {1.0}
	\begin{tikzpicture}[scale=\sc]
		% Vertices 1--10
		\node (10) at (0.0,0.0) {\tiny $10$};
		\draw[black] (10) circle[radius=6.25pt];
		\node (1) at (0.0,-2.) {\tiny $1$};
		\draw[black] (1) circle[radius=6.25pt];
		\node (2) at (1.7320508075688774,-1.) {\tiny $2$};
		\draw[black] (2) circle[radius=6.25pt];
		\node (3) at (1.7320508075688774,1.) {\tiny $3$};
		\draw[black] (3) circle[radius=6.25pt];
		\node (4) at (0.0,2.) {\tiny $4$};
		\draw[black] (4) circle[radius=6.25pt];
		\node (5) at (-1.7320508075688774,1.) {\tiny $5$};
		\draw[black] (5) circle[radius=6.25pt];
		\node (6) at (-1.7320508075688774,-1.) {\tiny $6$};
		\draw[black] (6) circle[radius=6.25pt];
		\node (7) at ($0.35*(2)-0.35*(10)$) {\tiny $7$};
		\draw[black] (7) circle[radius=6.25pt];
		\node (8) at ($0.35*(4)-0.35*(10)$) {\tiny $8$};
		\draw[black] (8) circle[radius=6.25pt];
		\node (9) at ($0.35*(6)-0.35*(10)$) {\tiny $9$};
		\draw[black] (9) circle[radius=6.25pt];
		
		% Edges #=18
		\draw[shorten <=6.25pt, shorten >=6.25pt] ($(1)$) -- ($(2)$);
		\draw[shorten <=6.25pt, shorten >=6.25pt] ($(2)$) -- ($(3)$);
		\draw[shorten <=6.25pt, shorten >=6.25pt] ($(3)$) -- ($(4)$);
		\draw[shorten <=6.25pt, shorten >=6.25pt] ($(4)$) -- ($(5)$);
		\draw[shorten <=6.25pt, shorten >=6.25pt] ($(5)$) -- ($(6)$);
		\draw[shorten <=6.25pt, shorten >=6.25pt] ($(6)$) -- ($(1)$);
		\draw[shorten <=6.25pt, shorten >=6.25pt] ($(1)$) -- ($(10)$);
		\draw[shorten <=6.25pt, shorten >=6.25pt] ($(3)$) -- ($(10)$);
		\draw[shorten <=6.25pt, shorten >=6.25pt] ($(5)$) -- ($(10)$);
		\draw[shorten <=6.25pt, shorten >=6.25pt] ($(1)$) -- ($(3)$);
		\draw[shorten <=6.25pt, shorten >=6.25pt] ($(3)$) -- ($(5)$);
		\draw[shorten <=6.25pt, shorten >=6.25pt] ($(5)$) -- ($(1)$);
		
		\draw[shorten <=6.25pt, shorten >=6.25pt] ($(3)$) -- ($(8)$);
		\draw[shorten <=6.25pt, shorten >=6.25pt] ($(8)$) -- ($(10)$);
		\draw[shorten <=6.25pt, shorten >=6.25pt] ($(5)$) -- ($(9)$);
		\draw[shorten <=6.25pt, shorten >=6.25pt] ($(9)$) -- ($(10)$);
		\draw[shorten <=6.25pt, shorten >=6.25pt] ($(1)$) -- ($(7)$);
		\draw[shorten <=6.25pt, shorten >=6.25pt] ($(7)$) -- ($(10)$);
		
		\node (11) at (2.3,0.0) {\tiny $v$};
		\draw[black] (11) circle[radius=6.25pt];
		\draw[shorten <=6.25pt, shorten >=6.25pt, dashed] ($(1)$) -- ($(11)$);
		\draw[shorten <=6.25pt, shorten >=6.25pt, dashed] ($(2)$) -- ($(11)$);
		\draw[shorten <=6.25pt, shorten >=6.25pt, dashed] ($(3)$) -- ($(11)$);
		\draw[shorten <=6.25pt, shorten >=6.25pt, dashed] ($(10)$) -- ($(11)$);
		
	\end{tikzpicture}    
	\caption{A chordal graph $G$ giving rise to a graphic arrangement which is free, but not accurate,
		and an extension $G'$ by one vertex $v$ resulting in a graphic arrangement which is ind-flag-accurate.}
	\label{fig:AGexLocRestr}
\end{figure}
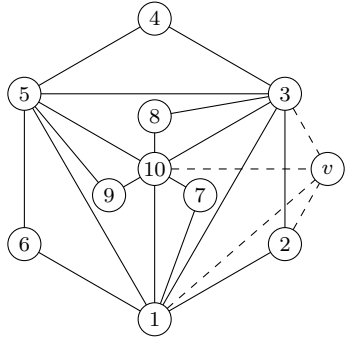

We begin with an example from \cite{mueckschroehrle:accurate}.
\begin{example}
	\label{ex:FA-Rest-Loc}
	Let $G$ be the chordal graph shown in Figure \ref{fig:AGexLocRestr}.
	It was shown in \cite[Ex.~5.7]{mueckschroehrle:accurate} that $\Ac_G$ is not accurate and so it is not flag-accurate.
	Extending $G$ by one additional vertex $v$ and the
	edges indicated by the dashed lines in Figure \ref{fig:AGexLocRestr} yields a new chordal graph $G'$.
	The corresponding graphic arrangement $\Ac_{G'}$ is free with exponents 
	$\exp(\Ac_{G'}) = (0,1,2^6,3^3)$.
	It is not hard to see, by contracting the appropriate edges, 
	that $\Ac_{G'}$ is flag-accurate. But for $H' = \ker(x_1-x_v) \in \Ac_{G'}$ (the hyperplane corresponding to the new edge $(1,v)$ of $G'$), 
	we have $\Ac_{G'}^{H'} = \Ac_G$. Moreover,  for $X = \bigcap_{1\leq i < j \leq 10} \ker(x_i-x_j)$ we also obtain $\left(\Ac_{G'}\right)_X \cong \Ac_G$.
	This shows that in general, flag-accuracy is neither inherited by restrictions, nor by localizations.
\end{example}

In \cite{TranTsujie22_MATfreeG}, Tsujie and the third author gave a  characterization of MAT-free graphic arrangements by means of  \emph{strongly chordal graphs}  -- the graphs that are chordal and \emph{$n$-sun-free} for $n\ge3$.
Recall that an  \emph{$n$-sun $ S_{n} $} ($n \ge 3$) is a (chordal) graph with vertex set $V_{S_{n}} = [n] \cup \{v_{1}, \dots, v_{n}\} $ and edge set 
\begin{align*}
E_{S_{n}} = \left\{\{i,j\} \mid 1 \leq i < j \leq n\right\}\ \cup \ \left\{\{v_{i}, j\} \mid 1 \leq i \leq n, j \in \{i, i+1\}\right\}, 
\end{align*}
where the vertices $1$ and $ n+1$ are identified; 
 see Figure \ref{fig:Qell} for illustrations of $S_3$ and $S_4$.
 
 \begin{theorem}
 	[{\cite[Thm.~2.10]{TranTsujie22_MATfreeG}}]
\label{thm:MAT-free-strong-chordal}
The graphic arrangement $\Ac_G$ is MAT-free if and only if $G$ is strongly chordal.
\end{theorem}

In the case of graphic arrangements, the concepts of supersolvability, inductive freeness, divisional freeness, almost accuracy and freeness are essentially equivalent. 
In particular, ind-flag-accuracy coincides with  flag-accuracy. 
However, a characterization of (flag-)accurate graphic arrangements is unknown.
We present in this section some classes of flag-accurate graphic arrangements.

Next, we first fix some notation which is used throughout the section. 
For a positive integer $n$ let $[n]:=\{1,2,\ldots,n\}$. 
For an edge $e \in E_G$ denote by  $H_e$  the hyperplane that is defined by $e$.
For a subset  $F \subseteq E_G$ let $X_F := \bigcap_{e \in F} H_e\in L(\Ac_G)$.
 
The following special \emph{$\Qc$-family} of graphs is of particular interest. For, the set of graphic arrangements stemming from $\Qc$ is closed under restrictions to a hyperplane, see Remark \ref{rem:AQell}.
 
\begin{definition}[$\Qc$-family]
	\label{def:AQell}
	For $\ell\ge 2$, let $K_\ell$  be the complete (undirected) graph  on $[\ell]$.
	 Let $M = (m_{ij} \mid 1 \leq i < j \leq \ell) \in \ZZ^{\ell(\ell-1)/2}_{\ge0}$ be a tuple of non-negative integers of cardinality $\ell(\ell-1)/2$ where the element $m_{ij} \in M$ corresponds to the edge $\{i,j\}$ of $K_\ell$. 
	 By $\Qc$ we denote the following family  of chordal graphs whose elements $Q_\ell=Q_\ell(M) \in \Qc$ are defined as follows: $Q_\ell$ consists of an ``inner" complete graph $K_\ell$, and to each edge $\{i,j\}$ of $K_\ell$ we attach $m_{ij} \ge 0$ many ``outer" triangles ($3$-cycles).  
	 (We may also think of $Q_\ell(M)$ as an edge-weighted graph with weight $m_{ij}$ in each edge $\{i,j\}$ of the complete graph $K_\ell$.)
	 See Figure \ref{fig:Qell} for an example.	 
	 \end{definition}
	 
	 \begin{example}
	\label{ex:n-sun} 
	The $n$-sun is a member in the $\Qc$-family given by $S_n = Q_n(M)$ where $m_{ij} = 1$ if $j=i+1$ for $1 \leq i \leq n$ and $m_{ij} = 0$ otherwise. 
\end{example}

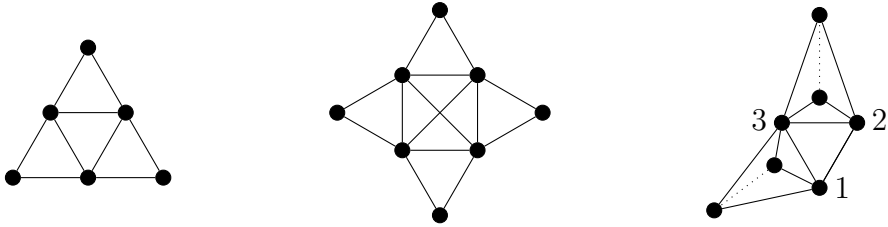
\begin{figure}[htbp]
\centering
\begin{subfigure}{.3\textwidth}
  \centering
\begin{tikzpicture}
\draw (0,0) node[v](1){};
\draw (1,0) node[v](2){};
\draw (2,0) node[v](3){};
\draw (0.5,0.865) node[v](4){};
\draw (1.5,0.865) node[v](5){};
\draw (1,1.73) node[v](6){};
\draw (6)--(4)--(1)--(2)--(3)--(5)--(6);
\draw (4)--(2)--(5)--(4);
\end{tikzpicture}
%  \caption*{$3$-sun}  
%  \label{fig:3sun}
\end{subfigure}%
\begin{subfigure}{.3\textwidth}
  \centering
\begin{tikzpicture}
\draw (0,0) node[v](x1){};
\draw (1,0) node[v](x2){};
\draw (1,1) node[v](x3){};
\draw (0,1) node[v](x4){};
\draw (0.5,-0.865) node[v](y1){};
\draw (1.865,0.5) node[v](y2){};
\draw (0.5,1.865) node[v](y3){};
\draw (-0.865,0.5) node[v](y4){};
\draw (x1)--(x2)--(x3)--(x4)--(x1)--cycle;
\draw (x1)--(x3);
\draw (x2)--(x4);
\draw (x1)--(y1)--(x2)--(y2)--(x3)--(y3)--(x4)--(y4)--(x1)--cycle;
\end{tikzpicture}
%  \caption*{$4$-sun}  
%  \label{fig:4sun}
\end{subfigure}
\begin{subfigure}{.3\textwidth}
  \centering
\begin{tikzpicture}
    \node (v1) at (1.3,0) {1};
        \node (v2) at (1.8,.9) {2};
            \node (v2) at (.2,.9) {3};
\draw (0.4,0.3) node[v](0){};
\draw (-0.4,-0.3) node[v](1){};
\draw (1,0) node[v](2){};
%\draw (2,0) node[v](3){};
\draw (0.5,0.865) node[v](4){};
\draw (1.5,0.865) node[v](5){};
\draw (1,1.2) node[v](6){};
\draw (1,2.3) node[v](7){};
\draw (6)--(4)--(1)--(2)--(5)--(6);
\draw (4)--(2)--(5)--(4);
\draw (4)--(7)--(5);
\draw (4)--(0)--(2);
\draw[dotted] (0)--(1);
\draw[dotted] (6)--(7);
\end{tikzpicture}
%  \caption*{$Q_3(M)$}  
%  \label{fig:Q3}
\end{subfigure}%
\caption{From left to  right : $S_3$, $S_4$ and $Q_3(M)$ with $m_{12}=0, m_{13}>0, m_{23}>0$.}
\label{fig:Qell}
\end{figure}

	 \begin{remark}
	\label{rem:AQell}
	It is not hard to see that the graphic arrangement $\Ac_{Q_\ell}$ defined by the graph $Q_\ell$ from Definition \ref{def:AQell} is free with exponents 
	$$\exp(\Ac_{Q_\ell}) = (2^{|m|},\ell-1,\ell-2,\ldots,2,1,0),$$ where  $|m|: = \sum_{m_{ij} \in M} m_{ij}$.
	We observe that regarding their exponents, up to symmetry there are three types of restrictions of $\Ac_{Q_\ell}$ to a hyperplane, or equivalently, three types of edge-contractions of $Q_\ell$. 
	Below  we give an edge representative of each type:
\begin{enumerate}[(I)]
\item $e \in E_{K_\ell}$ with $m_e =0$, i.e., $e$ is an edge of the inner complete graph $K_\ell$ with no outer triangles. 
In this case, the contraction $Q_\ell/e$ of $e$ on $Q_\ell$ results in a graph $Q_{\ell-1}(M')$ in the $\Qc$-family  for a tuple $M'$ of cardinality $(\ell-1)(\ell-2)/2$. 
Furthermore, the restriction of $\Ac_{Q_\ell}$ to the hyperplane $H_e$ defined by the edge $e$ is free and for $\ell \ge 2$,  
$$\exp\left(\Ac^{H_e}_{Q_\ell}\right) =\exp(\Ac_{Q_\ell})\setminus \{\ell-1\}.$$ 

\item  $e \in E_{K_\ell}$  with $m_e >0$. 
In this case, we have  
$$\exp\left(\Ac^{H_e}_{Q_\ell}\right) = (2^{|m|-m_e},1^{m_e},\ell-2,\ldots,2, 1,0).$$
In particular, $\exp\left(\Ac^{H_e}_{Q_\ell}\right)\ne \exp(\Ac_{Q_\ell})\setminus \{\ell-1\}$ if $\ell \ge 3$. 

\item $e \in E_{Q_\ell} \setminus E_{K_\ell}$. 
In this case, the contraction $Q_\ell/e$ is of the form $Q_{\ell-1}(M') \in \Qc$ with $M' = (M\setminus\{m_e\} )\cup\{m_e-1\}$. 
Thus, 
$$\exp\left(\Ac^{H_e}_{Q_\ell}\right) = (2^{|m|-1},\ell-1,\ell-2,\ldots,2, 1,0).$$ 
In particular, $\exp\left(\Ac^{H_e}_{Q_\ell}\right) \ne \exp(\Ac_{Q_\ell}) \setminus \{\ell-1\}$ if $\ell \ge 4$. 
\end{enumerate}	
	By the observation above, it is easily seen that  if $\ell\le 3$ then $\Ac_{Q_\ell(M)}$ is flag-accurate for any finite tuple $M$ of non-negative integers. 
	To see this, contract the edges of type (III) successively. 
\end{remark}

The following theorem gives a sufficient condition for the arrangements in the  $\Qc$-family to be flag-accurate.
\begin{theorem} 
\label{thm:Qfamily}
Let $M = (m_{ij}) \in \ZZ^{\ell(\ell-1)/2}_{\ge0}$ be a tuple of non-negative integers  such that $m_{ij} = 0$ for $\{i,j\} \notin \{\{1,2\},\ldots, \{\ell-1,\ell\},\{\ell,1\} \}$ (i.e.~the edges not in the great circle of $K_\ell$ have weight $0$). 
Then the graphic arrangement $\Ac_{Q_\ell(M)}$ is flag-accurate.
\end{theorem}

\begin{proof}
We argue by induction on $\ell$. 
The case $\ell \le 3$ is clear from Remark \ref{rem:AQell}. 
When $\ell =4$ we may contract an edge of type (I) (which reduces the problem to the case $\ell=3$) and then apply Lemma \ref{lem:flag-accuracy}. 
Suppose $\ell \ge 5$. 
Note that $m_{13} = m_{24} = 0$. 
In view of Remark \ref{rem:AQell}, we may compute $\exp\left(\Ac^{H_{13}}_{Q_\ell}\right) = \exp(\Ac_{Q_\ell})\setminus \{\ell-1\}$ and $\exp\left(\Ac^{H_{13} \cap H_{24}}_{Q_\ell}\right)=\exp(\Ac_{Q_\ell})\setminus \{\ell-1, \ell-2\}$. 
Moreover, the contraction $(Q_\ell/\{1,3\})/\{2,4\} = Q_{\ell-2}(M'')$ is a graph in the $\Qc$-family where $M''$ satisfies  the assumption of the theorem. 
By our induction hypothesis, $Q_{\ell-2}(M'')$ is flag-accurate.  
Again, Lemma \ref{lem:flag-accuracy} completes the proof.
\end{proof}

The following is an immediate consequence of Theorem \ref{thm:Qfamily}.
\begin{corollary}
	\label{cor:n-sun}
$\Ac_{S_n}$ is flag-accurate for $n \ge 3$.
\end{corollary}

An example of a non-accurate arrangement is given in
Example \ref{ex:FA-Rest-Loc}.
 %\cite[Ex.~ 5.7]{mueckschroehrle:accurate}. 
The following proposition extends this example.

\begin{proposition}
	\label{prop:uni-sun}
 Let $\ell \ge 4$ and $M = (m_{ij}) \in  \ZZ^{\ell(\ell-1)/2}_{>0}$ be a tuple of positive integers. 
Then the graphic arrangement $\Ac_{Q_\ell(M)}$ is not (flag-)accurate.
\end{proposition}

\begin{proof}
There is no edge of type (I).
\end{proof}

It is obvious that if an arrangement is $k$-coaccurate, then it is $(k+1)$-coaccurate (see Definition \ref{Def_TF}). 
Now we show that for any given $ k \ge 1$ there exists a graphic arrangement that is $(k+1)$-coaccurate but not $k$-coaccurate.
 
\begin{definition}
	\label{def:AQ4-ext}
		Let $ k \ge 1$. 
 Let  $G$ be the graph defined as follows: $G$ consists of the graph $Q_4(M)$ in the $\Qc$-family with $m_{ij}=k$ for all $\{i,j\} \in E_{K_4}$ and we draw $k$ many $4$-complete graphs $D_i$ with vertex set $\{1,2,3,v_i\}$ for $1 \le i \le k$.  
 See Figure \ref{fig:k=1} for the case $k=1$.	
	 \end{definition}

	 	\begin{figure}[htbp]
\centering
\begin{tikzpicture}
\draw (0,0) node[v](x1){};
\draw (1,0) node[v](x2){};
\draw (1,1) node[v](x3){};
\draw (0,1) node[v](x4){};
\draw (0.5,-0.865) node[v](y1){};
\draw (1.865,0.5) node[v](y2){};
\draw (0.5,1.865) node[v](y3){};
\draw (-0.865,0.5) node[v](y4){};

\draw (0.5,0.2) node[v](z1){};
\draw (0.5,0.8) node[v](z2){};
\draw (-0.5,1.7) node[v](v1){};

\draw (x1)--(x2)--(x3)--(x4)--(x1)--cycle;
\draw (x1)--(x3);
\draw (x2)--(x4);
\draw (x1)--(y1)--(x2)--(y2)--(x3)--(y3)--(x4)--(y4)--(x1)--cycle;
\draw (x1)--(z1)--(x3);
\draw (x1)--(v1)--(x3);
\draw (x4)--(v1);
\draw (x2)--(z2)--(x4);

\end{tikzpicture}
\caption{A graph whose corresponding graphic arrangement is accurate but not flag-accurate.}
\label{fig:k=1}
\end{figure}
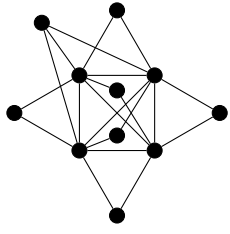

\begin{theorem} 
\label{thm:Q4-ext}
Let $ k \ge 1$. 
If $G$ is the graph from Definition \ref{def:AQ4-ext}, then the graphic arrangement $\Ac_G$  is $(k+1)$-coaccurate but not $k$-coaccurate. 
\end{theorem}
\begin{proof}
It is easily seen that   $\Ac_G$ is free with $\exp(\Ac_G) = (3^{k+1},2^{6k+1},1,0)$.	
We observe that regarding exponents up to symmetry there are three types of edge-contractions of $G$ whose edge representatives are given below:
\begin{enumerate}[(i)]
\item $e$ is an edge of a $4$-complete graph $D_i$ but not an edge of $Q_4(M)$, i.e., $e=\{v_i, j\}$ for some $1 \le i \le k$ and $1 \le j \le 3$. 
In this case, contracting the edge $e$ of $G$ simply removes the edges $\{v_i, j\}$ for $1 \le j \le 3$ which yields
 $\exp\left(\Ac^{H_e}_G\right) = \exp(\Ac_G)\setminus \{3\}$. 

\item  $e$  is an edge of the inner complete graph $K_4$ of $Q_4(M)$ (similar to type (II) in Remark \ref{rem:AQell}).  
In this case, the contraction $G/e$ produces $k+1$ many exponents $1$ in $\exp(\Ac^{H_e}_G)$. 
To study the contraction more explicitly, we distinguish two sub-types: 
(iia): $e \in  \{\{1,4\}, \{2,4\},\{3,4\} \}$ and (iib): $f \in  \{\{1,2\}, \{2,3\},\{3,1\} \}$. 
Then  $\exp\left(\Ac^{H_e}_G\right) =(3^{k},2^{5k+1},1^{k+1},0)$ and $\exp\left(\Ac^{H_f}_G\right) =(2^{6k+1},1^{k+1},0)$. 

\item $e$  is an edge of an outer triangle in $Q_4(M)$ but not an edge of $K_4$ (similar to type (III) in Remark \ref{rem:AQell}). 
In this case,  $\exp\left(\Ac^{H_e}_G\right) = \exp(\Ac_G)\setminus \{2\}$. 
\end{enumerate}	

\begin{claim}
	\label{cla:typei}
 Let $1 \le p \le k$. 
 For any  $F \subseteq E_G$ with $|F| =p$, we have $\exp\left(\Ac_G^{X_F}\right) = (3^{k-p+1},2^{6k+1},1,0)$, where $X_F = \bigcap_{e \in F} H_e$ if and only if all edges in $F$ are of type (i).
\end{claim}

\begin{proof}[Proof of Claim  \ref{cla:typei}]
The reverse implication is clear since contracting $p$ edges of type (i) in any order removes $p$ many exponents $3$ of $\Ac_G$. 

Write $F=\{e_1,e_2,\ldots,e_p\}$ so that $e_1<e_2<\cdots<e_p$ is the contracting order in $G/F$, i.e., $G/F = (((G/e_1)/e_2)\cdots )/e_p$. 
Suppose that $F$ contains an edge of type (ii) and let $e_q$ for $q \le p$ be the first edge of this type in the contracting order above. 
 Denote $F':= \{e_1<\cdots<e_q\} \subseteq F$.  
  Note that $G/F = G'/(F\setminus F')$  and  $ \Ac_{G'} = \Ac_G^{X_{F'}}$ where $G'=G/F'$.
 Since there are at most $q-1$ many edges of type (iii) in $F'$, we must have $1^{k+2-q} \in \exp(\Ac_{G'})$. 
 In order to achieve the desired exponents of $\Ac_G^{X_F}$, we have to remove $k-q+1$ many exponents $1$ from $\exp(\Ac_{G'})$ but this is impossible since $|F\setminus F'| = p-q < k-q+1$. 
Hence $F$ has no edge of type (ii). 
 It is thus easily seen that $F$ has no edge of type (iii) either. 
 This completes the proof of  the forward implication. 
\end{proof}

We return to the proof of Theorem \ref{thm:Q4-ext}.
First we show that the graphic arrangement $\Ac_G$  is not $k$-coaccurate. 
Suppose to the contrary that $\Ac_G$  is $k$-coaccurate. 
Then by definition there exists  $X_k \in L(\Ac_G)$ of codimension $k$ such that $\Ac^{X_k}_G$ is flag-accurate with $\exp\left(\Ac_G^{X_k}\right) = (3,2^{6k+1},1,0)$. 
By Claim \ref{cla:typei}, the underlying graph of $\Ac_G^{X_k}$  must have the form $Q_4(M')$ with $m_e = k$ for all $m_e \in M'$. 
However, this arrangement is not (flag-)accurate, by Proposition \ref{prop:uni-sun}, a contradiction.

Finally, we show that $\Ac_G$  is $(k+1)$-coaccurate. 
By Claim \ref{cla:typei},  for each $1 \le p \le k$ there exists $X_p \in L(\Ac_G)$ of codimension $p$ such that 
 $\exp\left(\Ac_G^{X_p}\right) = (3^{k-p+1},2^{6k+1},1,0)$. 
 It suffices to find $X_{k+1} \in L(\Ac_G)$ of codimension $k+1$ such that 
 $\exp\left(\Ac_G^{X_{k+1}}\right) = (2^{6k+1},1,0)$ and $\Ac_G^{X_{k+1}}$ is flag-accurate. 
This can be done as follows: Let $\{2,3, u_i\}$ for $1 \le i \le k$ be the set of the $k$ outer triangles based in $\{2,3\}$ (an edge of type (iib)), we first contract the edge $\{2,3\}$ of $G$ then contract the $k$ edges $\{3, u_i\}$ in any order. 
The resulting graph is of the form $Q_3(M'')$ where $M'' = (k, 2k,3k)$. 
The corresponding graphic arrangement has exponents $(2^{6k+1},1,0)$ and is  flag-accurate, by Remark \ref{rem:AQell}. 
\end{proof}

The following is a direct consequence of Theorem \ref{thm:Q4-ext} by taking $k=1$ (cf.~Figure \ref{fig:k=1}).
\begin{corollary}
	\label{coro:acc-nfa}
Accurate graphic arrangements need not be flag-accurate. 
\end{corollary}

\begin{remark}
	\label{rem:k-acc}
 Theorem \ref{thm:Q4-ext} implies that for a given $ k \ge 1$ there exists a graphic arrangement that is $(6k+3)$-accurate but not $(6k+4)$-accurate. 
 It would be interesting to know whether there exists an arrangement that is $k$-accurate but not $(k+1)$-accurate for \emph{any} $ k \ge 1$.
	\end{remark}

Next we show that the operation of adding a dominating vertex to a graph preserves $k$-coaccuracy. 
\begin{theorem} 
\label{thm:G+v}
  Let $G$ be a simple graph on $\ell$ vertices. 
  Denote by $G+v$ the graph obtained from $G$ by adding a new (dominating) vertex  $v\notin V_G$ adjacent to all the vertices of $G$, i.e., $V_{G+v} = V_G \cup \{v\}$ and $E_{G+v} = E_G \cup \{\, \{u,v\} \mid u \in V_G\}$. 
  Let $1 \le k \le \ell-1$. 
If $\Ac_G$  is   $k$-coaccurate, then so is $\Ac_{G+v}$. 
In particular, if $\Ac_G$  is (flag-)accurate, then so is $\Ac_{G+v}$. 
\end{theorem}

\begin{proof}
Suppose that $\Ac_G$  is   $k$-coaccurate. 
In particular, $\Ac_G$  is  free and we may assume  $\exp(\Ac_G) =(0, d_2, \ldots, d_\ell)_\le$. 
One can show that  $\Ac_{G+v}$  is also free (since any perfect elimination ordering of $G$ induces a perfect elimination ordering of $G+v$) and compute $\exp(\Ac_{G+v})=(0, 1, d_2+1, \ldots, d_\ell+1)_\le$. 

 Since $\Ac_G$  is   $k$-coaccurate, by definition there exist subsets $F_p \subseteq E_G$ for  $0 \le p \le \ell-1$ with  $\codim_{\KK^\ell}(X_{F_p}) =p$ such that  $\exp\left(\Ac_G^{X_{F_p}}\right) =(0, d_2, \ldots, d_{\ell-p})_{\le}$ and $F_k \subseteq F_{k+1} \subseteq  \cdots \subseteq F_{\ell-1}$. 

For any $F \subseteq E_G$,  it is easily seen that  $\codim_{\KK^\ell}(X_F) = \codim_{\KK^{\ell+1}}(X_F)$ and we have $(G/F) + v = (G+v)/F$.
Thus $\exp\left(\Ac_{G+v}^{X_{F_p}}\right) =(0, 1, d_2+1, \ldots, d_{\ell-p}+1)_{\le}$. 
This implies that $(X_{F_{\ell-1}} \subseteq \cdots \subseteq X_{F_1} \subseteq X_{F_0})$ is a witness for the accuracy of $\Ac_{G+v}$, and  $(X_{F_{\ell-1}} \subseteq \cdots \subseteq X_{F_{k}})$ is a witness for the flag-accuracy of $\Ac_{G+v}^{X_{F_k}} $.
Hence $\Ac_{G+v}$  is   $k$-coaccurate. 
\end{proof}

\begin{definition}
	\label{def:graph-P}
	Let $\Gc$ denote the smallest class of simple graphs satisfying the following four conditions:
		\begin{enumerate}[(1)]
		\item $K_1 \in \Gc$, i.e., any one-vertex graph is in $\Gc$.
		\item If $G,H \in \Gc$ then $G \sqcup H \in \Gc$, i.e., the disjoint union of two graphs in $\Gc$ is again a graph in $\Gc$.
		\item If $G,H \in \Gc$ then any new graph $Q$ formed by identifying a vertex in $G$ and a vertex in $H$ as a new vertex in $Q$ (often known as a \emph{vertex identification} of $G$ and $H$) is in $\Gc$. 
		\item If $G \in \Gc$ and $v\notin V_G$ is a new vertex then $G + v \in \Gc$, i.e., adding one dominating vertex to a graph in $\Gc$ gives a graph in $\Gc$.	
		\end{enumerate}
	 \end{definition}

	 	 \begin{example}
	\label{ex:TP} 
	\emph{Trivially perfect graphs} are the graphs that can be constructed by conditions (1), (2) and (4) in Definition \ref{def:graph-P}. Thus trivially perfect graphs are examples of the graphs in the class $\Gc$.
		\end{example}
	
	\begin{theorem} 
\label{thm:P-FA}
  If $G \in \Gc$ then  $\Ac_G$ is flag-accurate. 
  In particular, if $G$ is trivially perfect, then  $\Ac_G$ is flag-accurate. 
\end{theorem}

\begin{proof}
Conditions  (2) and (3) in Definition \ref{def:graph-P} both correspond to the operation of taking products of two (graphic) arrangements. 
Apply Remark \ref{rem_MATRest-flag}(v)  and Theorem \ref{thm:G+v}.
\end{proof}

We close this section with a comment on strong chordality and flag-accuracy. 
It is not hard to see that strong chordality is closed under each operation defining the class $\Gc$. 
Thus any graph in $\Gc$ is strongly chordal. However, the converse is not true (see Figure \ref{fig:SC-nP}).
Moreover, we have seen in Corollary \ref{cor:n-sun} that the $n$-suns (that are not strongly chordal) give rise to flag-accurate arrangements. 
We propose a conjecture based on these observations, which, in view of 
Theorem \ref{thm:MAT-free-strong-chordal}, is the graphic counterpart of Problem \ref{prob:MAT-FA}.

	\begin{conjecture} 
\label{conj:SC<FA}
  If $G$ is a strongly chordal graph, then  $\Ac_G$ is flag-accurate. 
\end{conjecture}	
	
	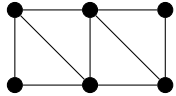
\begin{figure}[htbp]
\centering
\begin{tikzpicture}
\draw (0,0) node[v](1){};
\draw (1,0) node[v](2){};
\draw (2,0) node[v](3){};
\draw (0,1) node[v](4){};
\draw (1,1) node[v](5){};
\draw (2,1) node[v](6){};
\draw (1)--(2)--(3)--(6)--(5)--(4)--(1);
\draw (4)--(2)--(5)--(3);
\end{tikzpicture}
\caption{A strongly chordal graph that is not in the class $\Gc$.}
\label{fig:SC-nP}
\end{figure}

\section{Flag-accurate $\psi$-digraphic arrangements}
\label{sect:digraph-FA}
 
In this section we examine ind-flag-accuracy  in the class of \emph{$\psi$-digraphic} arrangements which was introduced by Abe, Tsujie and the last author \cite{AbeTranTsujie21_ShiIsh}.

We fix some notation throughout this section. 
Our base field  is $\RR$. 
By $ G = (V_{G}, E_{G}) $ we denote a directed graph or \emph{digraph}  on $ V_{G} = [\ell]$. 
A directed edge $(i,j) \in E_{G}$ is considered to be \emph{directed from $i$ to $j$}.
A  \emph{vertex-weighted digraph} is a pair  $(G,\psi)$ where $G$ is a digraph on $ [\ell]$ and a map $ \psi \colon [\ell] \to 2^{\mathbb{Z}} $, called a \emph{weight} on $G$. 
A weight   $ \psi$ is called an \emph{interval weight} if each image of $\psi$  is an integral interval, i.e., $ \psi(i) = [a_{i}, b_{i}]\subseteq \mathbb{Z}$ where $a_{i}\le b_{i}$ are integers for every $ i \in [\ell] $. 

We sometimes use the notation $(G,\psi(i)) $ for $(G,\psi) $ when we want to emphasize the precise evaluation  $\psi(i)$. 
In particular, if $\psi$ is a constant map with image $U$, we write $(G,U)$. 

\subsection{$\psi$-digraphic arrangements}
\label{subsect:digraph}
Next we define $ \psi $-digraphic arrangements and recall their basic properties following  \cite{AbeTranTsujie21_ShiIsh}.

\begin{definition}
Let $(G,\psi)$ be a vertex-weighted digraph.
The \emph{$ \psi $-digraphic arrangement} $ \Ac(G,\psi) $ in $ \mathbb{R}^{\ell} $ is defined by 
$$
\Ac(G,\psi) :=
\Cox(A_\ell) 
\cup \{ x_{i}-x_{j}=1  \mid (i,j) \in E_{G}\}
\cup \{  x_{i}=\psi(i) \mid 1 \leq i   \leq \ell\}, 
$$
where $\Cox(A_\ell) :=\{ x_{i}-x_{j}=0 \mid 1 \leq i < j \leq \ell\} $ is the  \emph{Coxeter (or Weyl) arrangement of type $A_{\ell-1}$} (see also Section \ref{SSec_WeylIdeal}). 
\end{definition}
Thus a $ \psi $-digraphic arrangement can be regarded as a deformation of a subarrangement of a Coxeter arrangement of type $B$.
 
The following digraphs play a crucial role in this  section.

\begin{definition}
\label{def:digraphs}
The \emph{transitive tournament} $  T_{[\ell]}$, \emph{complete digraph}
 $K_{[\ell]}$,  and \emph{edgeless digraph} $ \overline{K}_{[\ell]}$ on $ [\ell]$ are defined by 
\begin{align*}
E_{T_{[\ell]}} &:= \{(i,j) \mid 1 \leq i < j \leq \ell\},\\
E_{K_{[\ell]}} &:= \{(i,j) \mid i,j \in [\ell], \ i \neq j\},\\
E_{\overline{K}_{[\ell]}} &:= \emptyset.
\end{align*}
For simplicity we often use the notation $T_{\ell},  \overline{K}_{\ell}, K_{\ell}$\footnote{As we are only concerned with digraphs in this section, for simplicity of notation we continue to use $K_\ell$  for complete digraphs (previously used for complete undirected graphs in Sect.~\ref{sect:graph}).}
for $  T_{[\ell]} , \overline{K}_{[\ell]}, K_{[\ell]}$, respectively. 

For $1 \leq   k \leq \ell$, define the digraphs $ T_{\ell}^{k}$, $ K_{\ell}^{k}$ and $ \widehat K_{\ell}^{k}$ on $ [\ell] $ by 
\begin{align*} 
E_{T_{\ell}^{k}} & :=\{(i,j) \mid 1 \leq i < j \leq \ell - k + 1\}, \\
E_{K_{\ell}^{k}} & :=\{(i,j) \mid i,j \in [\ell - k + 1], \ i \neq j \}, \\
E_{\widehat K_{\ell}^{k}} & :=\{(i,j) \mid i,j \in [\ell - k], \ i \neq j \} \cup \{ (\ell - k + 1,i ) \mid i \in [\ell - k ] \}.
\end{align*}
\end{definition}

A \emph{simplicial vertex} in a simple undirected graph is a vertex whose  neighbors are mutually adjacent. 
The following is a counterpart of a simplicial vertex in a vertex-weighted digraph. 
\begin{definition}
\label{def:simplicial}
Let $(G,\psi)$ be a vertex-weighted digraph on $ [\ell] $. 
Let $v$ be a vertex in $G$ and let $X_v\in L(\mathbf{c}\Ac(G,\psi) )$ be the intersection of the following hyperplanes:
\begin{align*}
z&= 0, \\
x_{i}-x_{j} &= 0 \qquad (i,j\in [\ell]\setminus\{v\}), \\
x_{i}-x_{j} &= z \qquad ((i,j) \in E_{G}, i,j\in [\ell]\setminus\{v\}), \\
 x_{i} &= \psi(i)z\qquad (i\in [\ell]\setminus\{v\}). 
\end{align*}
The vertex $ v $ is said to be \emph{simplicial} in $ (G,\psi) $ if $X_v$ is a modular coatom of $ \mathbf{c}\Ac(G,\psi) $.
\end{definition}

Let  $ G\setminus v $ denote the subgraph obtained from $G$ by removing $ v $ and the edges incident on $v$.  
Thus 
$$\left( \mathbf{c}\Ac(G,\psi) \right)_{X_v}=\mathbf{c}\Ac\left(G\setminus v, \psi|_{[\ell]\setminus\{v\}}\right) \times \varnothing_{1}.$$ 

\begin{proposition}[{\cite[Prop.~3.12]{AbeTranTsujie21_ShiIsh}}]
 \label{prop:isolated}
 Let $(G,\psi)$ be a vertex-weighted digraph on $ [\ell] $ and $ v $ an isolated vertex of $ G $. 
If $ \psi(v) \subseteq \psi(i) $ for every $ i \in [\ell] $, then $ v $ is simplicial in $ (G,\psi) $. 
\end{proposition}

The following is immediate from Proposition \ref{prop:modular coatom}.

\begin{corollary}
\label{cor:simplicial}
 Let $(G,\psi)$ be a vertex-weighted digraph on $ [\ell] $ and $ v $ a simplicial  vertex of $(G,\psi)$. 
Then the following statements hold. 
\begin{itemize}
\item[(i)] The cone $ \mathbf{c}\Ac(G,\psi) $ is supersolvable (resp., (inductively) free) if and only if $ \mathbf{c}\Ac(G\setminus v,\psi|_{[\ell]\setminus\{v\}}) $ is supersolvable (resp., (inductively) free). 
In this case, 
$$\exp ( \mathbf{c}\Ac(G,\psi)  )   = \exp \left( \mathbf{c}\Ac(G\setminus v,\psi|_{[\ell]\setminus\{v\}})  \right)  \cup \{|\psi(v)|+e+\ell-1\},$$ where  $ e $ denotes the number of edges incident on $ v $.
\item[(ii)]  If  $\mathbf{c}\Ac\left(G\setminus v, \psi|_{[\ell]\setminus\{v\}}\right)$ is ((ind-)flag-)accurate whose exponents do not exceed  $|\psi(v)|+e+\ell-1$, then $ \mathbf{c}\Ac(G,\psi) $ is ((ind-)flag-)accurate.
\end{itemize}
\end{corollary}

\subsection{$N$-Ish arrangements}
\label{subsect:Nish}

Our first set of examples of $ \psi $-digraphic arrangements is the class of  \emph{$N$-Ish arrangements}  due to Abe, Suyama and Tsujie \cite{AST17} which we  now recall. 
Let $N = (N_1, N_2,\ldots, N_\ell)$ be an $\ell$-tuple of finite sets $N_i\subseteq\ZZ$ (not necessarily integer intervals). 
	 \begin{example}
	\label{ex:N-Ish} 
The  \emph{(essentialized) $N$-Ish arrangement} $\Ac(N)$ is the arrangement  in $ \mathbb{R}^{\ell} $ defined by 
$$
\Ac(N):=
\Cox(A_\ell) 
\cup \{  x_{i}= N_i \mid 1 \leq i   \leq \ell\}.
$$
Thus any $N$-Ish arrangement is a $ \psi $-digraphic arrangement: $\Ac(N) = \Ac(G,\psi)$ where $G =  \overline{K}_{\ell}$ is the edgeless digraph on $ [\ell]$ and $\psi(i) =N_i$ for each $i$.

\end{example}

\begin{theorem}[{\cite[Thm.~ 1.3]{AST17}}]
 \label{thm:nested}
The following are equivalent:
 \begin{itemize}
\item[(i)]  The cone $\cc\Ac(N)$ is supersolvable.
\item[(ii)] The cone $\cc\Ac(N)$  is inductively free.
\item[(iii)] The cone $\cc\Ac(N)$  is  free.
\item[(iv)] $N$ is nested, i.e.,  there exists a permutation $w$ of $[\ell]$ such that $  N_{w(i)} \subseteq N_{w(i-1)}$ for every $2 \le  i \leq \ell$.
\end{itemize}
In this case, the exponents of  $\cc\Ac(N)$ are given by $(1,  | N_{w(i)}| + i-1 )_{i=1}^\ell$ where  $w$ is any permutation of $[\ell]$ such that $  N_{w(i)} \subseteq N_{w(i-1)}$ for every $2 \le  i \leq \ell$.

 \end{theorem} 
 
 We call an  $\ell$-tuple $N$ as above \emph{strictly nested} if there exists a permutation $w$ of $[\ell]$ such that $  N_{w(i)} \subsetneq N_{w(i-1)}$ for every $2 \le  i \leq \ell$.
Our next lemma asserts that strict nestedness is sufficient for ind-flag-accuracy of $\cc\Ac(N)$. 

 \begin{lemma}
\label{lem:N-Ish-FA}
Let $N = (N_1, N_2,\ldots, N_\ell)$ with $ N_\ell  \subsetneq \cdots  \subsetneq N_2  \subsetneq N_1$. 
Then $\cc\Ac(N)$ is ind-flag-accurate.
\end{lemma}

\begin{proof}
Note that $ | N_{i} | -  | N_{i+1} | \ge 1$ for each $1 \le i \le \ell-1$. 
By Theorem \ref{thm:nested},  $\Ac: = \cc\Ac(N)$ is supersolvable hence inductively free with exponents 
$ (1,  | N_{\ell} | + \ell-1, \ldots,   | N_{2} | + 1, | N_{1} |)_{\le}. $
For each $1 \le i \le \ell-1$ fix $a_i \in N_{i}  \setminus N_{i+1}$, and set $M_i : = \{ a_j \mid 1 \le j \le i\}$. 
Also, let $H_i$ denote the hyperplane $x_i =a_iz$, and set $X_i : = \bigcap_{j=1}^{i} H_j$. 
Then $X_{\ell-1} \subseteq \cdots \subseteq  X_2  \subseteq X_1 \subseteq V'=\RR^{\ell+1} $ and $\dim_{V'}(X_i) =\ell+1-i$.
Moreover,  for each $1 \le i \le \ell-1$, one can show that $\Ac^{X_i}$ is identical to the  nested $N$-Ish arrangement defined by
$$N_\ell  \sqcup M_i \subsetneq \cdots  \subsetneq N_{i+2} \sqcup M_i  \subsetneq N_{i+1} \sqcup M_i .$$
Once again, according to Theorem \ref{thm:nested},  $\Ac^{X_i}$ is inductively free with exponents given by 
$ (1,  | N_{\ell} | + \ell-1, \ldots,   | N_{i+1} | + i)_{\le}. $
Thus $\cc\Ac(N)$ is ind-flag-accurate with a witness $(X_{\ell-1}, \ldots , X_2, X_1, V' )$.
\end{proof}

We complete this subsection by defining two   operations on vertex-weighted digraphs with interval weights needed for our subsequent discussion.

\begin{definition}
A vertex $ v $ in a digraph $ G $ is called a
\emph{source} (resp.~\emph{sink})
 %\emph{king} (resp., \emph{coking}) 
if $ (v,u) \in E_{G} $ (resp., $ (u,v) \in E_{G} $) for every $ u \in V_{G}\setminus\{v\} $. 
\end{definition}
 
\begin{definition}\label{deformation coking}
Let $(G,\psi)$ be a vertex-weighted digraph with non-empty interval weight, that is, $ G $ is a digraph on $ [\ell] $ and $ \psi \colon [\ell] \to 2^{\mathbb{Z}} $ is a  map such that  $ \emptyset \ne \psi(i) = [a_{i}, b_{i}]\subseteq \mathbb{Z}$ where $a_{i}\le b_{i}$ are integers for every $ i \in [\ell] $. Let $v$ be a vertex in $G$.

Next we define the \emph{mutation} of $(G,\psi)$ with respect to a sink or source.
\begin{enumerate}[(i)]
\item Suppose that $v$ is a sink in $ G $. 
The \emph{mutation}
of $(G,\psi)$ (w.r.t.~$v$) is a new vertex-weighted digraph $ (G^{\prime}, \psi^{\prime}) $ where $ G^{\prime}  =( [\ell] , E_{G'})$ is a digraph   and $ \psi^{\prime} $ is a weight given by
\begin{align*}
E_{G^{\prime}} :=E_{G}\setminus\{(i,v) \mid i \in [\ell]\setminus\{v\}\}, \qquad
\psi^{\prime}(i) :=\begin{cases}
[a_{i}-1, b_{i}] & (i \in [\ell]\setminus\{v\}), \\
[a_{v}, b_{v}] & (i=v). 
\end{cases}
\end{align*}
\item Dually, suppose that  $v$ is a source in $ G $. 
The \emph{mutation}
of $(G,\psi)$ (w.r.t.~$v$) is a new vertex-weighted digraph $ (G'', \psi'') $ given by
\begin{align*}
E_{G''} :=E_{G}\setminus\{(v,i) \mid i \in [\ell]\setminus\{v\}\}, \qquad
\psi''(i) :=\begin{cases}
[a_{i}, b_{i}+1] & (i \in [\ell]\setminus\{v\}), \\
[a_{v}, b_{v}] & (i=v). 
\end{cases}
\end{align*}
\end{enumerate}
\end{definition}

Subsequently, when speaking of a \emph{mutation} or of an \emph{evolution} 
of a $ \psi $-digraphic arrangement $ \Ac(G,\psi) $, we mean that mutation is applied to the underlying vertex-weighted digraph $(G,\psi)$. 
In particular, a mutation 
w.r.t.~the sink (source) $v$ induces a (set) bijection on the underlying arrangements $  \Ac(G,\psi) \to  \Ac(G^{\prime}, \psi^{\prime}) $ (resp., $  \Ac(G,\psi) \to  \Ac(G'', \psi'') $).

What we call a source, sink and mutation with respect to one of the former is called a ``king'', ``coking'' and ``(co)king elimination operation'', respectively, in \cite{AbeTranTsujie21_ShiIsh}.

\subsection{Shi genealogy}
\label{subsect:Shi}

Evolution leads to many ind-flag-accurate $ \psi $-digraphic arrangements. 
Among them is a deformation of the Shi arrangement of type $A$ which we are about to introduce.

 \begin{definition}
\label{def:Shi-ally}
Let $1 \le k \le \ell$ and recall   the digraph $ T_{\ell}^{k}$ from Definition \ref{def:digraphs}. 
Fix $\ell$ pairs of integers $c_i \le d_i$ for $1 \le i \le \ell$. 
Let  
$$\Aa^{k}_\ell:=  \Ac\left(T_{\ell}^{k}, \psi_{\ell}^{k}\right)  \quad \text{for } 1 \le k \le \ell
$$  
be the $ \psi $-digraphic arrangements defined by the following recurrence relation:
\begin{enumerate}[(i)]
\item  $\Aa^{1}_\ell =\Ac\left(T_{\ell}^1, \psi^1_\ell\right) $ with $ \psi^1_\ell(i) = [c_{i}, d_{i}] $  for each $ i \in [\ell] $.

\item For $k=1,2,\ldots, \ell-1$ the arrangement $\Aa^{ k+1}_\ell$ is obtained from $\Aa^{ k}_\ell$ by applying 
mutation w.r.t.~the sink
$ \ell-k+1 $ of the induced subgraph  $ \left(T_{\ell}^{k}[\ell-k+1], \psi_{\ell}^{k}|_{[\ell-k+1]}\right)$ of $\left(T_{\ell}^{k}, \psi_{\ell}^{k}\right)$ by $[\ell-k+1]$, and keeping the isolated vertices  $ i \in [\ell-k+2, \ell] $ with their weights unchanged. 
The relation induces a bijection between the arrangements $\Aa^{ k}_\ell$ and $\Aa^{ k+1}_\ell$, denoted $\tau_{\ell-k+1} :  \Aa^{ k}_\ell \to \Aa^{ k+1}_\ell.$
\end{enumerate}
We call the arrangements $\Aa^{k}_\ell$  \emph{Shi descendants}, and call the sequence 

$$ \Aa^{1}_\ell \stackrel{ \tau_{\ell} }{\longrightarrow} \Aa^{2}_\ell \stackrel{ \tau_{\ell-1} }{\longrightarrow} \cdots \stackrel{ \tau_{2} }{\longrightarrow}  \Aa^{\ell}_\ell$$
a \emph{Shi descendant sequence} with \emph{origin  $\Aa^{1}_\ell$} and \emph{end $\Aa^{\ell}_\ell$}. 
\end{definition}

The solution of the recurrence relation from Definition \ref{def:Shi-ally} is given below.
\begin{proposition} 
\label{prop:Shi-ally-formula}
With the notation in Definition \ref{def:Shi-ally}, for each $1 \le k \le \ell$, we have
$$
 \psi_{\ell}^{k}(i) =  \varphi_{\ell}^{k}(i), \text{ where } \varphi_{\ell}^{k}(i) := [1-\min\{\ell-i+1, k\}+c_i, d_i ] \quad  (1 \leq i \leq \ell). 
$$
In other words, $\Aa^{k}_\ell$ consists of the following hyperplanes:
\begin{align*}
x_{i} - x_{j} &= 0\quad  (1 \leq i < j \leq \ell), \\
x_{i} - x_{j} &= 1 \quad (1 \leq i < j \leq \ell+1-k), \\
x_i&= [1-k+c_i, d_i ] \quad (1 \leq i \leq \ell+1-k), \\
x_i&= [-\ell + i+ c_i, d_i ] \quad (\ell+1-k < i \leq \ell).
\end{align*}
\end{proposition}

\begin{proof}
It is clear that $\psi^1_\ell = \varphi^1_\ell$. 
 Fix $1 \le k \le \ell-1$ and set $v:= \ell-k+1 $. 
Each vertex $ i \in [v+1, \ell] $ is isolated and its weight is the same in both $ \left(T_{\ell}^{k}, \varphi_{\ell}^{k}\right) $ and $\left(T_{\ell}^{k+1}, \varphi_{\ell}^{k+1}\right)$ given by $ \varphi_{\ell}^{k}(i) =  \varphi_{\ell}^{k+1}(i) = [-\ell + i+ c_i, d_i ]$.
It suffices to prove that  after applying  a mutation to
%the CEO to 
$ \left(T_{\ell}^{k}[v], \varphi_{\ell}^{  k}|_{[v]}\right) $ w.r.t.~the 
%coking 
sink $v $, we obtain  $\left(T_{\ell}^{k+1}[v], \varphi_{\ell}^{ k+1}|_{[v]}\right)$.
It is clear that $ (T_{\ell}^{k}[v])^{\prime} = T_{\ell}^{k+1}[v] $. 
Moreover, 
$$( \varphi_{\ell}^{ k}|_{[v]})^{\prime} ( v ) = [1-k+c_v, d_v ] =  \varphi_{\ell}^{ k+1}|_{[v]} ( v ),$$
and for each $i \in [\ell-k]$
$$( \varphi_{\ell}^{ k}|_{[v]})^{\prime} ( i ) = [-k+c_i, d_i ] = \varphi_{\ell}^{ k+1}|_{[v]} ( i ),$$
Thus $( \varphi_{\ell}^{ k}|_{[v]})^{\prime} =  \varphi_{\ell}^{ k+1}|_{[v]}$, as desired.  
 \end{proof}

One of the main examples of Shi descendant sequences is the sequence of \emph{arrangements between Shi and Ish}  introduced recently by Duarte and Guedes de Oliveira \cite{DG18, DG19}. 
Let $2 \leq k \leq \ell$. 
Let $\Ss_{\ell}^{k}$ be the  arrangement consisting of the following hyperplanes:
\begin{align*}
x_{i} - x_{j} &= 0\quad  (1 \leq i < j \leq \ell), \\
x_{1}-x_{j} &= i \quad (1 \leq i < j \leq \ell, i < k), \\
x_{i}-x_{j} &= 1 \quad (k \leq i < j \leq \ell).
\end{align*} 
In particular, $\Ss_{\ell}^1 = \Ss_{\ell}^{2}$ is known as the \emph{type $A$ Shi arrangement} due to Shi \cite[Chap.~ 7]{Shi86}, whose essentialization agrees with $\Shi^1(A_{\ell-1}) $  from Definition \ref{Def_IdealShiCatalan}. 
In addition, $ \Ss_{\ell}^{\ell}  $ is known as the \emph{type $A$ Ish arrangement} due to Armstrong \cite{Arms13}. 
The Shi and Ish arrangements share many common properties, e.g., see  \cite{AR12}. 
The arrangements $\Ss_{\ell}^{k} $ interpolate between the Shi and Ish arrangements as $k$ varies. 

\begin{proposition}[{\cite[Prop.~2.11]{AbeTranTsujie21_ShiIsh}}]
\label{prop:Akl}
If $ 1 \leq k \leq \ell $ then
$$\Ss_{\ell+1}^{k+1} = \Aa_{\ell}^{k}[-1,0]\times \varnothing_{1},$$
where $\Aa_{\ell}^{k}[-1,0]$ denotes the Shi descendant with $c_i=-1$, $d_i=0$ for each $1 \le i \le \ell$.
\end{proposition}

\begin{theorem}[{\cite[Thm.~ 1.6]{AbeTranTsujie21_ShiIsh}}]
\label{thm:ATT-main1}
If $1 \leq k \leq \ell$, then the cone $\textbf{c} \Ss_{\ell}^{k} $ is free with exponents $(0, 1, \ell^{\ell-1})$. 
Moreover, if $\ell \ge 3$ then only $\textbf{c} \Ss_{\ell}^{\ell} $ is supersolvable  among the arrangements $\textbf{c} \Ss_{\ell}^{k} $.
\end{theorem}

 Now we define the protagonist of this subsection, a special subclass of Shi descendants at the same time extending ``vertically" the arrangements between Shi and Ish above.
 
 \begin{definition}
\label{def:Shi-ally-ext}
Let  $d, m\ge 0$, $1 \le k \le \ell$, $0 \le p \le \ell$ be integers. 
Let  $\Aa^{p,k}_\ell(m,d)  $ be the $ \psi $-digraphic arrangement defined by 
$$ 
\Aa^{p,k}_\ell(m,d) := \Ac\left(T_{\ell}^{k}, \psi_{\ell}^{p,k}\right),
$$
where $ T_{\ell}^{k}$ is the digraph from Definition \ref{def:digraphs} and the map $ \psi_{\ell}^{p,k} \colon [\ell] \to 2^{\mathbb{Z}} $ is given by 
$$
 \psi_{\ell}^{p,k}(i)  :=\begin{cases}
[1-\min\{\ell-i+1, k\} - m , d ] & (1 \leq   i \leq p), \\
[-\min\{\ell-i+1, k\} - m , d] & (p<  i \leq \ell). 
\end{cases}
$$
In other words, $ \Aa^{p,k}_\ell(m,d) $ consists of the following hyperplanes:
\begin{align*}
x_{i} - x_{j} &= 0\quad  (1 \leq i < j \leq \ell), \\
x_{i} - x_{j} &= 1 \quad (1 \leq i < j \leq \ell+1-k), \\
x_i&= [1-\min\{\ell-i+1, k\} - m , d ]  \quad (1 \leq   i \leq p), \\
x_i&= [-\min\{\ell-i+1, k\} - m , d]  \quad (p<  i \leq \ell).
\end{align*}
\end{definition}

 A convenient way to view the arrangements above is to place them in an $(\ell+1)\times \ell$ matrix $(a_{p,k})_{0 \le p \le \ell, 1 \le k \le \ell}$ with entry $a_{p,k}$ for the arrangement $\Aa^{p,k}_\ell (m,d)$. 
We call this matrix the  \emph{Shi descendant matrix}: 
$$
\begin{pmatrix}
    \Aa^{0,1}_\ell (m,d)& \overset{\small{k\, \text{varies}}}{\longrightarrow} & \Aa^{0,\ell}_\ell (m,d)\\
\downarrow    \scriptsize{\text{$p$ varies}}  & \vdots  & \vdots  \\
\Aa^{\ell,1}_\ell (m,d) & \cdots & \Aa^{\ell,\ell}_\ell (m,d) 
          \end{pmatrix}.
       $$
      
      The proposition below asserts that each arrangement $\Aa^{p,k}_\ell(m,d)  $ is indeed a Shi descendant.
\begin{proposition}
 \label{prop:Shi-pth-row}
Fix $0 \le p \le \ell$.
The  arrangements $\Aa^{p,k}_\ell(m,d)$ in the $p$-th row of the Shi descendant matrix constitute the Shi descendant  sequence

$$ \Aa^{p,1}_\ell (m,d) \ \stackrel{ \tau_{\ell} }{\longrightarrow}  \Aa^{p,2}_\ell (m,d)  \stackrel{ \tau_{\ell-1} }{\longrightarrow}  \cdots  \stackrel{ \tau_{2} }{\longrightarrow}  \Aa^{p,\ell}_\ell (m,d).
$$

Moreover, for fixed $1 \le k \le \ell$ each vertex $n \in [\ell-k+2, \ell] $ is isolated and simplicial in  $ \left(T_{\ell}^{k}[n], \psi_{\ell}^{p,k}|_{[n]}\right)$, the induced subgraph of $\left(T_{\ell}^{k}, \psi_{\ell}^{p,k}\right)$ by $[n]=\{1,2,\ldots,n\}$.
 \end{proposition} 

 \begin{proof}
The first statement is clear from Proposition \ref{prop:Shi-ally-formula}. 
To show the second statement we consider two cases: $p \le \ell + 1 -k$ and  $p > \ell + 1 -k$. 
If the former occurs, then by definition 
$$
 \psi_{\ell}^{p,k}(i)  =\begin{cases}
[1 - m -k, d ] & (1 \leq   i \leq p), \\
[ - m - k , d] & (p<  i \leq  \ell + 1 -k), \\
[ - m - \ell+i-1 , d] & (\ell + 1 - k <  i \leq  \ell ).
\end{cases}
$$
Therefore for each $i \in [\ell + 1 -k]$ we have
$$ \psi_{\ell}^{p,k}(i) \supseteq  \psi_{\ell}^{p,k}(\ell + 2 -k) \supseteq \cdots \supseteq  \psi_{\ell}^{p,k}(\ell).$$
Owing to Proposition \ref{prop:isolated}, we know that each  $n \in [\ell-k+2, \ell] $ is isolated and simplicial in  $ \left(T_{\ell}^{k}[n], \psi_{\ell}^{p,k}|_{[n]}\right)$. 
If the latter occurs, then
$$
 \psi_{\ell}^{p,k}(i) =\begin{cases}
[1 - m -k, d ] & (1 \leq   i \leq  \ell + 1 -k), \\
[ - m - \ell+i , d]  & ( \ell + 1 -k<  i \leq p), \\
[ - m - \ell+i-1 , d] & (p<  i \leq  \ell ).
\end{cases}
$$
Therefore, for each $i \in [\ell + 1 -k]$ we have
$$ \psi_{\ell}^{p,k}(i) \supseteq  \psi_{\ell}^{p,k}(\ell + 2 -k) \supseteq \cdots  \supseteq \psi_{\ell}^{p,k}(p)  = \psi_{\ell}^{p,k}(p+1)  \supseteq \cdots \supseteq  \psi_{\ell}^{p,k}(\ell).$$
The result follows from another application of Proposition \ref{prop:isolated}.
\end{proof}

Thus the arrangement $\Aa^{p,1}_\ell (m,d)$ in the first column of the Shi descendant  matrix given by
\begin{equation}
\label{eq:1stcol}
  \Aa^{p,1}_\ell (m,d): 
\begin{cases}
x_{i}- x_{j} &= [0 , 1]\quad  (1 \leq i < j \leq \ell), \\
x_i&=  [-m , d] \quad (1 \leq   i \leq p), \\
x_i&=[-m-1 , d]  \quad (p<  i \leq \ell),
\end{cases}
\end{equation}
 is the origin of the  Shi descendant  sequence in the $p$-th row and plays the role of a Shi-like arrangement. 
The end $\Aa^{p,\ell}_\ell (m,d)$  in the last column given by
\begin{equation}
\label{eq:lastcol}
  \Aa^{p,\ell}_\ell (m,d):
  \begin{cases}
x_{i}-x_{j} &= 0\quad  (1 \leq i < j \leq \ell), \\
x_i&=  [-m+i-\ell ,d] \quad (1 \leq   i \leq p), \\
x_i&=  [-m+i-\ell-1 ,d]   \quad (p<  i \leq \ell),
\end{cases}
\end{equation}
 is precisely a nested $N$-Ish arrangement hence it is  supersolvable  (Theorem \ref{thm:nested}). 
In particular, the   arrangements $\Aa_{\ell}^{k}[-1,0]$ in Proposition \ref{prop:Akl} can be found in the $\ell$-th row of the Shi descendant  matrix when $m=1, d=0$, or the $0$-th row when $m=0, d=0$. 
See Figure \ref{fig:ASI-arr} for the Shi descendant  matrix for $\ell=3$.

 \begin{figure}[htbp!]
\centering
\begin{subfigure}{.3\textwidth}
  \centering
\begin{tikzpicture}[scale=.5]
\draw (0,1.3) node[v](1){} node[above]{\tiny $\substack{ \textbf{1} \\ [-m-1,d]} $};
\draw (-0.75,0) node[v](2){} node[left]{\tiny$ \substack{ \textbf{2} \\ [-m-1,d]} $};
\draw (0.75,0) node[v](3){} node[right]{\tiny$\substack{ \textbf{3} \\ [-m-1,d] }$};
\draw[>=Stealth,->] (1) to (2);
\draw[>=Stealth,->] (1) to (3);
\draw[>=Stealth,->] (2) to (3);
\end{tikzpicture}
  \caption*{$ \Aa^{0,1}_3 (m,d) $}
\end{subfigure}%
\begin{subfigure}{.3\textwidth}
  \centering
\begin{tikzpicture}[scale=.5]
\draw (0,1.3) node[v](1){} node[above]{\tiny $ \substack{ \textbf{1} \\ [-m-2,d]} $};
\draw (-0.75,0) node[v](2){} node[left]{\tiny $\substack{ \textbf{2} \\ [-m-2,d]} $};
\draw (0.75,0) node[v](3){} node[right]{\tiny $\substack{ \textbf{3} \\ [-m-1,d]} $};
\draw[>=Stealth,->] (1) to (2);
\end{tikzpicture}
  \caption*{$ \Aa^{0,2}_3 (m,d) $}
\end{subfigure}%
\begin{subfigure}{.35\textwidth}
  \centering
\begin{tikzpicture}[scale=.5]
\draw (0,1.3) node[v](1){} node[above]{\tiny $ \substack{ \textbf{1} \\ [-m-3,d]} $};
\draw (-0.75,0) node[v](2){} node[left]{\tiny $\substack{ \textbf{2} \\ [-m-2,d]} $};
\draw (0.75,0) node[v](3){} node[right]{\tiny $\substack{ \textbf{3} \\ [-m-1,d]} $};
\end{tikzpicture}
  \caption*{$ \Aa^{0,3}_3 (m,d) $}
\end{subfigure}%

\bigskip
\begin{subfigure}{.3\textwidth}
  \centering
\begin{tikzpicture}[scale=.5]
\draw (0,1.3) node[v](1){} node[above]{\tiny $\substack{ \textbf{1} \\ [-m,d]} $};
\draw (-0.75,0) node[v](2){} node[left]{\tiny$ \substack{ \textbf{2} \\ [-m-1,d]} $};
\draw (0.75,0) node[v](3){} node[right]{\tiny$\substack{ \textbf{3} \\ [-m-1,d] }$};
\draw[>=Stealth,->] (1) to (2);
\draw[>=Stealth,->] (1) to (3);
\draw[>=Stealth,->] (2) to (3);
\end{tikzpicture}
  \caption*{$ \Aa^{1,1}_3 (m,d) $}
\end{subfigure}%
\begin{subfigure}{.3\textwidth}
  \centering
\begin{tikzpicture}[scale=.5]
\draw (0,1.3) node[v](1){} node[above]{\tiny $ \substack{ \textbf{1} \\ [-m-1,d]} $};
\draw (-0.75,0) node[v](2){} node[left]{\tiny $\substack{ \textbf{2} \\ [-m-2,d]} $};
\draw (0.75,0) node[v](3){} node[right]{\tiny $\substack{ \textbf{3} \\ [-m-1,d]} $};
\draw[>=Stealth,->] (1) to (2);
\end{tikzpicture}
  \caption*{$ \Aa^{1,2}_3 (m,d) $}
\end{subfigure}%
\begin{subfigure}{.35\textwidth}
  \centering
\begin{tikzpicture}[scale=.5]
\draw (0,1.3) node[v](1){} node[above]{\tiny $ \substack{ \textbf{1} \\ [-m-2,d]} $};
\draw (-0.75,0) node[v](2){} node[left]{\tiny $\substack{ \textbf{2} \\ [-m-2,d]} $};
\draw (0.75,0) node[v](3){} node[right]{\tiny $\substack{ \textbf{3} \\ [-m-1,d]} $};
\end{tikzpicture}
  \caption*{$ \Aa^{1,3}_3 (m,d) $}
\end{subfigure}%

\bigskip
\begin{subfigure}{.3\textwidth}
  \centering
\begin{tikzpicture}[scale=.5]
\draw (0,1.3) node[v](1){} node[above]{\tiny $\substack{ \textbf{1} \\ [-m,d]} $};
\draw (-0.75,0) node[v](2){} node[left]{\tiny$ \substack{ \textbf{2} \\ [-m,d]} $};
\draw (0.75,0) node[v](3){} node[right]{\tiny$\substack{ \textbf{3} \\ [-m-1,d] }$};
\draw[>=Stealth,->] (1) to (2);
\draw[>=Stealth,->] (1) to (3);
\draw[>=Stealth,->] (2) to (3);
\end{tikzpicture}
  \caption*{$ \Aa^{2,1}_3 (m,d) $}
\end{subfigure}%
\begin{subfigure}{.3\textwidth}
  \centering
\begin{tikzpicture}[scale=.5]
\draw (0,1.3) node[v](1){} node[above]{\tiny $ \substack{ \textbf{1} \\ [-m-1,d]} $};
\draw (-0.75,0) node[v](2){} node[left]{\tiny $\substack{ \textbf{2} \\ [-m-1,d]} $};
\draw (0.75,0) node[v](3){} node[right]{\tiny $\substack{ \textbf{3} \\ [-m-1,d]} $};
\draw[>=Stealth,->] (1) to (2);
\end{tikzpicture}
  \caption*{$ \Aa^{2,2}_3 (m,d) $}
\end{subfigure}%
\begin{subfigure}{.35\textwidth}
  \centering
\begin{tikzpicture}[scale=.5]
\draw (0,1.3) node[v](1){} node[above]{\tiny $ \substack{ \textbf{1} \\ [-m-2,d]} $};
\draw (-0.75,0) node[v](2){} node[left]{\tiny $\substack{ \textbf{2} \\ [-m-1,d]} $};
\draw (0.75,0) node[v](3){} node[right]{\tiny $\substack{ \textbf{3} \\ [-m-1,d]} $};
\end{tikzpicture}
  \caption*{$ \Aa^{2,3}_3 (m,d) $}
\end{subfigure}%

\bigskip
\begin{subfigure}{.3\textwidth}
  \centering
\begin{tikzpicture}[scale=.5]
\draw (0,1.3) node[v](1){} node[above]{\tiny $\substack{ \textbf{1} \\ [-m,d]} $};
\draw (-0.75,0) node[v](2){} node[left]{\tiny$ \substack{ \textbf{2} \\ [-m,d]} $};
\draw (0.75,0) node[v](3){} node[right]{\tiny$\substack{ \textbf{3} \\ [-m,d] }$};
\draw[>=Stealth,->] (1) to (2);
\draw[>=Stealth,->] (1) to (3);
\draw[>=Stealth,->] (2) to (3);
\end{tikzpicture}
  \caption*{$ \Aa^{3,1}_3 (m,d) $}
\end{subfigure}%
\begin{subfigure}{.3\textwidth}
  \centering
\begin{tikzpicture}[scale=.5]
\draw (0,1.3) node[v](1){} node[above]{\tiny $ \substack{ \textbf{1} \\ [-m-1,d]} $};
\draw (-0.75,0) node[v](2){} node[left]{\tiny $\substack{ \textbf{2} \\ [-m-1,d]} $};
\draw (0.75,0) node[v](3){} node[right]{\tiny $\substack{ \textbf{3} \\ [-m,d]} $};
\draw[>=Stealth,->] (1) to (2);
\end{tikzpicture}
  \caption*{$ \Aa^{3,2}_3 (m,d) $}
\end{subfigure}%
\begin{subfigure}{.35\textwidth}
  \centering
\begin{tikzpicture}[scale=.5]
\draw (0,1.3) node[v](1){} node[above]{\tiny $ \substack{ \textbf{1} \\ [-m-2,d]} $};
\draw (-0.75,0) node[v](2){} node[left]{\tiny $\substack{ \textbf{2} \\ [-m-1,d]} $};
\draw (0.75,0) node[v](3){} node[right]{\tiny $\substack{ \textbf{3} \\ [-m,d]} $};
\end{tikzpicture}
  \caption*{$ \Aa^{3,3}_3 (m,d) $}
\end{subfigure}%
\caption{The Shi descendant  matrix for $\ell=3$.}
\label{fig:ASI-arr}
\end{figure}
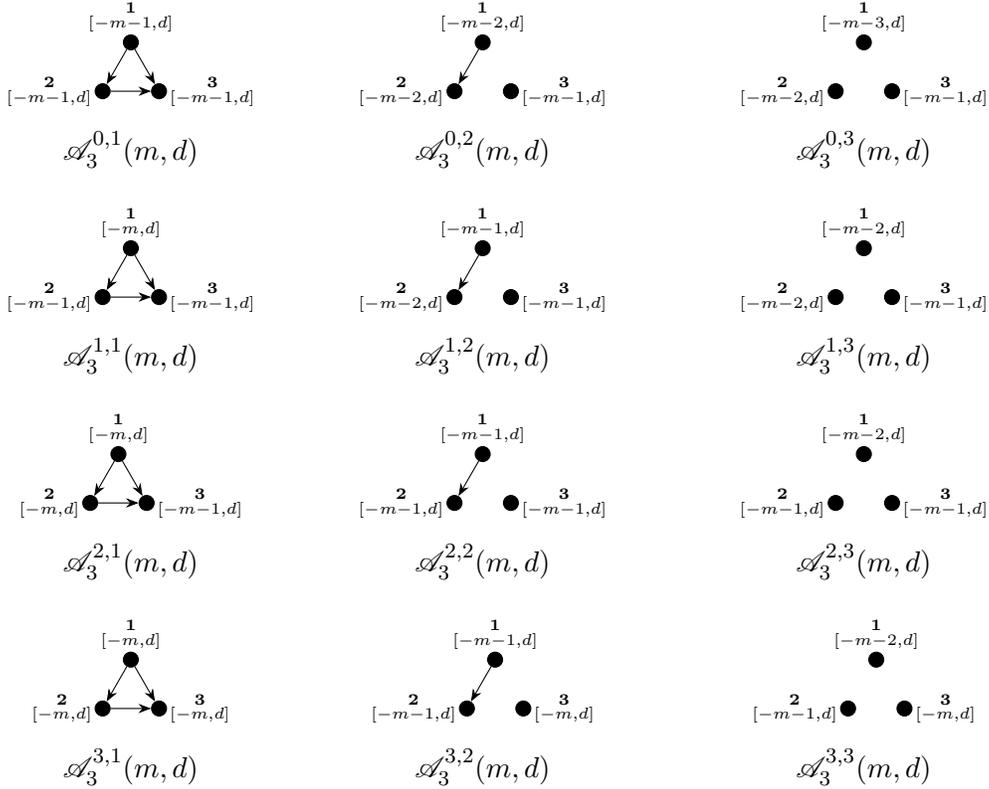

Our main result in this subsection is that all the arrangements in the Shi descendant  matrix are  ind-flag-accurate, which is given in Theorem \ref{thm:Shi-ally-FA}. 
Furthermore, all the arrangements in the same row have the same multiset of exponents, generalizing Theorem 
\ref{thm:ATT-main1}. 
Before addressing the proof of Theorem \ref{thm:Shi-ally-FA}, we comment on the freeness of the members in the Shi descendant  matrix.

	 \begin{remark}
	\label{rem:CKZ} 
	The freeness of the arrangements in the Shi descendant  matrix can be readily verified by using  \cite[Thms.~ 3.1 and 4.1]{AbeTranTsujie21_ShiIsh} (and the modular coatom technique in Corollary \ref{cor:simplicial}) which implies that the maps $\tau_{\ell}, \ldots, \tau_2$ in Proposition \ref{prop:Shi-pth-row} preserve freeness and characteristic polynomials of the arrangements involved. 
	This method was also used in the proof of  Theorem \ref{thm:ATT-main1}. 
	Thus the freeness and exponents of the Shi descendants follow from the supersolvability of the corresponding final terms, the ``supersolvable descendants'' in the last column. 
	However, to show the ind-flag-accuracy in the proof of Theorem \ref{thm:Shi-ally-FA}, we need  a different ``reverse'' approach: We first employ the ind-flag-accuracy of the Shi-like arrangements (the origins) in the first column, then observe that this property propagates to all the arrangements in the Shi descendant  matrix. 
\end{remark}

The Shi-like arrangements $\Aa^{p,1}_\ell (m,d)$ for $0 \le p \le \ell$ when $d=0$ are affinely equivalent to (the essentializations of) the inductively free arrangements in \cite[Thms.~3.1 and 3.3]{Atha98} due to Athanasiadis. 
The Shi arrangements therein were defined so that they contain sufficient deletions and restrictions in order to apply the addition-deletion theorem (Theorem \ref{thm:AD}) to guarantee the inductive freeness of all the arrangements in the family. 
This observation is also helpful for the proof of their ind-flag-accuracy which we show below as a generalization of \cite[Thm.~ 3.3]{Atha98}. 
Roughly speaking, in order to achieve the inductive freeness and ind-flag-accuracy, we consider the restriction to the ``$m$-hyperplane'' $x_i= -(m+1)z$ and the  ``$n$-hyperplane'' $x_i = (na-1)z$, respectively (see also Theorems \ref{thm:Cat-ext} and \ref{thm:nt-starter} for a similar observation for Catalan arrangements).

\begin{theorem}
 \label{thm:A-H}
Let $m\ge 0$, $a, n, \ell \ge 1$, and $0 \le p \le \ell$. 
Let $\Hc^{p}_\ell(m,a,n)$ be the  arrangement consisting of the following hyperplanes:
\begin{align*}
x_{i} - x_{j} &= [1-a, a] \quad  (1 \leq i < j \leq \ell), \\
x_i&=  [-m, na-1]  \quad (1 \leq   i \leq p), \\
x_i&=  [-m-1, na-1] \quad (p<  i \leq \ell).
\end{align*}
The cone over $\Hc^{p}_\ell(m,a,n)$
is  ind-flag-accurate with exponents 
$$ \exp\left(\cc\Hc^{p}_\ell(m,a,n)\right) = (1, 0^{p}, 1^{\ell-p}) +  (0,  (m + a (\ell+n-1))^{\ell} ).$$ 

As a consequence,  the origin $\Aa^{p,1}_\ell(m,d) = \Hc^{p}_\ell(m,1,d+1)$ in the first column of the Shi descendant matrix given in \eqref{eq:1stcol} has  ind-flag-accurate cone with exponents 
$$\exp\left(\cc\Aa^{p,1}_\ell(m,d)\right) = (1, (\ell + m+d)^{p}, (\ell + m+d+1)^{\ell-p}).$$ 
 \end{theorem} 

 \begin{proof}
 The  inductive  freeness of $\cc\Hc^{p}_\ell(m,a,n)$  when $n=1$ was shown in  \cite[Thm.~ 3.3]{Atha98} (after the transformation $x_{i}\mapsto x_{i}-x_0$ for all $i$) by an inductive method (see also \cite[Ex.~(2.6)]{JT84}). 
Our proof uses a similar argument but requires an additional treatment when $n>1$ (see Case $2$ below).
 
 Define the lexicographic order on 
$$\{ (\ell, \ell-p) \mid 0\le \ell- p \le \ell, \, \ell \ge 1\}.$$
First we prove that $\Ac :=\cc\Hc^p_\ell(m, a, n)$ belongs to  $\mathcal{IF}$ with the desired exponents, by induction on $ (\ell, \ell-p) $. 
When $\ell=1$, it is obvious. Suppose $\ell\ge2$.

Case $1$. First consider $0 \le p \le \ell-1$. 
Let  $H \in \Ac$ denote the hyperplane $x_{p+1} = -(m+1)z$. 
Then $\Ac'=\Ac\setminus \{H\} = \cc\Hc^{p+1}_\ell(m,a,n)\in \mathcal{IF}$ with exponents  $$\exp(\cc\Hc^{p+1}_\ell(m,a,n)) = (1, 0^{p+1}, 1^{\ell-p-1}) +  (0,  (m + a (\ell+n-1))^{\ell} )$$ by the induction hypothesis since $(\ell, \ell-p) > (\ell, \ell-p-1)$. 
Moreover, $\Ac''=\Ac^{H}$ consists of the following hyperplanes 
\begin{align*}
z &= 0 , \\
x_{i} - x_{j} &=  [1-a , a]z\quad  (1 \leq i < j \leq \ell, i\ne p+1,j\ne p+1 ), \\
x_i&=  [-(m+a) , na -1]z  \quad (1 \leq   i \leq p), \\
x_i&=  [-(m+a+1) , na -1]z \quad (p+1<  i \leq \ell).
\end{align*}
Thus $\Ac'' =    \cc\Hc^{p}_{\ell-1}(m+a, a,n)$. 
Hence   $\Ac'' \in \mathcal{IF}$ with exponents $$\exp(\cc\Hc^{p}_{\ell-1}(m+a, a,n)) =(1, 0^{p}, 1^{\ell-p-1}) +   (0,  (m + a (\ell+n-1))^{\ell-1} )$$  by the induction hypothesis since $(\ell, \ell-p) > (\ell-1,\ell-1)$. 
Therefore, by the addition part of Theorem \ref{thm:AD}, $\Ac =\cc\Hc^p_\ell(m, a, n)\in \mathcal{IF}$ with the desired exponents. 

Case $2$. Now consider  $\ell=p$. 
We may assume further that $m=0$  since $\Hc^{\ell}_\ell(m,a,n) = \Hc^{0}_\ell(m-1,a,n)$.
The arrangement  in question is $\Bc(n) := \cc\Hc^\ell_\ell(0,a,n)$ consisting of 
\begin{align*}
z &= 0 , \\
x_{i} - x_{j} &=  [1-a , a]z\quad  (1 \leq i < j \leq \ell), \\
x_i&=  [0,  na -1]z \quad (1 \leq   i \leq \ell).
\end{align*}

We are going to show that $\Bc(n)$ belongs to $\mathcal{IF}$  with exponents $ (1,   ( a (\ell+n-1))^{\ell} )$ by induction on $n$. 
The base case $n=1$ was already shown in  \cite[Thm.~3.3]{Atha98}.

Now we show that for $n\ge 1$ if $\Bc(n) \in \mathcal{IF}$ then $\Bc(n+1) \in \mathcal{IF}$ with the desired exponents. 
Let $0 \le k \le a$ and denote by $\Bc_k(n) $ the arrangement having the following hyperplanes 
\begin{align*}
z &= 0 , \\
x_{i} - x_{j} &=  [1-a , a]z\quad  (1 \leq i < j \leq \ell), \\
x_i&=  [0, na -1 +k]z \quad (1 \leq   i \leq \ell).
\end{align*}
In particular, we have  $\Bc_0(n) =  \Bc(n)$, and $\Bc_a(n) =  \Bc(n+1)$. 
Notice that for any $ \Bc_0(n) \subsetneq  \Dc \subseteq  \Bc_a(n)  $ and for any hyperplane $H \in \Dc$ of the form $x_d = sz$ for $1 \leq d \leq \ell$, $na \le s \le (n+1)a-1$, the restriction $\Dc^H$ is linearly equivalent to $\cc\Hc^{\ell-d}_{\ell-1}(s,a,1)$. 
Hence $\Dc^H \in \mathcal{IF}$  with exponents $ (1, 0^{\ell-d}, 1^{d-1}) +  (0,  (s + a (\ell-1))^{\ell-1} )$ by the induction hypothesis on $ (\ell, \ell-p) $. 
Now, thanks to Theorem \ref{thm:AD}, adding $\ell$ of the hyperplanes in $\Bc_1(n) \setminus \Bc_0(n) $ to $ \Bc_0(n) $ in the order $H_1, H_2,\ldots, H_\ell$, where $H_i$ denotes the hyperplane $x_i = naz$,  implies that $\Bc_1(n) \in \mathcal{IF}$  with exponents $ (1,   ( a (\ell+n-1)+1)^{\ell} )$.  
For fixed $0 \le k \le a-1$, adding the hyperplanes in $\Bc_{k+1}(n) \setminus \Bc_k(n) $ to $ \Bc_k(n) $, 
repeatedly,
in the order $x_1 = (na -1 +k)z$, $x_2= (na -1 +k)z, \ldots$, $x_\ell = (na -1 +k)z$, implies that $\Bc_{k+1}(n)  \in \mathcal{IF}$    with exponents $ (1,   ( a (\ell+n-1)+k)^{\ell} )$.
The case $k =a$ gives the desired result for $\Bc(n+1)$. 

This completes the proof of the inductive freeness of $\Ac =\cc\Hc^p_\ell(m, a, n)$. 

Finally, we show that $\Ac$ is ind-flag-accurate. 
We argue by induction on $\ell$. 
When $\ell=1$, it is obvious. Suppose $\ell\ge2$.
First consider the case $p =\ell$ and $m=0$. 
By the discussion above, $\Ac = \Bc(n) = \cc\Hc^\ell_\ell(0,a,n) \in \mathcal{IF}$ with $\exp(\Ac) =  (1,   ( a (\ell+n-1))^{\ell} )$. 
Let  $K_1, K_2 \in \Ac$ denote the hyperplanes $x_{\ell} = 0$, $x_{\ell} = (na-1)z$, respectively. 
If $n=1$ then  $\Ac^{K_1}= \textbf{c}\Hc^{\ell-1}_{\ell-1}(a, a, 1)$. 
If $n>1$ then  $\Ac^{K_2}= \textbf{c}\Hc^{\ell-1}_{\ell-1}(0, a, n+1)$. 
Hence  both  $\Ac^{K_1}, \Ac^{K_2}$ are  ind-flag-accurate   with  $\exp(\Ac^{K_1}) =(1, (a\ell )^{\ell-1})$, and $\exp(\Ac^{K_2}) =  (1,   ( a (\ell+n-1))^{\ell-1} )$, by the induction hypothesis on $\ell$.  
By Lemma \ref{lem:flag-accuracy}, $\Ac$  is ind-flag-accurate.

We can now assume $0 \le p \le \ell-1$ since $\Hc^{\ell}_\ell(m,a,n) = \Hc^{0}_\ell(m-1,a,n)$.
Then by Case $1$ in the proof of the  inductive freeness of $\Ac$, we have $\Ac^H =   \textbf{c}\Hc^{p}_{\ell-1}(m+a, a, n)$, where $H \in \Ac$ denotes the hyperplane $x_{p+1} = -(m+1)z$. 
Hence  $\Ac^H $ is ind-flag-accurate with exponents $ (1, 0^{p}, 1^{\ell-p-1}) +   (0,  (m + a (\ell+n-1))^{\ell-1} )$ by the induction hypothesis  on $\ell$. 
By Lemma \ref{lem:flag-accuracy}, $\Ac$  is ind-flag-accurate. 

This completes the proof of the theorem.
\end{proof}

	 \begin{remark}
	\label{rem:Shi-A-alter} 
 The cone over the extended Shi arrangement $ \Shi^m (A_\ell)$ of type $A_\ell$ (Definition \ref{Def_IdealShiCatalan}) is linearly equivalent to $\cc\Hc^{\ell}_\ell(m,m,1)$ hence it is  ind-flag-accurate. 
 Moreover, the inductive argument in the proof of Theorem \ref{thm:A-H} rederives the presence of a simple root witness for the ind-flag-accuracy of $\cc \Shi^m (A_\ell)$, as in  Theorem \ref{thm:ExtShi-FA}.
\end{remark}

We are finally in a position to present the main result of this subsection.

\begin{theorem}
 \label{thm:Shi-ally-FA}
Let $d, m\ge 0$, $1 \le k \le \ell$ and $0 \le p \le \ell$. 
The cone over the  arrangement $\Aa^{p,k}_\ell(m,d)$ in the Shi descendant matrix is ind-flag-accurate with exponents $$\exp\left(\Aa^{p,k}_\ell(m,d)\right) = (1, (\ell + m+d)^{p}, (\ell + m+d+1)^{\ell-p}).$$ 
 \end{theorem}

 \begin{proof}
By Theorem \ref{thm:A-H}, it suffices to prove that for any fixed $0 \le p \le \ell$ the cone  $\Ac:=\cc\Aa^{p,k}_\ell(m,d)$  for $2 \le k \le \ell$ is ind-flag-accurate with exponents as above.

There are two cases: $p \le \ell + 1 -k$ and  $p > \ell + 1 -k$. 
Since the proofs are similar, we give a proof only for the former. 
First note that by Proposition \ref{prop:Shi-pth-row} the arrangement $\Bc := \textbf{c} \Ac\left(T_{\ell}^{k}[\ell-k+1], \psi_{\ell}^{p, k}|_{[\ell-k+1]}\right) $ is given by 
\begin{align*}
z &= 0 , \\
x_{i} - x_{j} &= 0,z \quad  (1 \leq i < j \leq\ell + 1 -k), \\
x_i&=  [1-m-k , d]z \quad (1 \leq   i \leq p), \\
x_i&= [-m-k , d]z   \quad (p<  i \leq \ell + 1 -k).
\end{align*}
Hence 
$$\Bc= \textbf{c}\Aa^{p,1}_{\ell + 1 -k}(m+k-1,d),$$
which is  ind-flag-accurate with exponents $  (1, (\ell + m+d)^{p}, (\ell + m+d+1)^{ \ell + 1 -k-p})$, by Theorem \ref{thm:A-H}. 

Thanks to Proposition \ref{prop:Shi-pth-row}, each isolated vertex $n \in [\ell-k+2, \ell] $ is  simplicial in  $ \left(T_{\ell}^{k}[n], \psi_{\ell}^{p,k}|_{[n]}\right)$ with $ \psi_{\ell}^{p,k}(n) = [-m-\ell+n-1, d]$. 
By applying Corollary \ref{cor:simplicial}(i) repeatedly to the $k-1$ simplicial vertices $\ell ,\ell-1, \ldots, \ell-k+2 $ in this order, we get  
$$ \Ac \in \mathcal{IF} \quad  \Leftrightarrow  \quad \Bc    \in \mathcal{IF} .$$  
Thus the inductive freeness of $ \Ac$ is clear. 
Moreover, $$\exp( \Ac ) = (1, (\ell + m+d)^{p}, (\ell + m+d+1)^{\ell-p}),$$
since  $| \psi_{\ell}^{p,k}(n) |+n-1 =\ell + m+d+1 $.

Note that $\Bc$ is ind-flag-accurate, none of its exponents exceeds $\ell + m+d+1 $. 
Upon applying Corollary \ref{cor:simplicial}(ii) repeatedly to the $k-1$ simplicial vertices $\ell-k+2 , \ldots, \ell $ in this order, we know that for each $n \in [\ell-k+2, \ell] $ the arrangement $\cc\Ac \left(T_{\ell}^{k}[n], \psi_{\ell}^{p,k}|_{[n]}\right)$ is ind-flag-accurate, none of its exponents exceeds $\ell + m+d+1 $. 
In particular, $ \Ac$ is  ind-flag-accurate. 
This completes the proof of the theorem.
\end{proof}
 
\subsection{Catalan genealogy}
\label{subsect:Cat}
In this subsection we show that the type $A$ Catalan arrangement can be regarded as the origin of a \emph{Catalan descendant sequence} similar to the one defined above in the setting of Shi genealogy in Section \ref{subsect:Shi}.
Several results analogous to ones from the previous subsection are obtained.

 \begin{definition}
\label{def:Cat-ally}
Let $1 \le k \le \ell$ and recall   the digraphs $ K_{\ell}^{k}$ and $ \widehat K_{\ell}^{k}$  from Definition \ref{def:digraphs}. 
Fix $\ell$ pairs of integers $c_i \le d_i$ for $1 \le i \le \ell$. 
Let  $$\Rc^{k}_\ell:=  \Ac\left(K_{\ell}^{k}, \psi_{\ell}^{k}\right) \quad \text{for} \quad 1 \le k \le \ell
$$ and 
$$\widehat\Rc^{k}_\ell:=  \Ac\left(\widehat K_{\ell}^{k}, \widehat \psi_{\ell}^{k}\right)\quad \text{for }\quad 1 \le k \le \ell-1$$   be the $ \psi $-digraphic arrangements defined by the following recurrence relation:
\begin{enumerate}[(i)]
\item  $\Rc^{1}_\ell =\Ac(K_{\ell}^1, \psi^1_\ell) $ with $ \psi^1_\ell(i) = [c_{i}, d_{i}] $  for each $ i \in [\ell] $.

\item For $k=1,2,\ldots, \ell-1$ the arrangement $\widehat\Rc^{k}_\ell$ is obtained from $\Rc^{ k}_\ell$ by applying mutation w.r.t.~the sink $ \ell-k+1 $ of the induced subgraph  $ \left(K_{\ell}^{k}[\ell-k+1], \psi_{\ell}^{k}|_{[\ell-k+1]}\right)$ of $\left(K_{\ell}^{k}, \psi_{\ell}^{k}\right)$ by $[\ell-k+1]$, and keeping the isolated vertices  $ i \in [\ell-k+2, \ell] $ with their weights unchanged. 
The relation induces a bijection between the arrangements $\Rc^{ k}_\ell$ and $\widehat\Rc^{k}_\ell$, denoted by $\tau_{\ell-k+1} :  \Rc^{ k}_\ell \to \widehat\Rc^{k}_\ell$.

\item For $k=1,2,\ldots, \ell-1$ the arrangement $\Rc^{ k+1 }_\ell$ is obtained from $\widehat \Rc^{k}_\ell$ by applying mutation w.r.t.~the source $ \ell-k+1 $ of the induced subgraph  $\left(\widehat K_{\ell}^{k}[\ell-k+1], \widehat\psi_{\ell}^{k}|_{[\ell-k+1]}\right)$ of $\left(\widehat K_{\ell}^{k},\widehat \psi_{\ell}^{k}\right)$ by $[\ell-k+1]$, and keeping the isolated vertices  $ i \in [\ell-k+2, \ell] $ with their weights unchanged. 
The relation induces a bijection between the arrangements $ \widehat\Rc^{k}_\ell$ and $\Rc^{ k +1}_\ell$, denoted by $\eta_{\ell-k+1} :  \widehat\Rc^{k}_\ell  \to \Rc^{ k+1}_\ell$.
\end{enumerate}
We call the arrangements $\Rc^{k}_\ell$, $ \widehat\Rc^{k}_\ell$ \emph{Catalan descendants}, and call the sequence 

$$ \Rc^{1}_\ell \stackrel{ \tau_{\ell} }{\longrightarrow} \widehat\Rc^{1}_\ell \stackrel{ \eta_{\ell} }{\longrightarrow} \Rc^{2}_\ell \stackrel{ \tau_{\ell -1} }{\longrightarrow} \cdots  \stackrel{ \tau_{2} }{\longrightarrow} \widehat\Rc^{\ell-1}_\ell \stackrel{ \eta_{2} }{\longrightarrow}  \Rc^{\ell}_\ell$$
a \emph{Catalan descendant sequence}. 
The arrangements $\Rc^{1}_\ell$, $ \widehat\Rc^{1}_\ell$, $\Rc^{\ell}_\ell$, $\widehat\Rc^{\ell-1}_\ell$ are called  \emph{origin, first descendant, end, penultimate descendant} of the sequence, respectively. 
\end{definition}

The solution of the recurrence relation from Definition \ref{def:Cat-ally} is given below.
\begin{proposition} 
\label{prop:Cat-ally-formula}
With the notation as in Definition \ref{def:Cat-ally}, for each $1 \le k \le \ell$ we have $ \psi_{\ell}^{k}(i) =  \varphi_{\ell}^{k}(i)$, where
\begin{align*}
\varphi_{\ell}^{k}(i) & :=[1-\min\{\ell-i+1, k\}+c_i, d_i + \min\{\ell-i+1, k\} -1] \quad  (1 \leq i \leq \ell), \\
& = 
\begin{cases}
[1-k +c_i , d_i + k-1 ] & (1 \leq   i \leq \ell-k+1), \\
[-\ell+i +c_i , d_i + \ell-i ] & (\ell-k+1 <  i \leq \ell). 
\end{cases}
\end{align*}
In addition, for each $1 \le k \le \ell-1$ we have $ \widehat \psi_{\ell}^{k}(i) =  \widehat \varphi_{\ell}^{k}(i)$, where
$$
 \widehat\varphi_{\ell}^{k}(i) :=
 \begin{cases}
[-k +c_i , d_i + k-1 ] & (1 \leq   i \leq \ell-k), \\
[-\ell+i +c_i , d_i + \ell-i] & (\ell-k <  i \leq \ell). 
\end{cases}
$$

\end{proposition}

\begin{proof}
It is clear that $\psi^1_\ell = \varphi^1_\ell$. 
 Fix $1 \le k \le \ell-1$ and set $v:= \ell-k+1 $. 
Each vertex $ i \in [v+1, \ell] $ is isolated and its weight is the same in $ \left(K_{\ell}^{k}, \varphi_{\ell}^{k}\right) $, $\left(\widehat K_{\ell}^{k},\widehat \varphi_{\ell}^{k}\right)$ and $\left(K_{\ell}^{k+1}, \varphi_{\ell}^{k+1}\right)$ given by $ \varphi_{\ell}^{k}(i) = \widehat \varphi_{\ell}^{k} (i)= \varphi_{\ell}^{k+1}(i) = [-\ell + i+ c_i, d_i ]$.
It suffices to prove the statement firstly after applying  mutation to $ \left(K_{\ell}^{k}[v], \varphi_{\ell}^{  k}|_{[v]}\right) $ w.r.t.~the sink $v $, when we obtain  $\left(\widehat K_{\ell}^{k}[v], \widehat \varphi_{\ell}^{ k }|_{[v]}\right)$, and secondly after applying  mutation to $\left(\widehat K_{\ell}^{k}[v], \widehat \varphi_{\ell}^{ k }|_{[v]}\right)$ w.r.t.~the source $v $, after which we obtain  $\left(K_{\ell}^{k+1}[v], \varphi_{\ell}^{ k+1}|_{[v]}\right)$. 
It is clear that $ (K_{\ell}^{k}[v])' = \widehat K_{\ell}^{k}[v]$ and $ (\widehat K_{\ell}^{k}[v] )'' = K_{\ell}^{k+1}[v] $. 
Moreover, 
\begin{align*}
( \varphi_{\ell}^{ k}|_{[v]})^{\prime} ( v ) &= [1-k +c_v , d_v + k-1 ] =  \widehat \varphi_{\ell}^{ k }|_{[v]}(v) , \\
 ( \widehat \varphi_{\ell}^{ k }|_{[v]})'' (v) &= [1-k +c_v , d_v + k-1 ] =  \varphi_{\ell}^{ k+1}|_{[v]} (v),
\end{align*}
and for each $i \in [\ell-k]$
\begin{align*}
( \varphi_{\ell}^{ k}|_{[v]})^{\prime} ( i )  &= [-k +c_i , d_i + k-1 ] =  \widehat \varphi_{\ell}^{ k }|_{[v]}(i) , \\
 ( \widehat \varphi_{\ell}^{ k }|_{[v]})'' (i) &= [-k +c_i , d_i + k ] =  \varphi_{\ell}^{ k+1}|_{[v]} (i).
\end{align*}
Thus $( \varphi_{\ell}^{ k}|_{[v]})^{\prime} =   \widehat\varphi_{\ell}^{ k}|_{[v]}$ and $( \widehat\varphi_{\ell}^{ k}|_{[v]} )'' =  \varphi_{\ell}^{ k+1}|_{[v]}$, as desired.  
 \end{proof}
 
  Now we define the protagonist of this subsection, a special class of Catalan descendants.
 
 \begin{definition}
\label{def:Cat-ally-ext}
Let  $m\ge 0$,   $c\ge 1$, $1 \le k \le \ell$, and $0 \le p \le \ell$ be integers. 
Let  
$$\Rc^{p,k}_\ell(c,m) :=  \Ac(K_{\ell}^{k}, \psi_{\ell}^{p,k}) \quad \text{for}\quad 1 \le k \le \ell$$ 
and $$\widehat\Rc^{p,k}_\ell(c,m):=  \Ac\left(\widehat K_{\ell}^{k}, \widehat \psi_{\ell}^{p,k}\right) \quad \text{for}\quad 1 \le k \le \ell-1$$ be the Catalan descendants defined as follows: for each $0 \le p \le \ell$ the arrangement $\Rc^{p,1}_\ell(c,m)$ consists of the hyperplanes
\begin{align*}
x_{i}- x_{j} &= [-1 , 1]\quad  (1 \leq i < j \leq \ell), \\
x_i&=  [-c , m] \quad (1 \leq   i \leq p), \\
x_i&=[-c , m+1]  \quad (p<  i \leq \ell),
\end{align*}
and $\Rc^{p,1}_\ell(c,m)$ is the origin of the Catalan descendant sequence
$$ \Rc^{p,1}_\ell (c,m) \stackrel{ \tau_{\ell} }{\longrightarrow} \widehat\Rc^{p,1}_\ell (c,m) \stackrel{ \eta_{\ell} }{\longrightarrow} \Rc^{p,2}_\ell (c,m) \stackrel{ \tau_{\ell -1} }{\longrightarrow} \cdots  \stackrel{ \tau_{2} }{\longrightarrow} \widehat\Rc^{p,\ell-1}_\ell (c,m) \stackrel{ \eta_{2} }{\longrightarrow}  \Rc^{p,\ell}_\ell(c,m).$$
\end{definition}

 \begin{proposition}
 \label{prop:Cat-pth-row}
Fix $0 \le p \le \ell$.
For fixed $1 \le k \le \ell$ (resp., $1 \le k \le \ell-1$) each vertex $n \in [\ell-k+2, \ell] $ is isolated and simplicial in  $ \left(K_{\ell}^{k}[n], \psi_{\ell}^{p,k}|_{[n]}\right)$ (resp., $\left(\widehat K_{\ell}^{k}[n], \widehat \psi_{\ell}^{p,k}|_{[n]}\right)$).
 \end{proposition} 

 \begin{proof}
First  we show the  assertion for $\Rc^{p,k}_\ell(c,m)$. 
By Proposition \ref{prop:Cat-ally-formula}, if $p \le \ell  -k$, then 
$$
 \psi_{\ell}^{p,k}(i)  =\begin{cases}
[1 - c -k, m+k-1 ] & (1 \leq   i \leq p), \\
[1 - c -k, m+k ] & (p<  i \leq  \ell + 1 -k), \\
[ - c - \ell+i , m + \ell +1 - i] & (\ell + 1 - k <  i \leq  \ell ).
\end{cases}
$$
 If $p > \ell -k$, then
$$
 \psi_{\ell}^{p,k}(i) =\begin{cases}
[1 - c -k, m+k-1 ]  & (1 \leq   i \leq  \ell + 1 -k), \\
[ - c - \ell+i , m + \ell - i]   & ( \ell + 1 -k<  i \leq p), \\
[ - c - \ell+i , m + \ell +1 - i]  & (p<  i \leq  \ell ).
\end{cases}
$$
In either case for each $i \in [\ell + 1 -k]$, we have
$$ \psi_{\ell}^{p,k}(i) \supsetneq  \psi_{\ell}^{p,k}(\ell + 2 -k) \supsetneq \cdots \supsetneq  \psi_{\ell}^{p,k}(\ell).$$
Now we apply Proposition \ref{prop:isolated}.

The  assertion for $\widehat\Rc^{p,k}_\ell(c,m)$ is proved in a similar fashion. 
If $p \le \ell  -k$, then 
$$
 \widehat \psi_{\ell}^{p,k}(i)  =\begin{cases}
[- c -k, m+k-1 ] & (1 \leq   i \leq p), \\
[- c -k, m+k ] & (p<  i \leq  \ell  -k), \\
[ - c - \ell+i , m + \ell +1 - i] & (\ell - k <  i \leq  \ell ).
\end{cases}
$$
 If $p > \ell -k$, then
$$
 \widehat\psi_{\ell}^{p,k}(i) =\begin{cases}
[- c -k, m+k-1 ]  & (1 \leq   i \leq  \ell  -k), \\
[ - c - \ell+i , m + \ell - i]   & ( \ell  -k<  i \leq p), \\
[ - c - \ell+i , m + \ell +1 - i]  & (p<  i \leq  \ell ).
\end{cases}
$$
This completes the proof of the proposition.
\end{proof}

 \begin{remark}
	\label{rem:0-ell} 
 It is easily seen from Proposition \ref{prop:Cat-pth-row} that $ \Rc^{\ell,k}_\ell(c, m+1) =\Rc^{0,k}_\ell(c, m)$   for any $1 \le k \le \ell$, and $\widehat\Rc^{\ell,k}_\ell(c, m+1) = \widehat\Rc^{0,k}_\ell(c, m)$ for any $1 \le k \le \ell-1$.
\end{remark}

Analogous to the case of Shi descendants (Remark \ref{rem:CKZ}), the Catalan descendants $\Rc^{p,k}_\ell(c,m)$ and $\widehat\Rc^{p,k}_\ell(c,m)$ in the same  sequence from Definition \ref{def:Cat-ally-ext} have free cones with the same multiset of exponents.
Again this fact can be deduced from  \cite[Thms.~ 3.1 and 4.1]{AbeTranTsujie21_ShiIsh} by noting that the maps $\tau_{\ell},\eta_\ell, \ldots, \tau_2,\eta_2$ preserve both freeness and characteristic polynomials. 
	Thus the freeness and the exponents of the Catalan descendants follow from the supersolvability of the ends of the corresponding sequences (since they are nested $N$-Ish arrangements, see Proposition \ref{prop:Cat-finisher} below). 
	
	Our main result in this subsection asserts that the arrangements $\Rc^{p,k}_\ell(c,m)$ and $\widehat\Rc^{p,k}_\ell(c,m)$ have  ind-flag-accurate cones, which is given in Theorem \ref{thm:Cat-ally-FA} below. 
	The key point of the proof is the ind-flag-accuracy of the origin and first descendant of the Catalan descendant sequences. 
	
	In case  of the end $\Rc^{p,\ell}_\ell(c,m)$,
	the result follows from  Lemma \ref{lem:N-Ish-FA}. 
	We state that result separately. 
	
 \begin{proposition}
\label{prop:Cat-finisher}
 Let $0 \le p \le \ell$. 
 The arrangement $\Rc^{p,\ell}_\ell(c,m)$, consisting of the hyperplanes
\begin{align*}
x_{i}-x_{j} &= 0\quad  (1 \leq i < j \leq \ell), \\
x_i&= [ - c - \ell+i , m + \ell - i] \quad (1 \leq   i \leq p), \\
x_i&= [ - c - \ell+i , m + \ell +1 - i]   \quad (p<  i \leq \ell),
\end{align*}
has supersolvable and ind-flag-accurate cone with exponents $$\exp\left(\cc\Rc^{p,\ell}_\ell(c,m)\right) = (1,  \ell -p+1, 2, 3,\ldots, \ell) +  (0,  (m + \ell+c-1)^{\ell} ).$$
\end{proposition}

Next we study the ind-flag-accuracy  of the origins in their sequence in analogy to 
Theorem \ref{thm:A-H}.

\begin{theorem}
 \label{thm:Cat-ext}
Let $a, m \ge 0$, $n, \ell \ge 1$ and $0 \le p \le \ell$. 
Let $\Ec^{p}_\ell(n, a, m)$ be the  arrangement consisting of the following hyperplanes:
\begin{align*}
x_{i} - x_{j} &= [-a, a] \quad  (1 \leq i < j \leq \ell), \\
x_i&=  [-na, m]  \quad (1 \leq   i \leq p), \\
x_i&=  [-na, m+1] \quad (p<  i \leq \ell).
\end{align*}
The cone over $\Ec^{p}_\ell(n, a, m)$
is  ind-flag-accurate with exponents 
$$\exp\left(\cc\Ec^{p}_\ell(n, a, m)\right) = (1,  \ell -p+1, 2, 3,\ldots, \ell) +  (0,  (m + a (\ell+n-1))^{\ell} ).$$

As a consequence,  the  origin $\Rc^{p,1}_\ell(c, m) = \Ec^{p}_\ell(c,1, m)$ of the Catalan descendant sequence from Definition \ref{def:Cat-ally-ext} has  ind-flag-accurate cone with exponents 
$$\exp\left(\cc\Rc^{p,1}_\ell(c, m)\right) = (1,  \ell -p+1, 2, 3,\ldots, \ell) +  (0,  (m + \ell+c-1)^{\ell} ).$$ 
 \end{theorem} 

 \begin{proof}
To show inductive freeness of $\Ac :=\cc\Ec^p_\ell(n, a, m) $, we use a similar argument  to the one from the proof of Theorem \ref{thm:A-H}. 
Define the lexicographic order on 
$$\{ (\ell, \ell-p) \mid 0\le \ell- p \le \ell, \, \ell \ge 1\}.$$
First we prove $\Ac \in \mathcal{IF}$ with the desired exponents by induction on $ (\ell, \ell-p) $. 
When $\ell=1$, it is obvious. Suppose $\ell\ge2$.

Case $1$. First consider $0 \le p \le \ell-1$. 
Let  $H \in \Ac$ denote the hyperplane $x_{p+1} = (m+1)z$. 
Then $\Ac'=\Ac\setminus \{H\} = \cc\Ec^{p+1}_\ell(n, a, m)\in \mathcal{IF}$ with exponents  
$$\exp\left(\cc\Ec^{p+1}_\ell(n, a, m)\right) = (1,  \ell -p, 2, 3,\ldots, \ell) +  (0,  (m + a (\ell+n-1))^{\ell} ),$$
 by the induction hypothesis since, $(\ell, \ell-p) > (\ell, \ell-p-1)$. 
Moreover, $\Ac''=\Ac^{H}$ consists of the following hyperplanes 
\begin{align*}
z &= 0 , \\
x_{i} - x_{j} &=  [-a , a]z\quad  (1 \leq i < j \leq \ell, i\ne p+1,j\ne p+1 ), \\
x_i&=  [-na , m+a+1 ]z \quad (1 \leq   i \leq \ell).
\end{align*}
Thus $\Ac'' =    \cc\Ec^{0}_{\ell-1}(n, a, m+a)$. 
Hence   $\Ac'' \in \mathcal{IF}$ with exponents 
$$\exp\left(\cc\Ec^{0}_{\ell-1}(n, a, m+a)\right) = (1, 2, 3,\ldots, \ell) +  (0,  (m + a (\ell+n-1))^{\ell-1} ),$$ again  by the induction hypothesis, since $(\ell, \ell-p) > (\ell-1,\ell-1)$. 
Therefore, by Theorem \ref{thm:AD}, $\Ac =\cc\Ec^p_\ell(n, a, m)\in \mathcal{IF}$ with the desired exponents. 

Case $2$. Now consider  $\ell=p$. 
We may assume further that $m=0$  since $\Ec^{\ell}_\ell(n, a, m) = \Ec^{0}_\ell(n, a, m-1)$.
The arrangement  in question is $$\Bc(n) := \cc\Ec^\ell_\ell(n, a, 0),$$ consisting of 
\begin{align*}
z &= 0 , \\
x_{i} - x_{j} &=  [-a , a]z\quad  (1 \leq i < j \leq \ell), \\
x_i&=  [-na , 0]z \quad (1 \leq   i \leq \ell).
\end{align*}
Let $\Csc$ be the cone over the type $A$ extended Catalan arrangement, consisting of the hyperplanes 
\begin{align*}
z &= 0 , \\
x_{i} - x_{j} &=  [-a , a]z\quad  (1 \leq i < j \leq \ell).
\end{align*}
Note that $\Csc$
is inductively free with exponents $ (1, 0, a\ell  +i )_{i=1}^{\ell-1}$, by our induction hypothesis, since it is  linearly equivalent to $\cc\Ec^{\ell-1}_{\ell-1}(1, a, a) \times \varnothing_1$ (via $x_1-x_{i}\mapsto x_{i-1}$).  

We are going to show that $\Bc(n) \in \mathcal{IF}$  with exponents $ (1, a (\ell+n-1)  +i )_{i=1}^{\ell}$, by induction on $n$. 
If $n=1$, by Theorem \ref{thm:AD}, adding $r$ of the hyperplanes $x_i = sz$ for $1 \leq   i \leq \ell$, $s \in [-a,0]$ in any order to $\Csc$ 
produces an inductively free arrangement with exponents $ (1, r, a\ell  +i )_{i=1}^{\ell-1}$.
Indeed, the restriction to the last hyperplane added at each step is linearly equivalent to $\cc\Ec^{\ell-1}_{\ell-1}(1,a,a)$. 
The case $r =a\ell  +\ell $ gives the desired result for $\Bc(1)$. 

Now we show that for $n\ge 1$, if $\Bc(n)$ is inductively free, then so is $\Bc(n+1)$ with the desired exponents. 
It requires a more delicate order of addition of the hyperplanes. 
Let $0 \le k \le a$ and denote by $\Bc_k(n) $ the arrangement having the following hyperplanes 
\begin{align*}
z &= 0 , \\
x_{i} - x_{j} &=  [-a , a]z\quad  (1 \leq i < j \leq \ell), \\
x_i&=  [-(na+k) , 0]z \quad (1 \leq   i \leq \ell).
\end{align*}
In particular, we have $\Bc_0(n) =  \Bc(n)$, and $\Bc_a(n) =  \Bc(n+1)$. 
Note that for any $ \Bc_0(n) \subsetneq  \Dc \subseteq  \Bc_a(n)  $ and for any hyperplane $H \in \Dc$ of the form $x_i = sz$ for some $-(na+a) \le s \le -(na+1)$, the restriction $\Dc^H$ is linearly equivalent to $\cc\Ec^{\ell-1}_{\ell-1}(1,a,-s)$. 
Hence $\Dc^H \in \mathcal{IF}$ with exponents $ (1, -(s+a) + a\ell  +i )_{i=1}^{\ell-1}$ by the induction hypothesis on $ (\ell, \ell-p) $.  
Now by Theorem \ref{thm:AD}, adding $\ell$ of the hyperplanes in $\Bc_1(n) \setminus \Bc_0(n) $ in any order to $ \Bc_0(n) $ shows that $\Bc_1(n)$ is inductively free with exponents $ (1, a (\ell+n-1)  +i +1)_{i=1}^{\ell}$.
Now,  
for fixed $0 \le k \le a-1$, adding the hyperplanes in $\Bc_{k+1}(n) \setminus \Bc_k(n) $ in any order to $ \Bc_k(n) $, implies that $\Bc_{k+1}(n)$ is inductively free with exponents $ (1, a (\ell+n-1)  + i +k+1)_{i=1}^{\ell}$. 
The case $k =a-1$ gives the desired result for $\Bc(n+1)$. 

This completes the proof of the inductive freeness of $\Ac =\cc\Ec^p_\ell(n, a, m)$. 

Finally, we derive that $\Ac$ is  ind-flag-accurate. 
We argue by induction on $\ell$. 
When $\ell=1$, the statement is obvious. Suppose $\ell\ge2$.
It suffices to consider $1\le p \le \ell$, since $\Ec^{\ell}_\ell(n, a, m+1) = \Ec^{0}_\ell(n, a, m)$.
Let  $K \in \Ac$ denote the hyperplane $x_{1} = -naz$. 
Then the restriction $\Ac^{K}$ consists of the following hyperplanes 
\begin{align*}
z &= 0 , \\
x_{i} - x_{j} &=  [-a , a]z\quad  (2 \leq i < j \leq \ell  ), \\
x_i&=  [-(n+1) a, m ]z \quad (2 \leq   i \leq p   ), \\
x_i&=[-(n+1) a, m+1] z \quad (p+1 \le  i \leq \ell).
\end{align*}
Thus $\Ac^K =   \cc\Ec^{p-1}_{\ell-1}(n+1, a, m)$. 
Hence  $\Ac^K $ is  ind-flag-accurate by the induction hypothesis with exponents 
$$\exp\left(\cc\Ec^{p-1}_{\ell-1}(n+1, a, m)\right) = (1,  \ell -p+1, 2, 3,\ldots, \ell-1) +  (0,  (m + a (\ell+n-1))^{\ell-1} ).$$ 
Finally, thanks to Lemma \ref{lem:flag-accuracy}, $\Ac$  is ind-flag-accurate. 
This completes the proof of the theorem.
\end{proof}

	 \begin{remark}
	\label{rem:p=0,1} 
	We remark that in the cases  $p = 0$, or $p = 1$, one may consider a different restriction to deduce the  ind-flag-accuracy of $\cc\Ec^p_\ell(n, a, m)$. 
Indeed, by Case $1$ in the proof of Theorem \ref{thm:Cat-ext} the restriction of $\cc\Ec^p_\ell(n, a, m)$ to the hyperplane given by $x_{p+1} = (m+1)z$ is identical to $\cc\Ec^{0}_{\ell-1}(n, a, m+a)$ whose multiset of exponents is $(1, m + a (\ell+n-1)  +i +1)_{i=1}^{\ell-1}$. 
\end{remark}

We give some comments on the cone  $\cc \Cat^m (A_\ell)$ over the extended Catalan arrangement of type $A_\ell$. 
It was shown to be inductively free by  Edelman and  Reiner \cite[Thm.~3.2]{ER96_FreeArrRhombicTilings}. 
Recently, Tsujie and Nakashima \cite[Cor.~4.2]{NT22} proved that $\cc \Cat^m (A_\ell)$ is even \emph{hereditarily inductively free} i.e., every restriction of $\cc \Cat^m (A_\ell)$ is inductively free. 
Our result below confirms the ind-flag-accuracy of $\cc \Cat^m (A_\ell)$, thanks to Theorem \ref{thm:Cat-ext}.

\begin{corollary}
	\label{cor:CatA-FA}
The extended Catalan arrangement $ \Cat^m (A_\ell)$ of type $A_\ell$ (Definition \ref{Def_IdealShiCatalan}) affinely equivalent to $\Ec^{\ell}_\ell(1, m, m)$ or $\Ec^{0}_\ell(1, m, m-1)$ with $m\ge1$ has  ind-flag-accurate cone with exponents $ (1,  m(\ell+1)  +i )_{i=1}^\ell$.  

Moreover,  the inductive argument in the proof of Theorem \ref{thm:Cat-ext} provides a simple root witness for the  ind-flag-accuracy of  $\cc \Cat^m (A_\ell)$ as follows.
Let $H_1, \ldots, H_\ell$ denote the hyperplanes $x_1 = -mz$ and $x_{i-1} - x_{i} = mz$ for $2 \le i \le \ell$ in $\cc\Ec^{0}_\ell(1, m, m-1)$, or equivalently  $x_i - x_{i+1} = mz$   ($1 \le i \le \ell$)  in $\cc \Cat^m (A_\ell)$. 
Set $X_i : = \bigcap_{j=1}^{i} H_j$ for $1 \le i \le \ell-1$. Then  $X_{\ell-1} \subseteq \cdots \subseteq  X_1$  is a witness for the  ind-flag-accuracy of $\cc \Cat^m (A_\ell)$.
\end{corollary}

Next we investigate the ind-flag-accuracy of the first descendants. 

\begin{theorem} 
\label{thm:nt-starter}
Let $m\ge 0$,  $c, \ell \ge 1$ and $0 \le p \le \ell$. 
The first descendant $\widehat\Rc^{p,1}_\ell(c,m)$ of the Catalan descendant sequence from Definition \ref{def:Cat-ally-ext} has  ind-flag-accurate cone with exponents 
$$\exp\left(\cc\widehat\Rc^{p,1}_\ell(c,m)\right) = (1,  \ell -p+1, 2, 3,\ldots, \ell) +  (0,  (m + \ell+c-1)^{\ell} ).$$ 
\end{theorem}

\begin{proof}
First we derive the inductive freeness of $\Ac :=\cc\widehat\Rc^{p,1}_\ell(c,m)$.  
Note that when $p=\ell$, the arrangement $\Ac = \cc\widehat\Rc^{\ell,1}_\ell(c,m)$ consists of the hyperplanes 
\begin{align*}
z &= 0 , \\
x_{i} - x_{j} &= [-1, 1]z \quad  (1 \leq i < j \leq \ell-1), \\
x_{i} - x_{\ell} &= [-1, 0] z\quad  (1 \leq i  \leq \ell-1), \\
x_i&=  [-c-1, m]z  \quad (1 \leq   i  \leq \ell-1), \\
x_\ell&=  [-c, m] z.
\end{align*}
One may apply Proposition \ref{prop:MC-criterion} to show that the subarrangement $\Bc \subseteq \Ac$ consisting of 
\begin{align*}
z &= 0 , \\
x_{i} - x_{j} &= [-1, 1]z \quad  (1 \leq i < j \leq \ell-1), \\
x_i&=  [-c-1, m]z  \quad (1 \leq   i  \leq \ell-1), \\
\end{align*}
is a modular coatom of $\Ac$. 
In graphical terms, the vertex $\ell$ is a source and a simplicial vertex of the underlying digraph of $\Ac$. 
Moreover, according to Theorem \ref{thm:Cat-ext}, we see that $\Bc =\varnothing_1 \times \cc\Ec^{\ell-1}_{\ell-1}(c+1, 1, m)$ is inductively free with exponents  
$$\exp(\Bc) = (0,1, 1, 2, 3,\ldots, \ell-1) +  (0,0,  ( m + \ell+c-1  )^{\ell-1} ).$$ 
Note also that $| \Ac \setminus \Bc | = m + 2\ell+c-1$.
Now applying Proposition \ref{prop:modular coatom}(ii)
 allows us to deduce that $\Ac  \in \mathcal{IF}$ with the desired exponents. 

Suppose $0 \le p \le \ell-1$. 
Define the lexicographic order on 
$$\{ (\ell, \ell-p) \mid 1\le \ell- p \le \ell, \, \ell \ge 1\}.$$
We show $\Ac =\cc\widehat\Rc^{p,1}_\ell(c,m)  \in \mathcal{IF}$ with the desired exponents by induction on $ (\ell, \ell-p) $. 
When $\ell=1$, it is obvious. Suppose $\ell\ge2$.

Case $1$. First let $0 \le p \le \ell-2$. 
Then $\Ac$ consists of the hyperplanes 
\begin{align*}
z &= 0 , \\
x_{i} - x_{j} &= [-1, 1]z \quad  (1 \leq i < j \leq \ell-1), \\
x_{i} - x_{\ell} &= [-1, 0] z\quad  (1 \leq i  \leq \ell-1), \\
x_i&=  [-c-1, m]z  \quad (1 \leq   i \leq p), \\
x_i&=  [-c-1, m+1] z\quad (p<  i \leq \ell-1), \\
x_\ell&=  [-c, m+1] z.
\end{align*}
Let  $H \in \Ac$ denote the hyperplane $x_{p+1} = (m+1)z$. 
Then $\Ac\setminus \{H\} = \cc\widehat\Rc^{p+1,1}_\ell(c,m) \in \mathcal{IF}$ with exponents  $(1,  \ell -p, 2, 3,\ldots, \ell) +  (0,  ( m + \ell+c-1  )^{\ell} )$, by the induction hypothesis. 
Moreover, $\Ac^{H}=    \cc\widehat\Rc^{0,1}_{\ell-1}(c,m+1)$. 
Hence   $\Ac^{H} \in \mathcal{IF}$ with exponents 
$$\exp\left(\cc\widehat\Rc^{0,1}_{\ell-1}(c,m+1)\right) = (1, 2, 3,\ldots, \ell) +  (0,  ( m + \ell+c-1 )^{\ell-1} ),$$  by the induction hypothesis. 
Therefore, by Theorem \ref{thm:AD}, $\Ac  \in \mathcal{IF}$ with the desired exponents. 

Case $2$. Now consider $p=\ell-1$. Then $\Ac$ is given by replacing the fourth and fifth set of the hyperplanes in Case $1$ by 
$$
x_i=   [-c-1, m]z  \quad (1 \leq   i \leq  \ell-1).
$$
Let  $K \in \Ac$ denote the hyperplane $x_{\ell} = (m+1)z$. 
Then $\Ac\setminus \{K\} = \cc\widehat\Rc^{\ell,1}_\ell(c,m)$ is inductively free with exponents  $(1, 1, 2, 3,\ldots, \ell) +  (0,  ( m + \ell+c-1  )^{\ell} )$, by the discussion at the beginning. 
Moreover, $\Ac^{K}=   \cc\Ec^{\ell-1}_{\ell-1}(c+1, 1, m+1)$. 
By Theorem \ref{thm:Cat-ext},  $\Ac^{K} \in \mathcal{IF}$ with exponents $(1, 2, 3,\ldots, \ell) +  (0,  ( m + \ell+c-1 )^{\ell-1} )$. 
Therefore, by Theorem \ref{thm:AD}, $\Ac  \in \mathcal{IF}$ with the desired exponents.

 This completes the proof of the  inductive freeness of $\Ac = \cc\widehat\Rc^{p,1}_\ell(c,m)$. 

Finally, we show that $\Ac$ is ind-flag-accurate. 
We argue by induction on $\ell$. 
When $\ell=1$, it is obvious. Suppose $\ell\ge2$.
It suffices to consider  the case when $1\le p \le \ell$, since $\widehat\Rc^{\ell,1}_\ell(c, m+1) = \widehat\Rc^{0,1}_\ell(c, m)$ (see Remark \ref{rem:0-ell}). 
The proof is more straightforward than the one above, as we do not have to distinguish between different values of $p$. 
Let  $F \in \Ac$ denote the hyperplane $x_{1} = -(c+1)z$. 
Then $\Ac^F =   \cc\widehat\Rc^{p-1}_{\ell-1}(c+1, m)$. 
Hence  $\Ac^F $ is ind-flag-accurate with exponents 
$$\exp\left(\cc\widehat\Rc^{p-1}_{\ell-1}(c+1, m)\right) = (1,  \ell -p+1, 2, 3,\ldots, \ell-1) +  (0,  (m + \ell+c-1)^{\ell-1} ),$$ 
by the induction hypothesis. 
Owing to Lemma \ref{lem:flag-accuracy}, $\Ac$  is ind-flag-accurate. 
This completes the proof of the theorem.
 \end{proof}

We are now in a position to derive the main result of this subsection. 
We remark that unlike in the case of Shi genealogy, the modular coatoms from simplicial isolated vertices of the digraph do not automatically imply the ind-flag-accuracy of the Catalan descendants. 
It is crucial to explicitly construct a witness. 

\begin{theorem}
 \label{thm:Cat-ally-FA}
Let  $m\ge 0$,   $c\ge 1$, $1 \le k \le \ell$,  and $0 \le p \le \ell$. 
The cones over the   Catalan descendants $\Rc^{p,k}_\ell(c,m)$ and $\widehat\Rc^{p,k}_\ell(c,m)$ are   ind-flag-accurate with exponents 
$$\exp\left(\cc\Rc^{p,k}_\ell(c,m)\right) = 
\exp\left(\cc\widehat\Rc^{p,k}_\ell(c,m)\right) = 
(1,  \ell -p+1, 2, 3,\ldots, \ell) +  (0,  (m + \ell+c-1)^{\ell} ).$$ 
 \end{theorem}

 \begin{proof}
By Theorems  \ref{thm:Cat-ext},  \ref{thm:nt-starter}, Remark \ref{rem:0-ell} and Proposition \ref{prop:Cat-finisher}, it suffices to prove for any fixed $1 \le p \le \ell$ and $2 \le k \le \ell-1$ that the cones  
$\cc\Rc^{p,k}_\ell(c,m)$, $\cc\widehat\Rc^{p,k}_\ell(c,m)$  are ind-flag-accurate with exponents $ (1,  \ell -p+1, 2, 3,\ldots, \ell) +  (0,  (m + \ell+c-1)^{\ell} )$. 
There are two cases to consider, $p \le \ell  -k$ versus $p > \ell  -k$. 
Since the proofs use a similar method, we give an argument only for the former (which is the more difficult one). 

First  we show the  assertion for $\Ac : = \cc\Rc^{p,k}_\ell(c,m)$. 
Note that by Proposition \ref{prop:Cat-pth-row} the arrangement $\Bc : = \textbf{c} \Ac\left(K_{\ell}^{k}[\ell-k+1], \psi_{\ell}^{p, k}|_{[\ell-k+1]}\right) $ is given by 
\begin{align*}
z &= 0 , \\
x_{i} - x_{j} &= [-1,1]z \quad  (1 \leq i < j \leq\ell + 1 -k), \\
x_i&=  [1-c-k , m+k-1]z \quad (1 \leq   i \leq p), \\
x_i&= [1-c-k , m+k]z   \quad (p<  i \leq \ell + 1 -k).
\end{align*}
Hence 
$$\Bc  = \textbf{c}\Rc^{p,1}_{\ell + 1 -k}(c+k-1,m+k-1),$$
is  ind-flag-accurate with $$\exp(\Bc) = (1,  \ell -p+1, k+1, k+2,\ldots, \ell) +  (0,  (m + \ell+c-1)^{\ell-k+1} ),$$  by Theorem \ref{thm:Cat-ext}. 

Owing to  Proposition \ref{prop:Cat-pth-row}, each isolated vertex $n \in [\ell-k+2, \ell] $ is  simplicial in  $ \left(K_{\ell}^{k}[n], \psi_{\ell}^{p,k}|_{[n]}\right)$ with $ \psi_{\ell}^{p,k}(n) = [ - c - \ell+n , m + \ell +1 - n] $. 
Applying Corollary \ref{cor:simplicial}(i) repeatedly to the $k-1$ simplicial vertices $\ell ,\ell-1, \ldots, \ell-k+2 $ in this order,  we are able to deduce that  
$$ \Ac \in \mathcal{IF} \quad  \Longleftrightarrow  \quad \Bc    \in \mathcal{IF} .$$ 
Thus the inductive freeness of $ \Ac$ is clear. 
Moreover, $$\exp( \Ac ) = (1,  \ell -p+1, 2, 3,\ldots, \ell) +  (0,  (m + \ell+c-1)^{\ell} ),$$
since the isolated vertices $n \in [\ell-k+2, \ell] $ contribute  the exponents 
$$| \psi_{\ell}^{p,k}(n) |+n-1 =(m + \ell+c-1) + (\ell-n+2)$$ with $2 \le \ell-n+2 \le k$ 
to $\exp( \Ac) $. 

Now we demonstrate the ind-flag-accuracy of $ \Ac$. 
Note that unlike in the proof of Theorem \ref{thm:Shi-ally-FA} for Shi descendants, we cannot appeal to Corollary \ref{cor:simplicial}(ii) here, since the exponent $(m + \ell+c-1) +e$  in $\exp(\Bc)$ for some $k+1 \le e \le \ell$ exceeds $| \psi_{\ell}^{p,k}(n) |+n-1$ for each $n \in [\ell-k+2, \ell] $. 
We may overcome this difficulty by making use of a certain witness for the ind-flag-accuracy of $\Bc$ hinted at in 
Theorem \ref{thm:Cat-ext}. 

Define $\ell-k$ hyperplanes $H_1, \ldots, H_{\ell-k} \in \Ac$ as follows:
$$H_1: x_1 = (1-c-k)z, \quad H_j: x_{j-1} - x_{j} = z \quad (2 \le j \le \ell-k).$$
For each $1 \le s \le \ell-k$, set 
\begin{align*}
X_s & : = \bigcap_{j=1}^{s} H_j, \text{ and } \\
M_s & : = \{ (1-c-k)-(j-1) \mid 1 \le j \le s\} = [2-c-k-s, 1-c-k].
\end{align*} 
Then $X_{\ell-k} \subseteq \cdots \subseteq  X_2  \subseteq X_1 \subseteq V'=\RR^{\ell+1} $ and $\dim_{V'}(X_s) =\ell+1-s$. 

We show that this flag is part of a witness for the ind-flag-accuracy of $\Ac$.
First we show that each restriction $\Ac^{X_s}$ is inductively free and compute its exponents. 

If $1 \le s \le p$ then by the proof of the ind-flag-accuracy in Theorem \ref{thm:Cat-ext}, $\Ac^{X_s}$ consists of the hyperplanes
\begin{align*}
z &= 0 , \\
x_{i} - x_{j} &= [-1,1]z \quad  (s+1 \leq i < j \leq\ell + 1 -k), \\
x_i&=  [1-c-k-s , m+k-1]z \quad (s+1 \leq   i \leq p), \\
x_i&= [1-c-k-s , m+k]z   \quad (p<  i \leq \ell + 1 -k), \\
x_i&= (M_s \sqcup  [ - c - \ell+i , m + \ell +1 - i])z   \quad (\ell + 1 -k <  i \leq \ell  ).
\end{align*}
Upon applying Corollary \ref{cor:simplicial}(i) repeatedly to the $k-1$ simplicial isolated vertices $\ell ,\ell-1, \ldots, \ell-k+2 $ in this order, we delineate that  
$$\Ac^{X_s} \in \mathcal{IF} \quad  \Longleftrightarrow  \quad \textbf{c}\Rc^{p-s,1}_{\ell + 1 -k-s}(c+k+s-1,m+k-1)    \in \mathcal{IF} .$$ 
By Theorem \ref{thm:Cat-ext}, $\Ac^{X_s}$ is  indeed inductively free with 
$$\exp\left(\Ac^{X_s}\right) = (1,  \ell -p+1, 2,\ldots, k, k+1, \ldots, \ell-s) +  (0,  (m + \ell+c-1)^{\ell-k+1-s} ).$$
Note that the isolated vertices from $[\ell-k+2, \ell] $ continue to contribute the exponents $(m + \ell+c-1) + (\ell-i+2)$ with $2 \le \ell-i+2 \le k$ to $\exp\left(\Ac^{X_s}\right)$.

If $p< s  \le \ell-k$ then $\Ac^{X_s}$ is given by replacing the third and fourth set of the hyperplanes in the previous case by 
$$
x_i=  [1-c-k-s , m+k]z \quad (s+1 \leq   i \leq \ell + 1 -k).
$$
Therefore, in this case
$$\Ac^{X_s} \in \mathcal{IF} \quad \Longleftrightarrow  \quad \textbf{c}\Rc^{\ell + 1 -k-s,1}_{\ell + 1 -k-s}(c+k+s-1,m+k)    \in \mathcal{IF} .$$ 
Again, by Theorem \ref{thm:Cat-ext}, $\Ac^{X_s}$ is inductively free with 
$$\exp\left(\Ac^{X_s}\right) = (1, 2,\ldots, k, k+1, \ldots, \ell-s+1) +  (0,  (m + \ell+c-1)^{\ell-k+1-s} ).$$

A simple comparison of the exponents shows that the flag $X_{\ell-k} \subseteq \cdots \subseteq  X_2  \subseteq X_1 \subseteq V'=\RR^{\ell+1} $ is indeed part of a witness for the ind-flag-accuracy of $\Ac$. 
Moreover, $\Ac^{X_{\ell-k}}$ is a strictly nested $N$-Ish arrangement hence is ind-flag-accurate, by Lemma \ref{lem:N-Ish-FA}. 

By Lemma \ref{lem:flag-accuracy}, we conclude that $\Ac$  is ind-flag-accurate. 
Note that the flag above can be easily extended to a proper witness for $\Ac$ thanks to the construction of a witness in Lemma \ref{lem:N-Ish-FA}. 
This completes the proof of the ind-flag-accuracy of $\Ac$.

The proof of the  assertion for $ \cc\widehat\Rc^{p,k}_\ell(c,m)$ runs essentially along the same lines as the argument above with the use of Theorem \ref{thm:nt-starter} in place of Theorem \ref{thm:Cat-ext}.
\end{proof}

%%%%%%%%%%%%%%%%%%%%%%%%%%%%%%%%%%%%%%%%%%%%%%%%%%%%%%%%%%%%%%%%%%%%%%%%%%%%%%%%%%%%%%%

\bigskip
{\bf Acknowledgements}:
Work on this paper began during a 
visit to the Mathematisches Forschungsinstitut Oberwolfach while one of us (PM) held a Leibniz Fellowship; we thank them for their support. 
The third author is supported by a postdoctoral fellowship of the Alexander von Humboldt Foundation. 
The authors would like to thank Shuhei Tsujie for stimulating conversations concerning graphic arrangements.

\newcommand{\etalchar}[1]{$^{#1}$}
\providecommand{\bysame}{\leavevmode\hbox to3em{\hrulefill}\thinspace}
\providecommand{\MR}{\relax\ifhmode\unskip\space\fi MR }
\providecommand{\MRhref}[2]{%
	\href{http://www.ams.org/mathscinet-getitem?mr=#1}{#2}
}
\providecommand{\href}[2]{#2}

\end{document}